\documentclass[english,11pt]{article}
\usepackage[german,french,english]{babel}
\usepackage[cp850]{inputenc}
\usepackage{latexsym,graphicx, fancybox}
\usepackage{graphicx}
\usepackage{booktabs}

\usepackage{amssymb, amsmath, amsfonts}
\usepackage{amsmath, amssymb}
\usepackage{color}
\usepackage{tikz}
\usepackage{latexsym}

\usepackage[colorlinks=true, allcolors=blue]{hyperref}    
\usepackage{tocloft}        
\usepackage{comment}
\usepackage{mathrsfs}

\usepackage{float}   
\usepackage{graphicx}  
\usepackage{hyperref}  
\usepackage{latexsym}  
\usepackage{xcolor}   
\usepackage{lipsum}   
\usepackage{subcaption}

\usepackage{amssymb}  
\usepackage{amsmath}  
\usepackage{amsbsy}
\usepackage{amsthm}
\usepackage{bbm}    
\usepackage{eucal}   
\usepackage{mathrsfs}  

\usepackage{multirow}  
\usepackage{ulem}    
\usepackage{dsfont}

\usepackage[round]{natbib}
\usepackage[textsize=small]{todonotes}

\usepackage{etoolbox}
\apptocmd{\thebibliography}{\setlength{\itemsep}{3.246pt}}{}{}

\usepackage{microtype}
	\usepackage{amsthm}
	
	\newcommand{\dd}{\,\mathrm{d}}
	\newcommand{\R}{\mathbb{R}}
	\newcommand{\N}{\mathbb{N}}
	\newcommand{\E}{\mathbb{E}}
	\renewcommand{\P}{\mathbb{P}}

	\newcommand{\p}{\mathbb{P}}
	\newcommand{\F}{\mathbb{F}}
	
	\newcommand{\cF}{{\mathcal F}}
	\newcommand{\cA}{{\mathcal A}}

	\newcommand{\id}{\mathds{1}}
	\newcommand{\as}{\mbox{{\rm a.s.}}}
	
	\newcommand{\Q}{\mathbb{Q}}

	\usepackage{authblk}
	
	\newcommand{\T}{\top}
	\renewcommand{\c}{\alpha}
	\newcommand{\Mid}{{\ \Big|\ }}

	\DeclareMathOperator{\diag}{diag}

	\newtheorem{Theorem}{Theorem}[section]
	\newtheorem{Definition}[Theorem]{Definition}
	\newtheorem{Proposition}[Theorem]{Proposition}

	\newtheorem{Lemma}[Theorem]{Lemma}
	
	\newtheorem{Remark}[Theorem]{Remark}
	\newtheorem{Example}[Theorem]{Example}
	
	\newtheorem{assumption}{Assumption}[section]
	\newtheorem*{note}{Note}
	
	\usepackage{thm-restate}
	\usepackage{float}
	\usepackage{hyperref}
	
\begin{document}

\selectlanguage{english}


\title{\bf 
	On Utility Maximization under Multivariate Fake Stationary Affine Volterra Models 
}

\author{
	Emmanuel Gnabeyeu\footnote{Laboratoire de Probabilités, Statistique et Modélisation, UMR 8001, Sorbonne Université and Universit\'e Paris Cit\'e, 4 pl. Jussieu, F-75252 Paris Cedex 5, France. E-mail: {\tt emmanuel.gnabeyeu\_mbiada@sorbonne-universite.fr}}\footnote{This research benefited from the support of "Ecole Doctorale Sciences Mathematiques de Paris Centre".}
} 
 \date{March 12, 2026} 
\maketitle
\vspace{-1cm} 
\renewcommand{\abstractname}{Abstract}
\begin{abstract}
	 This paper is concerned with Merton's portfolio optimization problem in a Volterra stochastic environment described by a multivariate fake stationary Volterra--Heston model.
	Due to the non-Markovianity and non-semimartingality of the underlying processes, the classical stochastic control
	approach cannot be directly applied in this setting. Instead, the problem is tackled using a stochastic factor solution to a Riccati backward stochastic differential equation (BSDE). 
	Our approach is inspired by the martingale optimality principle combined with a suitable verification argument.	
	The resulting optimal strategies for Merton's problems are derived in semi-closed form depending on the solutions to time-dependent multivariate Riccati-Volterra equations, while the optimal value is expressed using the solution to this original Riccati BSDE. 
	Numerical results on a two dimensional fake stationary rough Heston model illustrate the impact
	of stationary rough volatilities on the optimal Merton strategies.
\end{abstract}



\setcounter{tocdepth}{2}
\vspace{-.5cm}
\section{Introduction}
\vspace{-.2cm}
\noindent The modeling of asset price dynamics has undergone a paradigm shift with the empirical observation that both implied and realized volatilities of major financial indices exhibit significantly rougher sample paths~\citep{GatheralJR2018} than those generated by classical Brownian-motion-based models. This observation has sparked a rapidly expanding body of research on rough volatility models. Motivated by the widespread practical success of the celebrated~\cite{Heston1993} stochastic volatility framework, several rough extensions have been proposed, most notably the rough Heston model~\citep{el2019characteristic} builds on market microstructure and its generalisation, the Volterra Heston model in~\cite{abi2019affine}. Recent remarkable advances include the introduction of the so-called \textit{fake stationary Volterra Heston models} in~\cite{EGnabeyeuPR2025, EGnabeyeuR2025} with the aim of providing a unified and consistent framework that captures both short- and long-maturity behaviors and enables robust fitting across the entire term structure. This broader class of volatility models, encompassing the aforementioned specifications, is obtained by modeling the volatility process as a stochastic Volterra equation of convolution type with a time-dependent diffusion coefficient. Therefore, this paper focuses on the financial market within the fake stationary affine Volterra model.



\medskip
\noindent Substantial progress has recently been achieved in the study of option pricing problems and asymptotic analysis under rough volatility dynamics. By contrast, portfolio optimization in such models remains comparatively underexplored, although it has attracted growing interest in recent years. 
Notable contributions include \cite{fouque2018aoptimal,BaeuerleDesmettre2020,HanWong2020b}, which investigate optimal investment problems with power utility in fractional Heston-type models. 
Despite these advances, the overwhelming majority of developments in rough volatility, whether for asset modeling, derivative pricing, or portfolio selection, have been confined to the mono-asset case. From a practical standpoint, however, multi-asset allocation with correlated sources of risk constitutes a fundamental dimension of portfolio management; see, for instance, \cite{BuraschiPorchiaTrojani2010}. 

\medskip
\noindent Merton's portfolio optimization problem which consists in maximizing an investor's expected utility from terminal wealth with respect to a given utility function has served as a cornerstone in mathematical finance. It is the most classic financial economic approach to understand how the market volatility affect investment demands.
In the classical~\cite{Heston1993} stochastic volatility framework, this problem was solved explicitly in~\cite{Kraft2005}, building on the representation results of~\cite{zariphopoulou2001solution}, and extensions to general affine stochastic volatility models were obtained in~\cite{KallsenMuhleKarbe2010} using stochastic control theory.
In the Volterra framework, however, the volatility process is non-Markovian, which prevents the direct application of the classical stochastic control methods based on the Hamilton--Jacobi--Bellman (HJB) partial differential equation. 
In order to circumvent this difficulty, \cite{HanWong2020b}, inspired by~\cite{fouque2018aoptimal}, adopt a martingale distortion ansatz and apply the martingale optimality principle to derive explicit optimal investment strategies in a single-asset or mono-asset Volterra--Heston model.

\medskip
\noindent Motivated by several important empirical stylized facts about real financial markets such as choice among multiple assets, rough volatility behavior, correlations across stocks or assets and leverage effects (i.e., correlation between
a stock and its volatility), multivariate rough volatility models have recently been developed ; see, e.g.,~\cite{abi2019affine,TomasRosenbaum2021}. 
 In~\cite{AichingerDesmettre2021}, the authors analyze the Merton portfolio problems with power utility for a class of multivariate Volterra Heston models that features 
correlation between a
stock and its volatility.
\noindent In the present paper, we solve the Merton portfolio problem for investors with \textit{power, exponential} and  \textit{logarithmic utility functions}, within the class of the so-called \textit{fake stationary multivariate affine Volterra models}.

\medskip
\noindent {\bf Main contributions.}
Building upon recent developments in 
Volterra models ~\citep{ 
 EGnabeyeuPR2025,EGnabeyeuR2025} and motivated by recent works and advances on multivariate Volterra volatility modeling \cite{abi2019affine, TomasRosenbaum2021,AichingerDesmettre2021}, the primary objective of this paper is to advance the literature on utility maximization along two main directions:
\begin{itemize}
	\item[(i)] We introduce a class of \textit{ multivariate fake stationary affine Volterra stochastic volatility models} that capture key stylized features of financial markets, including heterogeneous roughness across assets, 
	and leverage effects namely, dependence between asset returns and their respective volatilities while maintaining a \textit{consistent modeling framework across time scales}, from short to long maturities.
	\item[(ii)] This model preserve analytical tractability, thereby enabling explicit characterization of the optimal investment strategy for the Merton's problem, despite the intrinsic challenges posed by multivariate non-Markovian dynamics. The optimal value is characterized by the solution of an associated Riccati BSDE (Riccati BSDE) whose generator is quadratic in the Brownian stochastic integrand (denoted by \(\Lambda\)). 
	
	 \item[(iii)] It further extends the models of~\cite{abi2019affine,TomasRosenbaum2021} to an inhomogeneous setting, following~\cite{EGnabeyeuPR2025,EGnabeyeuR2025}. Within the framework of Volterra Heston models, our results extend the one-dimensional exponential utility case of~\cite{HanWong2020b} to the multidimensional setting, while encompassing the power utility results of~\cite{HanWong2020b, AichingerDesmettre2021} as special cases.
	  In addition, we incorporate the logarithmic utility case.
\end{itemize}

\medskip
\noindent {\sc \textbf{Organization of the Work.}}  
\noindent The outline of the paper is as follows: Section~\ref{Sect:affine} gives an overview of the model which is needed throughout the paper:  We introduce the multi asset financial
market, where volatility is modeled by a multivariate class of \textit{fake stationary Volterra square root process}, and we state the optimization problem. 
\noindent For such a market model we consider in section~\ref{Sec:SolMerton} 
two different approaches to solve the Merton portfolio problem. More precisely, in section~\ref{Sec:SolMerton} we investigate the classical problem of maximizing the expected utility of terminal wealth in a multi-asset fake stationary Volterra--Heston volatility market, for each of the power, exponential and logarithmic utility preferences.
We first adapt the martingale distortion transformation used in~\cite{fouque2018aoptimal,HanWong2020b} to the (degenerate) multivariate case. However, as it is pointed out in~\cite{AbiJaberMillerPham2021,AichingerDesmettre2021}, this only works if the correlation structure is highly degenerate.
Inspired by the techniques
used in~\cite{HuImkellerMueller2005, BauerleLi2013, AichingerDesmettre2021}, we then provide a solution for the Merton portfolio optimization problem for a more general correlation structure using a verification argument.  In Section~\ref{Sec:Num}, we demonstrate the practical implications of our findings through numerical experiments based on a two-dimensional fake stationary rough Heston model. 
 Finally, Section~\ref{sect:proofMresult} is devoted to the proofs of the main results.

\noindent {\sc \textbf{Notations.}} 

\smallskip
\noindent $\bullet$ Denote $\mathbb{T} = [0, T] \subset \mathbb{R}_+$, ${\rm Leb}_d$ the Lebesgue measure on $(\R^d, {\cal B}or(\R^d))$, $\mathbb H :=\R^d, $ etc.

\noindent $\bullet$ $\mathbb{X} := C([0,T], \mathbb H) (\text{resp.} \quad C_0([0,T], \mathbb H))$ denotes the set of continuous functions(resp. null at 0)  from $[0,T]$ to $\mathbb H $ and ${\cal B}or(C_d)$ denotes the  Borel $\sigma$-field of ${ C}_d$ induces by the $\sup$-norm topology. 

\smallskip 
\noindent $\bullet$ For $p\in(0,+\infty)$, $L_{\mathbb H}^p(\P)$ or simply $L^p(\P)$ denote the set of  $\mathbb H$-valued random vectors $X$  defined on a probability space $(\Omega, {\cal A}, \P)$ such that $\|X\|_p:=(\E[\|X\|_{\mathbb H}^p])^{1/p}<+\infty$. 

\smallskip 
\noindent $\bullet$ Let \(\mathcal{M}\) denote the space of all $(\R_+, {\cal B}or(\R))$-measurable functions \(m\) on \(\mathbb{R}_+\) such that the restriction \(\mu|_{[0, T]}\), for any \(T > 0\), is a \(\mathbb{R}\)-valued finite measure (i.e. the restriction $m|_{[0,T]}$ with $T > 0$ is well-defined). For \(m \in \mathcal{M}\) and a compact set \(E \subset \mathbb{R}_+\), we define the total variation of \(m\) on \(E\) by:

\centerline{$|m|(E) := \sup \left\{ \sum_{j=1}^N |m(E_j)| : \{E_j\}_{j=1}^N \text{ is a finite measurable partition of } E \right\}.$}
\noindent We assume that the set of measure $m \in \mathcal{M}$ on \(\R_+\) is of locally bounded variation.

\smallskip 
\noindent $\bullet$ Convolution between a function and a measure. Let \(f : (0, T] \to \mathbb{R}\) be a measurable function and \(m \in \mathcal{M}\). Their convolution (whenever the integral is well-defined) is defined by
\vspace{-.2cm}
\begin{equation}\label{eq:convolmeasure}
	(f * m)(t)= \int_{[0,t)} f(t - s) \, dm(s) = \int_{[0,t)} f(t - s) \, m(ds) = (f\stackrel{m}{*}\mathbf{1})_t, \quad t \in (0, T].
\end{equation}
\noindent $\bullet$ $X\perp \! \! \!\perp Y$  stands for independence of random variables, vectors or processes $X$ and $Y$.  

\noindent $\bullet$ For a measurable function \( \varphi: \mathbb{R}^+ \to \mathbb{R} \), \(\forall p \geq 1,\) we denote: 

\centerline{$
	  \| \varphi \|^p_{L^p([0,T])} := \int_0^{T} |\varphi(u)|^p \, du,  \; \displaystyle \|\varphi\|_{\infty}=\|\varphi\|_{\sup} := \sup_{u\in \mathbb{R}^+}|\varphi(u)| \; \text{and} \; \displaystyle \|\varphi\|_{\infty,T}=\|\varphi\|_{\sup,T} := \sup_{u\in [0,T]}|\varphi(u)|.
	$}

\smallskip 
\noindent $\bullet$ $\Gamma(a) = \int_0^{+\infty} u^{a-1} e^{-u} \, du, \quad a > 0, \quad 
\text{and} \quad 
B(a, b) = \int_0^1 u^{a-1} (1 - u)^{b-1} \, du, \quad a, b > 0.$
We set \(\R_+=[0,+\infty)\), \(\R_-=(-\infty,0]\).

\smallskip 
\noindent $\bullet$ Let \([0,T]\) be a finite time horizon, where \(T<\infty\). Given a complete probability space $(\Omega,\cF,\P)$ and a  filtration $\F=(\cF_t)_{t \geq 0}$ satisfying the usual conditions (We equip $(\Omega,\cF,\P)$ with a right-continuous, $\P-$complete filtration $\F$), we denote by
\vspace{-.2cm} 
\begin{align*}
	L^{\infty}_{\F}([0,T], \R^d) &= \left\{ Y:\Omega \times [0,T]\mapsto \R^d, \; \F-\text{prog.~measurable and bounded a.s.} \right\} \\
	L^p_{\F}([0,T], \R^d) &= \left\{ Y:\Omega \times [0,T]\mapsto \R^d, \; \F-\text{prog.~measurable s.t.~} \E\Big[ \int_0^T |Y_s|^p ds \Big] < \infty   \right\} 
\end{align*}
\begin{align*}
	{\mathbb{S}^{\infty}_{\F}([0,T], \R^d)} &= \left\{ Y:\Omega \times [0,T]\mapsto \R^d, \; \F-\text{prog.~measurable s.t.~} \sup_{t\leq T} |Y_t(w)|< \infty \mbox{ a.s.} \right\}\\
	{\mathbb{S}^{p}_{\F}([0,T], \R^d)} &= \left\{ Y:\Omega \times [0,T]\mapsto \R^d, \; \F-\text{prog.~measurable s.t.~} \E\Big[\sup_{0\leq t\leq T} |Y_t|^p\Big] < \infty \right\}. 
\end{align*}
Here $|\cdot|$ denotes the Euclidian norm on $\R^d$.  Classically, for $p \in (1, \infty  )$, we define $L^{p, loc}_{\F}([0,T], \R^d)$ as the set of progressive 
processes $Y$ for which there exists a sequence of increasing stopping times 
$\tau_n \uparrow \infty$ such that the stopped processes $Y^{\tau_n}$ are in $L^{p}_{\F}([0,T], \R^d)$ for every $n \geq 1$, and we recall that it consists of all progressive processes $Y$ s.t. 
$ \int_0^T |Y_t|^p dt$ $<$ $\infty$, a.s. Likewise for $\mathbb{S}^{p, loc}_{\F}([0,T], \R^d)$. To unclutter notation, we write $L^{p, loc}_{\F}([0,T])$  instead of $L^{p, loc}_{\F}([0,T], \R^d)$ when 
the context is clear. 

\smallskip 
\noindent $\bullet$  We will use the matrix norm \(|A| = \operatorname{tr}(A^{\top}A)\)
in this paper.

\smallskip 
\noindent $\bullet$ Let $M$ denote a continuous semimartingale. The stochastic exponential
$\mathcal{E}(M)$ is given by
\[
\mathcal{E}(M)_t = \exp\left( M_t - \tfrac{1}{2} \langle M \rangle_t \right), 
\quad t \in [0,T],
\]
where $\langle M \rangle$ denotes the quadratic variation of $M$.

\medskip
\noindent Our problem is defined under a given complete probability space \((\Omega,\mathcal{F},\P)\), with a filtration
\(\mathbb{F} = \{\mathcal{F}_t\}_{0\leq t\leq T}\) satisfying the usual conditions, supporting a \( 2d\)-dimensional Brownian motion
\( (B, B^\top)\) for \(d\geq1\). The filtration \(\mathbb{F}\) is not necessarily the augmented filtration generated by \( (B, B^\top)\) ;
thus, it can be a strictly larger filtration. Here \(\P\) is a real-world probability measure from which a family of equivalent probability measures can be generated. 
\vspace{-.5cm}
\section{Preliminaries: Multivariate fake stationary affine Volterra models}\label{Sect:affine}
\noindent Fix $T > 0$, $d\in \N$.
We let $K=\diag(K_1,\ldots,K_d)$ be diagonal with scalar kernels $K_i\in L^{2}([0,T],\R)$ on the diagonal, $\varphi=\diag(\varphi^1,\ldots,\varphi^d)$, $\nu=\diag(\nu_1,\ldots,\nu_d)$, $\varsigma=\diag(\varsigma^1,\ldots,\varsigma^d)$ with \(\varsigma^i\) a (locally) bounded Borel function and $D:=-\diag(\lambda_1,\ldots,\lambda_d) \in \R^{d\times d}$.
Let $V=(V^1,\ldots, V^d)^\top$  be the following $\R^d_+$--valued scaled Volterra square--root process driven by an $d$-dimensional process $W=(W^1,\ldots,W^d)^\T$:
\vspace{-.2cm}
\begin{equation} 
	\label{VolSqrt_}
	\begin{aligned}
		V_t = \varphi(t) V_0 + \int_0^t K(t-s) \big(\mu(s) + D V_s\big) ds  + \int_0^t K(t-s) \nu \varsigma(s)\sqrt{\diag(V_s)}dW_s, \quad V_0\perp\!\!\!\perp W.
	\end{aligned}
\end{equation}
Here $\mu:\R_+\to \R^{d}$, $W$ is a $d$-dimensional Wiener process.
Note that the drift $b(t,x) = \mu(t)+ Dx$ is clearly Lipschitz continuous in $x\in\R^d$,  uniformly in $t\!\in \mathbb{T}_+$ and both the drift term \(b\) and the diffusion coefficient \(\sigma(t, x)=\nu\varsigma(t)\sqrt{\diag(x)}\) are of linear growth, i.e. there is a constant \(C_{b,\sigma} > 0\) such that
\vspace{-.1cm}
\[
\|b(t, x)\| + \||\sigma(t, x)|\| \leq C_{b,\sigma}(1 + \|x\|),
\quad \text{for all } t \in [0, T] \text{ and } x \in \mathbb{R}^d.
\] 
We always work under the assumption below, which applies to the inhomogeneous Volterra equation~\eqref{VolSqrt_}.
\begin{assumption}[On Volterra Equations with convolutive kernels]\label{assump:kernelVolterra} 
	Assume that \(K\) is diagonal with scalar kernels $K_i$ on the diagonal for \(i=1,\ldots,d\) that is completely monotone on $(0, \infty)$ and satisfies for any $T > 0$:
	\noindent
	\begin{enumerate}
		\item[(i)] The kernel  $K_i$ is strictly positive 
		and fulfills:
		\begin{itemize}
			\item The integrability assumption: The following is satisfied for some $\widehat\theta_i\in (0,1]$.
			\vspace{-.2cm}
			\begin{equation}\label{eq:contKtilde}
				(\widehat {\cal K}^{cont}_{\widehat \theta})\;\;\exists\,\widehat{\kappa_i}< +\infty,\;\forall\bar{\delta}\!\in (0,T],\; \widehat \eta(\delta) :=  \sup_{t\in [0,T]} \left[\int_{(t-\bar{\delta})^+}^t \hskip-0,25cm K_i\big(t-u\big)^2 du\right]^{\frac12}\le \widehat \kappa_i \,\bar{\delta}^{\,\widehat \theta_i}.
			\end{equation}
			\item  The continuity assumption: \(({\cal K}^{cont}_{\theta}) \;\;  
			\exists\, \kappa_i< +\infty,\; \exists \; \theta_i\in (0,1] \; \text{such that}\; \forall \,\bar{\delta}{\in (0, T)}\)
			\vspace{-.2cm}
			\begin{equation}\label{eq:Kcont}
				({\cal K}^{cont}_{\theta}) \; \forall \,\bar{\delta}{\,\in (0, T)},\; \eta(\bar{\delta}):= \sup_{t\in [0,T]} \left[\int_0^t |K_i(\big(s+\delta)\wedge T\big)-K_i(s)|^2ds \right]^{\frac 12} \le  \kappa_i\,\bar{\delta}^{\theta_i}.
			\end{equation}
			
		\end{itemize}  
		\item[(ii)] Finally, assume that \( V_0^i \in L^p(\mathbb{P}) \) for some suitable \( p \in (0, +\infty) \), such that
		the process $t \to v_0^i(t) =V_0 \varphi^i(t)$ is absolutely continuous and $(\mathcal F_t)$-adapted.
		Moreover, for some $\delta_i > 0$, for any $p > 0$,  
		\[
		\mathbb{E} \,\!\Big(\sup_{t \in [0,T]} |v_0^i(t)|^p\Big) < +\infty,\quad 
		\mathbb{E}\!\big[\,|v_0^i(t') - v_0^i(t)|^p\,\big] 
		\le C_{T,p} \Big( 1 + \mathbb{E}\,\big[\sup_{t \in [0,T]} |v_0^i(t)|^p\big] \Big) |t' - t|^{\delta_i p}.
		\]
	\end{enumerate}
\end{assumption}
\vspace{-.2cm}
\noindent {\bf Remark:} For \(i=1,\ldots,d\) , as $K_i$ is completely monotone on $(0,\infty)$ and not identically zero, we have that
$K_i$ is nonnegative, not identically zero, non-increasing and continuous on $(0,\infty)$, and it follows from \cite[Theorem 5.5.4]{gripenberg1990} that $K_i$ has a resolvent of the first kind $r_i$ which is nonnegative and non-increasing in the sense that \(s \mapsto r_i([s,s+t])\)
is non-increasing $\forall t \ge 0$.

\medskip
\noindent In the case of \(\alpha-\) fractional kernel (corresponding to \(K_i=K_{\alpha_i}\) with \(\alpha_i \in [\frac12,1)\)), by~\cite{EGnabeyeuR2025}, Equation~\eqref{VolSqrt_} admits at least a unique-in-law positive weak solution as a scaling limit of a sequence of 
time-modulated Hawkes processes with heavy-tailed kernels in a nearly unstable regime.
Moreover, under assumption~\ref{assump:kernelVolterra} for some $p>0 $, a solution \( t \mapsto V_t^i \) to  Equation~\eqref{VolSqrt_}  starting from   $V_0^i$ 
has a \( \big( \delta_i \wedge \theta_i \wedge \widehat \theta_i - \eta \big) \)-H\"older pathwise continuous modification  on $\R_+$ for sufficiently small \( \eta > 0 \) and satisfying (among other properties),  
\vspace{-.3cm}
\begin{equation}\label{eq:L^p-supBound}
	\forall\, T>0, \; \exists \,C_{_{T,p} }>0,\quad \big\| \sup_{t\in[0,T]}\|V_t\| \big\|_p \le C_{_{T,p} } \left( 1 + \big\| \sup_{t \in [0,T]} \|\varphi(t)V_0 \|\big\|_p \right).
\end{equation}
Note that under our assumptions, if \( p > 0 \) and \( \mathbb{E}[\|\varphi(t)V_0\|^p] < +\infty \) for every \(t\geq0\), then by~\eqref{eq:L^p-supBound}, \( \mathbb{E}[\sup_{t \in [0,T]} \|V_t\|^p] < C_T (1 + \mathbb{E}[\sup_{t \in [0,T]}\|\varphi(t) V_0\|^p]) < +\infty \) for every \( T > 0 \). Combined with the linear growth in Assumption~\ref{assump:kernelVolterra}(ii) \( \||\sigma(t,x)|\| \leq C'_T(1 + \|x\|) \) for \( t \in [0,T] \), this implies \( \mathbb{E}[ \sup_{t \in [0,T]}  \||\sigma(t, X_{ t})|\|^p] < C'_T (1 +\mathbb{E}[\sup_{t \in [0,T]}\|\varphi(t) V_0\|^p]) < +\infty \) for every \( T > 0 \), enabling the unrestricted use of both regular and stochastic Fubini's theorems.
Sufficient conditions for interchanging the order of ordinary integration (with respect to a finite measure) and stochastic integration (with respect to a square integrable martingale) are provided in  \cite[Thm.1]{Kailath_Segall},  and further details can be found in \cite[Thm. IV.65]{Protter}, \cite[ Theorem 2.6]{Walsh1986}, \cite[ Theorem 2.6]{Veraar2012}.
\begin{Remark}\label{rm:Kernels}
	This covers, for instance,  constant non-negative kernels,  fractional kernels  of the form $\frac{t^{\alpha-1}}{\Gamma(\alpha)}\mathbf{1}_{\mathbb{R}_+}$ with $\alpha  \in(\frac12,1]$, exponentially decaying kernels ${\rm e}^{-\beta t}$ with $\beta>0$ and more generally the gamma kernel \(K(t) = \frac{t^{\alpha-1}}{\Gamma(\alpha)} e^{-\beta t}\mathbf{1}_{\mathbb{R}_+}
	\) with \( \alpha \in \left( \tfrac{1}{2}, 1 \right] \) and \( \beta \ge 0 \) ( see e.g. \cite[Propositions 6.1 and 6.3]{EGnabeyeu2025}, \cite[ Example 2.2 ]{GnabeyeuPages2026}). These kernels satisfy conditions~\eqref{eq:contKtilde}--\eqref{eq:Kcont}, that is, $(\widehat {\cal K}^{cont}_{\widehat \theta})$ and $({\cal K}^{cont}_{\theta})$, for $\alpha >1/2$ with $\theta = \widehat \theta = \min\bigl( \alpha-\frac12,\; 1\bigr)$.
	
	\noindent The roughness of the volatility paths is determined by the parameter $\alpha$ linked to the Hurst parameter $H$ via the relation $\alpha=H+\frac{1}{2}$.
	 For $\alpha\rightarrow 1$ we recover the classical markovian square root process.
\end{Remark}

\subsection{Stabilizer and fake stationarity regimes.}\label{subsec:fakeStat}
\begin{Definition}[Fake Stationarity Regimes]
	Let \( (V_t)_{t \geq 0} \) be a solution to the scaled Volterra equation~\eqref{VolSqrt_} starting from any \( V_0 \in L^2 (\P) \). 
	Then, the process $(V_t)_{t\ge 0}$ exhibit a \textit{fake stationary regime of type I} in the sense of \cite{Pages2024,EGnabeyeu2025} if it has constant mean and variance over time i.e.:
	{\small
		\begin{equation}\label{eq:fs1_}
			\forall\, t\ge 0, \quad \mathbb{E}[V_t] = \textit{c}^{\text{ste}} \quad \mbox{and} \quad \text{Var}(V_t) = \textit{c}^{\text{ste}} = v_0 \in \R_+^d.
		\end{equation}
	}
\end{Definition}
\noindent For every $\lambda \!\in \R$,  the \textit{ resolvent or Solvent core} $R_{\lambda}$ associated to a real-valued kernel $K$, known as the \textit{ $\lambda$-resolvent of $K$} is defined as the unique solution -- if it exists --  to the deterministic Volterra equation
\begin{equation}\label{eq:Resolvent_}
	\forall\,  t\ge 0,\quad R_{\lambda}(t) + \lambda \int_0^t K(t-s)R_{\lambda}(s)ds = 1.
\end{equation}
or, equivalently, written in terms of convolution, 
$R_{\lambda}+\lambda K*R_{\lambda} = 1$ and admits the formal \textit{Neumann series expansion} \(R_{\lambda} =\mbox{\bf 1}* \big(\sum_{k\ge 0} (-1)^k \lambda^k K^{k*}\big)\)
where \(K^{k*}\) denotes the \(k\)-th convolution of \(K\)  with the convention, $K^{0*}= \delta_0$ (Dirac mass at $0$).

\medskip
\noindent{\bf Remark} If $K$ is regular enough (say continuous) the resolvent $R_{\lambda}$ is differentiable and one checks that $f_{\lambda}=-R'_{\lambda}$ satisfies for every  $t>0$, \(-f_{\lambda}(t) +\lambda \big( R_{\lambda}(0)K(t) - K *f_{\lambda}(t)\big)=0\)
that is $f_{\lambda}$ is solution to the equation
\begin{equation}\label{eq:flambda-eq}
	f_{\lambda} +\lambda K *f_{\lambda}=\lambda   K \quad \text{and reads}\quad f_{\lambda} = \sum_{k\ge 1} (-1)^k \lambda^k K^{k*},\quad K^{0*}= \delta_0. 
\end{equation}
\begin{Example}\label{Ex:fractionalkernel}
	Denote by \(E_\alpha\) the standard Mittag-Leffler function.  For the \(\alpha-\)fractional kernels defined in Remark~\ref{rm:Kernels}  the identity \(K_\alpha * K_{\alpha'} = K_{\alpha+\alpha'}\) holds for \(t \geq 0\) so that
	\vspace{-.2cm}
	\begin{equation*}
		R_{\alpha,\lambda}(t) 
		= \sum_{k \geq 0} (-1)^k \frac{\lambda^k t^{\alpha k}}{\Gamma(\alpha k + 1)} 
		= E_\alpha(-\lambda t^\alpha), \; \text{and } \;
		f_{\alpha,\lambda}(t) = -R'_{\alpha,\lambda}(t) 
		= \lambda t^{\alpha - 1} \sum_{k \geq 0} (-1)^k \lambda^k 
		\frac{t^{\alpha k}}{\Gamma(\alpha(k+1))}.
	\end{equation*}
\end{Example}
\vspace{-.2cm}
\noindent We will always work under the following assumption.
\begin{assumption}[$\lambda$-resolvent $R_{\lambda}$ of the kernel]\label{ass:resolvent} For \(i=1,\cdots,d\), we assume that the $\lambda_i$-resolvent $R_{\lambda_i}$ of the kernel $K_i$ satisfies the following for every $\lambda_i > 0$:
	\begin{equation}\label{eq:hypoRlambda_}
		({\cal K})\quad
		\left\{
		\begin{array}{ll}
			(i) & R_{\lambda_i}(t) \text{ is } \text{differentiable on } \mathbb{R}^+,\; R_{\lambda_i}(0)=1 \text{ and } \lim_{t \to +\infty}R_{\lambda_i}(t) =a_i \in [0,1[, \\
			(ii) &   f_{\lambda_i} \in {\cal L}_{\text{loc}}^2(\mathbb{R}_+, \text{Leb}_1), \; \text{ for } \; t > 0,\; L_{f_{\lambda_i}}(t) \neq 0\; dt-a.e., \text{ where } f_{\lambda_i} := -R'_{\lambda_i},\\
			(iii) & \varphi^i \in {\cal L}^1_{\mathbb{R}_+}(\text{Leb}_1), \text{ is continuous, satisfying} \; \lim_{t \to \infty}\varphi^i(t) = \varphi_\infty^i, \text{ with } a_i \varphi_\infty^i < 1, \\
			(iv) & \mu \text{ is a 
				$ C^1$-function such that }  \|\mu\|_{\sup}  <\infty  \text{ and }  \lim_{t\to +\infty} \mu (t) = \mu_{\infty} \in \mathbb{R}^d.
		\end{array} 
		\right.
	\end{equation}
\end{assumption}
\vspace{-.2cm}
\noindent {\bf Remark:}
Under the assumption \(	({\cal K})\), \( f_{\lambda_i} \) is a \((1-a_i)\)-sum measure, i.e., \( \int_0^{+\infty} f_{\lambda_i}(s) \, ds = 1-a_i \). Furthermore, \(\lim_{t\to +\infty} \int_0^t f_{\lambda_i}(t-s) \mu^i (s)ds = \mu_{\infty}^i \)  and 	$
\lim_{t \to +\infty} \varphi^i(t) - (f_{\lambda_i} * \varphi^i)(t) = \varphi_\infty^i \,a_i.
$ (see  \cite[Lemma 3.1]{EGnabeyeu2025}).
Finally, if $f_{\lambda_i} = -R'_{\lambda_i} > 0 \text{ for } t > 0$, then $f_{\lambda_i}$ is a probability density in which case,  $R_{\lambda_i}$  is non-increasing.
This is in particular the case for the Mittag-Leffler density function \( f_{\alpha_i, \lambda_i} \) for \( \alpha_i \in(\frac12,1)\), in which case \(f_{\alpha_i, \lambda_i}\) is a completely monotonic function (hence convex), decreasing to 0 while \(1-R_{\alpha_i, \lambda_i}\) is a Bernstein function (see e.g. \cite[Proposition 6.1]{EGnabeyeu2025}).
The Proposition below shows what are the consequences of the three constraints in equation~\eqref{eq:fs1_}.
\begin{Proposition}[Fake stationary Volterra square root process.]\label{prop:timeDen_} 
	Let \( (V_t)_{t \geq 0} \) be a solution to the scaled Volterra square root equation in its form~\eqref{VolSqrt_} starting from any random variable $V_0\in L^2(\Omega, \mathcal{F}, \mathbb{P})$. 
	Then, a necessary and sufficient condition for the relations~\eqref{eq:fs1_} to be satisfied is that for \(i=1,\ldots,d\)
	\vspace{-.1cm}
	{\small	
		\begin{align}
			&\ \mathbb{E}[V_0^i] = \frac{1-a_i}{1-a_i\varphi_\infty^i}\frac{\mu_\infty^i}{\lambda_i}:= v_\infty^i \quad \text{and} \quad	\forall\, t\ge 0, \quad \varphi^i(t)   =1 - \lambda_i \int_0^t K_i(t-s) \left( \frac{\mu^i(s)}{\lambda_i v_\infty^i} - 1 \right) \, \dd s. \label{eq:CondMean_}
			\\ 
			&\text{so that~\eqref{VolSqrt_} reads:}\;	V_t^i = V_0^i - \frac{1}{\lambda_i v_\infty^i}\Big(V_0^i - v_\infty^i\Big) \int_0^t f_{\lambda_i}(t-s) \mu^i(s)\dd s +  \frac{1}{\lambda_i}\int_0^t f_{ \lambda_i}(t-s)\varsigma^i(s)\sqrt{ V^i_{s}}dW^i_s.\label{eq:ConstMean_}
		\end{align}
	}
	\noindent and the couple \( (v_0^i, \varsigma^i(t)) \), where \( v_0^i = \text{Var}(V_0^i) \) must satisfy the functional equation:
	\vspace{-.1cm}
	\begin{equation} \label{eq:VolterraVarTime_1}
		\textit{($E_{\lambda_i, c_i}$)}: \;\forall\, t\ge 0, \; c_i \lambda_i^2\big(1-(\varphi^i(t)-(f_{\lambda_i} * \varphi^i)_t)^2 \big) =  (f_{\lambda_i}^2 * \varsigma^{i 2})(t) \;  \textit{where} \; c_i = \frac { v_0^i }{\nu_i^2v_\infty^i} \;  \textit{i.e.} \; \varsigma^i = \varsigma^i_{\lambda_i,c_i} .
	\end{equation}
\end{Proposition}
\noindent {\bf Proof :} This is a  straightforward extension to the multi-dimensional setting of \cite[Proposition 3.4 and Theorem 3.5]{EGnabeyeu2025} (see also \cite[Proposition 4.2 and 4.4]{EGnabeyeuPR2025}).

\begin{Definition}
	We will call the stabilizer (or corrector) of the scaled stochastic Volterra equation ~\eqref{VolSqrt_} the (locally) bounded Borel function $\varsigma=\diag(\varsigma^1,\ldots,\varsigma^d)$ where \( \varsigma^i \) is a solution(if any) to the functional equation \(\textit{($E_{\lambda_i, c_i}$)}\) in~\eqref{eq:VolterraVarTime_1} for \(i=1,\ldots,d\).
\end{Definition}
\begin{Example}\label{Ex:FractionalKernel2}
	Within the setting
	\(\varphi^i(t) = \varphi^i(0) = 1\) for all  \(t \geq 0 \) and \(K_i\) the \(\alpha-\)fractional kernel
	defined in Remark~\ref{rm:Kernels} and Example~\ref{Ex:fractionalkernel} with \(\alpha_i \in \left(\frac{1}{2}, 1\right)\), we have \(\lim_{t \to +\infty} R_{\alpha_i,\lambda_i}=0\). Setting \(a_k = \frac{1}{\Gamma(\alpha k + 1)},
	b_k = \frac{1}{\Gamma(\alpha(k + 1))}, \; k \geq 0\), then the stabilizer \( \varsigma = \varsigma_{\alpha_i,\lambda_i,c_i} \) exists as a non-negative, non-increasing concave function, on \( (0, +\infty) \) (see \cite[Sections 5.1 and 5.2 ]{Pages2024}, \cite[Sections 5.1 and 5.2 ]{EGnabeyeu2025}), such that:
	
	\(\varsigma^2_{\alpha_i,\lambda_i,c_i}(t) = c_i \lambda_i^{2-\frac1{\alpha_i}}\varsigma_{\alpha_i}^2(\lambda_i^{\frac1{\alpha_i}} t)\) where \(\varsigma_{\alpha_i}^2(t):= 2\,t^{1-\alpha_i}\sum_{k\ge 0} (-1)^k c_k t^{\alpha_i k}\) and  the coefficients $(c_k)_{k\geq0}$ are defined by the recurrence formula \(c_0=\frac{\Gamma(\alpha)^2}{\Gamma(2\alpha-1)\Gamma(2-\alpha)}\) and for every \(k\ge 1\)
	{\small 
		\begin{equation}\label{eq:ck_}
			c_k = \frac{\Gamma(\alpha)^2\Gamma(\alpha(k+1))}{\Gamma(2\alpha-1)\Gamma(\alpha k+2-\alpha)}\left[ (a*b)_k- \alpha(k+1)\sum_{\ell=1}^kB\big(\alpha(\ell+2)-1,\alpha(k-\ell-1)+2\big) (b^{*2})_{\ell}  c_{k-\ell}  \right].
		\end{equation}
	}
	where for two sequences of real numbers \( (u_k)_{k \geq 0} \) and \( (v_k)_{k \geq 0} \), the Cauchy product is defined as \( (u * v)_k = \sum_{\ell = 0}^k u_\ell v_{k - \ell} \) and \( B(a, b) = \int_0^1 u^{a-1}(1 -
	u)^{b-1} du\) denoting the beta function. 
	
	\noindent Moreover, \( \left( \liminf_k \left( |c_k|^{1/k} \right) \right)^{-1/\alpha} =\infty\), \(\varsigma_{\alpha_i,\lambda_i,c_i}(0) =0\) and \(\lim_{t \to +\infty} \varsigma_{\alpha_i,\lambda_i,c_i}(t) = \frac{\sqrt{c_i}\lambda_i}{\|f_{\alpha_i,\lambda_i}\|_{L^2(\text{Leb}_1)}}\).
\end{Example}
\noindent Set \(E_{D,c} = \bigcup_{i=1}^{d} E_{\lambda_i, c_i}\). From now on, we will assume that there exists a unique positive bounded Borel solution \(\varsigma = \varsigma_{D,c}\) on \((0,+\infty)\) of the system of equation \((E_{D, c})\) so that, the corresponding time-inhomogeneous Volterra square root equation~\eqref{VolSqrt_} is refered to as a \textit{Multivariate Stabilized Volterra Cox-Ingersoll-Ross (CIR) equation} or 
as a \textit{Multivariate fake stationary Volterra CIR equation} if, in addition, equation~\eqref{eq:CondMean_} holds. The function \(\varsigma\) can be interpreted as a control acting on the volatility process~\eqref{VolSqrt_}, thereby ensuring that its second moment remains constant over time (see, e.g., Figure~\ref{fig:_variance}).
\vspace{-.3cm}
\subsection{Formulation of the stochastic Market model}
\noindent We  consider a financial market on $[0,T]$  on some filtered probability space $(\Omega,\cF,\F:=(\cF_t)_{t \geq 0},\P)$ with \(d+1\) securities, consisting of a  bond and \(d\) stocks. The non--risky asset  $S^0$ satisfies the (stochastic) ordinary differential equation: 
\vspace{-.3cm}
\begin{align*}
	dS^0_t = S^0_t r(t) dt,
\end{align*}
with a time-dependent deterministic  short risk-free rate $r:\R_+ \to \R$, and  $d$ risky assets (stock or index) whose return vector process $(S_t)_{t \ge 0} = (S_{t}^1, \ldots, S_{t}^d)_{t \ge 0}$ is defined via the dynamics given by the vector-stochastic differential equation (SDE):
\vspace{-.1cm}
\begin{align}
	\label{eq:stocks}
	dS_t = \diag(S_t) \big[ \big( r(t) {\bold{1}_d} + \sigma_t \lambda_t  \big)dt + \sigma_t dB_t \big],
\end{align}
driven by a $d$-dimensional Brownian motion $B$, with a $d\times d$-matrix valued  continuous stochastic volatility process $\sigma$ whose dynamics is driven by~\eqref{VolSqrt_}
and a $\R^d$-valued continuous stochastic process $\lambda$, 
called {\it market price of risk}.  Here ${\bold{1}_d}$ denotes the vector in $\R^d$ with all components equal to $1$ and the correlation structure of $W$ with $B$, denoted \(\Sigma:=[\Sigma_1,\cdots,\Sigma_d]\) is given by
\vspace{-.1cm}
{\begin{align}\label{eq:correstructureheston}
		W^i = \rho_i B^i + \sqrt{1-\rho_i^2} B^{\perp,i} =  \Sigma_i^\top  B_t + \sqrt{1-\Sigma_i^\T \Sigma_i} B^{\perp,i}_t, \quad i=1,\ldots,d,
	\end{align}
	for some $(\rho_1,\ldots,\rho_d)\in[-1,1]^d$},
where $(0,\ldots,\rho_i,\ldots,0)^\T:=\Sigma_i \in \R^{d}$ is such that $\Sigma_i^\T \Sigma_i\leq1$,  and $B^{\perp}$ $=$ $(B^{\perp,1},\ldots,B^{\perp,d})^\T$ is an $d$--dimensional Brownian motion independent of $B$. The correlation $\rho_i $ between stock price \(S^i\) and variance \(V^i\) is assumed constant.
Note that $d\langle W^i \rangle_t = dt$  but $W^i$ and $W^j$ can be correlated, hence $W$ is not necessarily a Brownian motion.

\medskip
\noindent Observe that 
processes $\lambda$ and $\sigma$ are $\F$-adapted, possibly unbounded,  but not necessarily adapted to the filtration generated by $W$.  We point out that $\F$ may be strictly larger than the augmented filtration generated by $B$ and $B^{\perp}$ as we deal with weak solutions to stochastic Volterra equations. 

\medskip
\noindent
We assume that  $\sigma$ in~\eqref{eq:stocks} is given by $\sigma = \sigma(V) = \sqrt{\diag(V)}$, where the $\R^d_+$--valued scaled process $V$  is defined in~\eqref{VolSqrt_} with \(\varsigma = \varsigma_{D,c}\) and Equation~\eqref{eq:CondMean_} holds true. We will be chiefly interested in the case where \(\lambda_t\) is linear
in \(\sigma_t\). More specifically, the the market price of risk (risk premium) is assumed to be in the form
{$\lambda$ $=$ $\big(\theta_1\sqrt{V^1},\ldots,\theta_d \sqrt{V^d}\big)^\top$}, for some constant {$\theta_i \geq 0$},  so that the dynamics  for the stock prices \eqref{eq:stocks} reads following \cite{Kraft2005,abi2019affine}
\begin{align}
	\label{eq:hestonS}
	dS^i_t = S^i_t \left( r(t)  + \theta_i V^i_t   \right) dt + S^i_t \sqrt{V^i_t} dB^i_t, \quad i=1,\ldots, d.
\end{align}
Since \(S\) is fully determined by \(V\), the existence of $S$ readily follows from that of $V$. In particular, weak existence of H\"older pathwise continuous solution $V$ of~\eqref{VolSqrt_} such that~\eqref{eq:L^p-supBound} holds is established  under suitable assumptions on the kernel $K$ and specifications $g_0$ as shown in the following remark. 

\medskip
\noindent We state the following existence and uniqueness result from \cite{EGnabeyeuR2025} which is extended to the  multi-dimensional setting.
\begin{Theorem}(\cite[Theorem 3.1 and Remark on Theorem 3.2]{EGnabeyeuR2025}). Under Assumption \ref{assump:kernelVolterra}, the stochastic Volterra equation (\ref{eq:hestonS})-(\ref{VolSqrt_}) has a unique in law continuous $\R^{d}_+ \times \R^{d}_+$-valued weak solution $(S,V)$ for any initial condition $(S_0, V_0) \in \R^{d}_+ \times \R^{d}_+$ defined on some filtered probability space $(\Omega,\mathcal F, (\mathcal F)_{t\geq 0}, \mathbb P)$ such that 
	\vspace{-.4cm}
	\begin{align}\label{eq:moments V1}
		\sup_{t\leq T} \E\left[ \|V_t\|^p \right] < \infty, \quad p > 0.
	\end{align}
\end{Theorem}
\vspace{-.2cm}
\noindent From now on, we set \( g_0(t):= \varphi(t) V_0 + \int_0^t K(t-s)\mu(s) ds \)  the initial input curve and \(Z_t := \int_{0}^{t}D V_s ds + \nu\varsigma(s) \sqrt{\diag(V_s)}dW_s,\) for all \( t \geq 0\) so that Equation~\eqref{VolSqrt_} reads \(\P \otimes dt-\)a.e.
\vspace{-.2cm}
\begin{equation} 
	\label{VolSqrt2}
	\begin{aligned}
		V_t = g_0(t) + \int_0^t K(t-s) D V_s ds  + \int_0^t K(t-s) \nu \varsigma(s)\sqrt{\diag(V_s)}dW_s =  g_0(t) + \int_0^t K(t-s) dZ_s.
	\end{aligned}
	\vspace{-.1cm}
\end{equation}
Finally, we consider the $\R^d$-valued process for \( s\geq t,\) 
\vspace{-.2cm} 
\begin{align}\label{eq:processg}
	g_t(s)= g_0(s) + \int_0^t K(s-u) \big(D V_u du + \nu \varsigma(s) \sqrt{\diag(V_u)}dW_u\big) =  g_0(s) + \int_0^t K(s-u) dZ_u.
\end{align}
One notes that for each, $0\leq s\leq T$, $(g_t(s))_{t\leq s}$ is the adjusted forward process 
\vspace{-.2cm}
\begin{equation}\label{eq:Condprocessg}
	g_t(s) = \; \mathbb E^\P\Big[  V_s - \int_t^s K(s-u)DV_udu \Mid \cF_t\Big],\quad \P \text{- a.s., for a.e.}\; s > t.
	\vspace{-.1cm}
\end{equation}
\noindent This adjusted forward process is commonly used (see, e.g.,~\cite{AbiJaberMillerPham2021}) to elucidate the affine structure of affine Volterra processes with continuous trajectories.\\
More generally, in what follows, under a new probability measure \(\widetilde{\mathbb{P}}\sim \mathbb{P} \), supporting the \(d-\) dimensional Wiener Process \(\widetilde{W}\), we will denotes by 
\(\widetilde{g}_t(s)\), the conditional $\widetilde{\mathbb{P}}$-expected ajusted variance process, namely,
\begin{equation}\label{eq:Condprocessg_}
	\widetilde{g}_t(s):= \mathbb E^{\widetilde{\P}}\Big[  V_s - \int_t^s K(s-u)H_u V_udu | \cF_t\Big],\quad \widetilde{\mathbb{P}} \text{- a.s., for a.e.}\; s > t.
	\vspace{-.1cm}
\end{equation}
 where the dynamics of \(V\) under \(\widetilde{\mathbb{P}}\) reads
\begin{equation} 
	\label{VolSqrt2_}
	\begin{aligned}
		V_t =  g_0(t) + \int_0^t K(t-s) d\widetilde{Z}_s,\quad \text{with} \quad \widetilde{Z}_t := \int_{0}^{t}H_s V_s ds + \nu\varsigma(s) \sqrt{\diag(V_s)}d\widetilde{W}_s.
	\end{aligned}
	\vspace{-.1cm}
\end{equation}
\noindent The process in~\eqref{VolSqrt2} is non-Markovian and non-semimartingale in general. Note that our model (\ref{eq:hestonS})-(\ref{eq:correstructureheston})-(\ref{VolSqrt_}) features correlation between the stocks and between a stock and its volatility. Moreover, the methodology developed in this paper, and hence the results obtained, remain valid if the matrix \( D \) in \eqref{VolSqrt_} is not assumed to be diagonal, but only satisfies
\vspace{-.2cm}
\[
D \in \mathbb{R}^{d \times d}, 
\qquad 
D_{ij} \ge 0 \ \text{for } i \neq j.
\]
\vspace{-.1cm}
This also provides an extension to the inhomogeneous setting of the models considered in~\cite{abi2019affine,AbiJaberMillerPham2021,AichingerDesmettre2021}. Consequently, our main results (Theorems~\ref{Thm:powerUtilityGeneral} and~\ref{Thm:ExpoUtilityGeneral}) extends \cite[Theorem 3.3 and Theorem 3.6]{HanWong2020b}, \cite[Theorem 3.2]{AichingerDesmettre2021} to the multivariate, \textit{time-dependent} diffusion coefficient case. 

\vspace{-.4cm}
\section{ Merton's portfolio problem: Utility maximisation}\label{Sec:SolMerton} 
\vspace{-.2cm}
\noindent {$\rhd$ {\em Preliminaries and Problem formulation}:}  As we deal with weak solutions to stochastic Volterra equations (\ref{eq:hestonS})-(\ref{VolSqrt_}), Brownian motion is also a part of the solution. However, expected utility only depends on the expectation of the wealth process. In the sequel, we fix a version of the solution $(S, V, B, B^\T)$ to  (\ref{eq:hestonS})-(\ref{VolSqrt_}) as other solutions have the same law.

\smallskip
\noindent
We consider the classical problem of maximizing the expected utility of terminal wealth. 
A portfolio strategy $\alpha_t = (\alpha_{t,1}, \ldots, \alpha_{t,d})^\top$ is an $\mathbb{R}^d$-valued,  $\mathbb{F}$-progressively measurable process, where $\alpha_{t,k}$ represents the proportion of wealth invested in asset $k$ at time $t$. Under a fixed portfolio strategy, the corresponding wealth process $(X_t^\alpha)_{t \ge 0}$ depend on \(V\) in (\ref{VolSqrt_}) and has a certain dynamic to be specified latter.
By \(\mathcal{A}\) we denote the set of admissible portfolio or investment strategies i.e. the set of all $\mathbb{F}$-progressively measurable processes $ (\alpha_t)_{t \in [0,T]}$ valued in the Polish space  \(\R^d\).

\medskip
\noindent To ease notation burden, whenever the context is clear, we simply write $X$, instead of $X^{\alpha}$, as the wealth process under $\alpha \in \cA$ with initial condition $X_0 = x_0 > 0$ and $V_0> 0$.
\vspace{-.1cm}
\begin{Definition}
	\begin{enumerate}
		\item A utility function is a strictly increasing and strictly concave function \(U : E \subset \R \to \mathbb{R} \cup \{-\infty\},\)
		which is continuously differentiable on $E$.
		
		\noindent In the following, $U$ denotes a general utility function. Later, we will focus on power utility functions of the form \(U(x) = \frac{x^{\gamma}}{\gamma}, 
		\; \text{for } \gamma \in \mathbb{R}_+ \setminus \{0,1\},\)
		or alternatively on exponential utility functions of the form \(U(x) = -\frac1\gamma e^{-\gamma x}, 
		\; \text{for } \gamma > 0.\) \(\gamma\) is typically called the \textit{risk aversion parameter}.
		
		\item An admissible strategy $\alpha \in \mathcal{A}$ is said to be optimal (for terminal wealth) if it maximizes \(\alpha \longmapsto \mathbb{E}\big[ U(X_T^\alpha) \big]\)
		over all admissible strategies in $\mathcal{A}$. That is the portfolio strategy \(\alpha^*\) for which the supremum is attained 
	\end{enumerate}
\end{Definition}
\noindent
The investor's goal in the Merton problem is to find an optimal strategy so as to maximize the expected utility of terminal wealth. Specifically, given the utility function $U$ on $(0,\infty)$ and starting from an initial capital $x_0>0$,
the objective of the agent is
\vspace{-.2cm} 
\begin{equation}\label{eq:value0}
	\mathcal{V}(x_0,V_0):= \sup_{\alpha(\cdot) \in \mathcal A}
	\mathbb E_{x_0, V_0}\!\left[ U\!\left(X_T^\alpha\right) \right],  \;\text{given} \; x_0 \;\text{and} \;V_0.
\end{equation}
with $X^\alpha$ the wealth
controlled by $\alpha \in \mathcal A$, starting from $x_0$ at time $0$
and $\mathcal A$ the subset of controls $\alpha \in \mathcal A$ such that the family \(\left\{ U\!\left(X_\tau^\alpha\right) : \tau \in [0,T] \right\}\)
is uniformly integrable.
The preferences of the agent are thus described by the utility function $U$. 
Here $\mathbb{E}_{x_0, V_0}$ denotes expectation under the conditional distribution given $X_0 = x_0$ and $V_0$, and the supremum is taken over all admissible portfolio strategies. 
Mathematically, denoting by \(\mathcal S_t := \big(X_t^\alpha, g_t(\cdot)\big)\) the state variable, we aim at identifying the optimal value associated with the utility maximization problem defined over a finite horizon $T$ by 
\vspace{-.15cm}
\begin{equation}\label{eq:valueFunction}
	\mathcal{V}_t(\mathcal S_t ):= \operatorname*{ess\,sup}_{\alpha \in \mathcal A_t}
	\mathbb E\!\,\big[ U\!\big(X_T^{\alpha}\big) \mid \mathcal S_t  \big],\qquad t \in [0,T].
	\vspace{-.1cm}
\end{equation}
and the optimal strategy $\alpha^\ast$, given the preference described by the utility function $U(\cdot)$. 
This is a standard stochastic control problem. The set $\mathcal A_t$ is the class of all admissible strategies from t 
with zero being an absorbing state for $X_t^\alpha$ (bankruptcy).

\medskip
\noindent The  $\mathbb F$-adapted process $\{\mathcal{V}_t,\, 0 \le t \le T\}$
is a supermartingale.
Moreover, there exists an optimal control $\alpha^* \in \mathcal A$ for~\eqref{eq:value0}
if and only if the martingale property holds, that is,
the process $\{\mathcal{V}^*_t,\, 0 \le t \le T\}$ for \(\alpha^*\) is a martingale.
The idea is that an admissible control (or startegy) is optimal
 if the associated value process is a martingale and for any other admissible control, it is a supermartingale. That is the classic martingale \textit{optimality principle}, see, e.g., \cite{HuImkellerMueller2005}, \citet[Section 6.6.1]{pham2009continuous} or \cite{JeanblancEtAl2012}. 
\vspace{-.2cm} 
\begin{Definition}[Martingale optimality principle]\label{def:Martopt}
	The Problem (\ref{eq:value0}) can be solved by constructing a family of processes $\{ J^\alpha_t \}_{ t \in [0, T]}$, $\alpha \in \cA$, satisfying the conditions:
	\begin{enumerate}
		\item $J^\alpha_T = U(X_T)$ for all $\alpha \in \cA$; 
		\item $J^\alpha_0$ is a constant, independent of  $\alpha \in \cA$;
		\item $J^\alpha_t$ is a supermartingale for all $\alpha \in \cA$, and there exists  $\alpha^* \in \cA$ such that $J^{\alpha^*}$ is a martingale.
	\end{enumerate}
\end{Definition}
\noindent A family of processes with the above properties can now be used to compare the expected utilities of an arbitrary strategy $\alpha \in \cA$ and the strategy $\alpha^*$:
\vspace{-.2cm} 
\begin{equation*}
	\E[ U(X_T^\alpha) ] = \E[ J^\alpha_T ] \leq J^\alpha_0 = J^{\alpha^*}_0 = \E[ J^{\alpha^*}_T]  = \E[ U(X^{\alpha^*}_T)]=\mathcal{V}(x_0,V_0).
\end{equation*}
where $X^{\alpha^*}$ is the wealth process under $\alpha^*$. Thus the strategy $\alpha^*$ is indeed our desired optimal portfolio strategy.

\medskip
\noindent As stated before,  Problem (\ref{eq:value0}) seen as an optimization problem with state process \(X^{\alpha}\) is non-Markovian and the standard stochastic control approach cannot be applied.
Heuristically speaking, the non-Markovian and non-semimartingale
characteristics of the multivariate fake stationary Volterra Heston model (\ref{eq:hestonS})-(\ref{VolSqrt_}) are overcome in the degenerate correlation case by applying the martingale optimality principle
and constructing an ansatz, which is inspired by the martingale distortion transformation and the exponential-affine representation (see e.g. \cite[Theorem A.4.]{Gnabeyeu2026a}) in term of an auxiliary process denoted \(\Gamma\equiv\Gamma\big(g_t(\cdot)\big)\).
In the general correlation case, we provide a verification argument avoiding restrictions on the correlation structure linked to the martingale distortion transformation.

\smallskip
\noindent In fact, \cite{fouque2018aoptimal} showed that if the Sharpe-ratio $\lambda$ in~\eqref{eq:stocks} is bounded and has bounded derivative, then the value process $\mathcal{V}_t$ can be expressed as $\mathcal{V}_t=J_t(X_t^\alpha,g_t(\cdot)\big))$, where 
\begin{equation}\label{eq:ansatz}
	J_t(X_t^{\alpha},g_t(\cdot)\big)):=F(t,X_t^{\alpha})\Gamma(g_t(\cdot)\big))
\end{equation}
for the power utility case (in which \(F(t,X_t^{\alpha})=U(X_t^{\alpha})=\frac{(X_t^{\alpha})^{\gamma}}{\gamma}\)) even if the volatility process $V_t$ is non-Markovian.
This approach is called the martingale distortion transformation and was first introduced in the seminal paper of \cite{zariphopoulou2001solution} and later transferred to a non-Markovian setting with H\"older-type inequalities in \cite{Teh04} for both power utility  and exponential utility case (in which \(F(t,X_t^{\alpha})=- \frac{1}{\gamma} \exp\Big( - \gamma e^{\int^T_t r(s)ds} X_t^\alpha\Big)\)). The extension to the multi-asset case is straight forward in the case of a bounded risk premium and degenerate correlation structure(cf. \cite{fouque2018aoptimal}, Remark 2.5.).

\medskip
\noindent We will provide concrete specifications of the process $\Gamma$ depending on the problem at hand.
The semi-closed form exponential--affine representation~\cite[Theorem A.4.]{Gnabeyeu2026a} of the functional $\Gamma$ can be expressed in terms of the continuous solution $\psi$
of an associated time-inhomogeneous Riccati-type Volterra equation, which can be solved by well-known efficient numerical methods. 
\begin{note}
	It is worth noting that, in contrast to existing literature on utility maximization under Volterra stochastic environments (see, e.g., \cite{HanWong2020b,AichingerDesmettre2021}), our approach presents the stochastic factor, or auxiliary process \(\Gamma\), as the solution to a BSDE. Furthermore, its exponential representation mentioned earlier is expressed in terms of the adjusted forward process, as defined in equations~\eqref{eq:processg}--\eqref{eq:Condprocessg}, rather than the forward process used in these works.
\end{note}

\noindent We will offer explicit solutions to the optimal portfolio policies that depend on the function  $\psi$, solution of the above-mentionned time-dependent multivariate Riccati-Volterra equation.  
Let $\Lambda$ be defined as 
\begin{equation}\label{eq:Lambda}
	\Lambda_t^i =  \nu_i\varsigma^i(t) \psi^i(T-t) \sqrt{V^i_t}, \quad i=1,\ldots,d, \quad 0 \leq t \leq T, 
\end{equation}
We will work under the following assumption,
\vspace{-.2cm} 
\begin{assumption}\label{assm:gen}
	Assume that there exists a solution
	$\psi \in C([0,T],\mathbb{R}^d)$ to the above-mentioned inhomogeneous Riccati--Volterra equation satisfying the below appropriate
	boundedness condition i.e. such that 
	\begin{equation}\label{eq:condtheta}
		\max_{1 \leq i \leq d} \sup_{t \in [0,T]} \left( \theta_i^2 + \nu_i^2 \varsigma^i(t)^2 \psi^i(T-t)^2 \right) \leq \frac{a}{a(p)},
	\end{equation}
	holds for a sufficient large $p > 1$, where the constant $a(p)$ is given by 
	{  \begin{equation}a(p)=\max \Big[p \left(2 + |\Sigma| \right),   {2 (8p^2 {- 2p}) \left( 1  + |{\Sigma}|^2  \right)}, { p \left( 1  + |{\Sigma}|^2  \right)}  \Big]. \label{eq:constap}
	\end{equation}}
	and the constant $a>0$ is such that $\E\left[\exp\big(a\int_0^T \sum_{i=1}^d V^i_s ds\big)\right] < \infty$.
\end{assumption}
\noindent {\bf Remark on Assumption~\ref{assm:gen}:} Note that if Assumption~\ref{assm:gen} hold, then \begin{equation}
	\label{eq:assumption_novikov}
	\E \Big[ \exp\Big( a(p)\int_0^T \big(  |\lambda_s|^2 + \left|\Lambda_s\right|^2 \big)ds \Big) \Big] \;< \;  \infty,
\end{equation}
holds for some $p > 1$ and a constant $a(p)$ given by~\eqref{eq:constap}.

\noindent In fact, under Assumption~\ref{assm:gen}, we will have 
\vspace{-.4cm} 
\begin{equation}
	a(p)\left( |\lambda_s|^2 + \left|\Lambda_s\right|^2 \right) \; = \; a(p) \sum_{i=1}^d V_s^i \left( \theta_i^2 + \nu_i^2 \varsigma^i(s)^2\psi^i(T-s)^2 \right) \; \leq \;  a \sum_{i=1}^d V_s^i,
	\vspace{-.1cm}
\end{equation} 
which implies that $\E \left[ \exp\left(  a(p) \int_0^T \left( |\lambda_s|^2 + \left|\Lambda_s\right|^2  \right)ds\right) \right]< \infty$.


\medskip
\noindent{\bf Remark:} 
Condition \eqref{eq:condtheta} concerns the risk premium constants  $(\theta_1,\ldots, \theta_d)$. For a large enough constant $a>0$, from~\cite[Theorem A.4.]{Gnabeyeu2026a} ( with \(\mathcal{M} \ni m(\dd s) := a{\bold{1}_d}\,\delta_0(\dd s) \)),   a sufficient condition ensuring 
$\E\big[\exp\big(a\int_0^T \sum_{i=1}^d V^i_s ds\big)\big]<\infty$ is the existence of a continuous solution $\tilde{\psi}$  on $[0,T]$ to the inhomogeneous Riccati--Volterra equation 
\vspace{-.2cm} 
\begin{equation}\label{eq:CondExpoRepreA}
	\tilde{\psi}^i(t) = \; \int_0^t K_i(t-s) \Big(a+ \big(D^\T\tilde{\psi}(s)\big)_i + \frac{\nu_i^2}{2}(\varsigma^i(T-s)\tilde{\psi}^i(s)) ^2 \Big) ds, \;\; i=1,\ldots,d, \;\; 0 \leq t \leq T.
\end{equation} 
\noindent 
Power, exponential and logarithmic utility functions are widely adopted in the literature and display distinct risk-aversion properties. Specifically, the power utility function exhibits constant relative risk aversion (CRRA), 
whereas the exponential utility function is characterized by constant absolute risk aversion (CARA). Optimal strategies corresponding to these utility functions are presented in the sequel, within the framework of the model introduced in Section~\ref{Sect:affine}.
\vspace{-.2cm}
\subsection{Optimal strategy for the power utility maximization problem}\label{subsect:Power}
\noindent Here, we denote by  $\pi_t$ the proportion of wealth invested in the risky assets \(S\).
An agent invests at any time $t$ the proportion $\pi_t$ of his wealth $X^\pi_t$ in the stocks \(S\)
of price's vector $S_t$, and the remaining proportion $1-\pi_t^\top {\bold{1}_d}$ in a bond of price $S_t^0$
with interest rate $r(t)$.  The notation $X^\pi_t$ emphasizes the dependence of the wealth on the strategy $\pi = (\pi_t)_{t \ge 0}$.
The portfolio strategy $\pi_t=(\pi_{t}^1,\dots, \pi_{t}^d)$ is an $(\R^d)^*$ valued, progressively measurable process, where $\pi_{t}^{k}$ represents the proportion of wealth invested into stock $k$ at time $t$. 
Assuming the model~\eqref{eq:stocks}--~\eqref{eq:hestonS} for $S_t$, the dynamics of the controlled wealth process is given by
\vspace{-.2cm} 
\begin{align*}
	\mathrm d X^\pi_t
	&= X^\pi_t \bigl(  \pi_t^\top \big( r(t) {\bold{1}_d} + \sigma(V_t) \lambda_t  \big) + (1 - \pi_t^\top {\bold{1}_d})r(t) \bigr)\,\mathrm dt + X^\pi_t \pi_t^\top\sigma(V_t)  \,\mathrm d B_t \\
	&= X^\pi_t \bigl(r(t) + \pi_t^\top  \sigma(V_t) \lambda_t\bigr)\,\mathrm dt + X^\pi_t \pi_t^\top   \sigma(V_t) \,\mathrm d B_t=X_t^{\pi}(r(t)+\pi_t^{\top}{\diag(V_t)}{\theta})dt+X_t\pi_t^{\top}\sqrt{\diag(V_t)}dB_{t}.
\end{align*}
where ${\theta}=(\theta_1,\dots,\theta_d)^\top$.
The investor faces the portfolio constraint that, at any time $t$,
$\pi_t$ takes values in a closed convex subset  $\mathcal{A}\subset \mathbb R^d$ denoting the set of admissible portfolio strategies. 

\medskip
\noindent From now, we let $\alpha_t := \sigma(V_t)^\top \pi_t$ be the investment strategy. Then, under a fixed portfolio strategy, the wealth process $X^\alpha_t$, controlled by $\alpha$ is governed by:
\vspace{-.2cm} 
\begin{equation}\label{Eq:wealth}
	d X^\alpha_t = \big(r(t) + \alpha_t^\top \lambda_t \big)  X^\alpha_t dt + \alpha_t^\top X^\alpha_t dB_{t}, \; X_0 = x_0 > 0.
	\vspace{-.1cm}
\end{equation}
By solving the linear SDE \eqref{Eq:wealth}, the wealth process admits the explicit representation
\begin{equation}\label{eq:wealthPowerSol}
	X_t^\alpha
	=
	x_0 \exp\!\left(
	\int_0^t \Big(r(s) + \alpha_s^\top \lambda_s - \tfrac12 \left|\alpha_s\right|^2\Big)\,ds
	+
	\int_0^t \alpha_s^\top \, dB_{s}
	\right), \quad x_0\geq0.
\end{equation}
By a solution to~\eqref{Eq:wealth}, we mean an $\F$-adapted process $X^{\alpha}$ satisfying~\eqref{Eq:wealth} on $[0,T]$  with $\P$-a.s. continuous sample paths and such that
\vspace{-.2cm} 
\begin{align}
	\label{eq:estimateXPower}
	\E\big[\sup_{t\leq T} |X^{\alpha}_t|^p \big] &< \;  \infty\quad \text{for some }\quad p>1.
\end{align}
The conditions under which we consider a strategy to be admissible is specified in the below definition. 
\vspace{-.2cm}
\begin{Definition}\label{Def:adm}
	In the setting described above, we say that an investment strategy $\alpha(\cdot)$ is admissible if
	\begin{enumerate}
		\item[$(a)$] The SDE (\ref{Eq:wealth}) for the wealth process $(X_t^{\alpha})$ has a unique solution in terms of $(S,V,B)$; with $\p$-$\as$ continuous paths and $X_t \geq 0$,  $\forall \; t \in [0, T]$, $\p$-$\as$;
		\item[$(b)$] $\E[\frac{1}{\gamma}(X_T^{\alpha})^{\gamma}]<\infty$ for all $0<\gamma<1$;
		\item[$(c)$] $\alpha(\cdot)$ is $\F$-adapted and $\int^t_0 \left|\alpha_s \right|^2ds < \infty$, $\forall \; t \in [0, T]$,  $\p$-$\as$.
	\end{enumerate}
	The set of all admissible investment strategies is denoted as $\cA$ and is naturally defined by
	\vspace{-.2cm} 
	$$\mathcal A  = \left\{ \alpha \in {L^{2, loc}_{\F}([0,T], \R^d)} \mbox{ such that \eqref{Eq:wealth} has a  solution satisfying } ~\eqref{eq:estimateXPower}. 
	\right\}$$
\end{Definition}
\noindent We want to solve the Merton power utility optimization problem, i.e. our aim is to find the value function $\mathcal{V}(x_0,V_0)$ for the CRRA utility function such that
\vspace{-.2cm} 
\begin{equation}\label{obj_power}
	\mathcal{V}(x_0,V_0)=	\sup_{ \alpha(\cdot) \in \cA }\E_{x_0,V_0}[\frac{1}{\gamma}(X_T^{\alpha})^{\gamma}];\quad 0<\gamma<1,
\end{equation}
where $\E_{x_0,V_0}$ is the conditional expectation given $x_0$ and $V_0$. The parameter $\gamma$ represents the relative
risk aversion of the investor. Smaller $\gamma$ correspond to higher risk aversion. 
\vspace{-.4cm}
\subsubsection{The degenerate correlation case}\label{degenerate}
\noindent We assume that the correlation in ~\eqref{eq:correstructureheston} is of the form $(\rho,\dots, \rho)$ for $\rho\in [-1,1]$. To construct the family of processes $\{ J^\alpha_t \}_{ t \in [0, T]}$, $\alpha \in \cA$, satisfying conditions~$(1)$-~$(2)$-~$(3)$ in Definition~\ref{def:Martopt}, we introduce the new probability measure  $\tilde{\mathbb{P}}$ defined via the Radon-Nikodym density at \(\mathcal{F}_T\) from
\vspace{-.2cm} 
\[
\frac{d\tilde{\mathbb{P}}}{d\mathbb{P}}|_{\mathcal{F}_t}= \mathcal E\Big( \frac{\gamma}{1-\gamma}\int_0^t \sum_{i=1}^d\theta_i \sqrt{V^i_s}dB^i_s \Big) =\operatorname{exp}\Big(\frac{\gamma}{1-\gamma}\int_0^t \lambda_s^\top dB_{s}-\frac{\gamma^2}{2(1-\gamma)^2}\int_0^t \left| \lambda_s \right|^2 ds\Big)
\]
where the stochastic exponential is a true martingale by Lemma~\ref{lm: extended_m_AJ_lemma}
together with the new standard brownian motion under $\tilde{\mathbb{P}}$,  \(\widetilde{B}_{t}=B_{t}-\frac{\gamma}{1-\gamma}\int_0^t\lambda_s ds\).
Define the new process \(\widetilde{W}\) by
\vspace{-.3cm} 
\[\widetilde{W}_t =  \Sigma  \widetilde{B}_t + \sqrt{I-\Sigma^\T \Sigma} B^{\perp}_t = W_{t}-\frac{\gamma}{1-\gamma}\int_0^t\Sigma\lambda_s ds,\]
\vspace{-.1cm} 
Notice that, by the 
Girsanov theorem, \(\widetilde{B}\) and \(\widetilde{W}\) are standard Wiener processes under the measure $\tilde{\mathbb{P}}$.
As in the one dimensional case, the Ansatz (\(\delta\) is called the distorsion coefficient)
\vspace{-.2cm}  
\begin{equation}\label{eq:ansatz1}
J_t^{\alpha}=\frac{(X_t^{\alpha})^{\gamma}}{\gamma}\Big(\E^{\tilde{\mathbb{P}}}\Big[\operatorname{exp}\Big(\int_t^T \frac{\gamma}{\delta} \Big(r_s+\frac{\left| \lambda_s \right|^2}{2(1-\gamma)}\Big)ds\Big)|\mathcal{F}_t\Big]\Big)^\delta=:\frac{(X_t^{\alpha})^{\gamma}}{\gamma}\Gamma_t \; \text{where} \; \delta:=\frac{1-\gamma}{1-\gamma+\gamma\rho^2}.
\end{equation}
is inspired by the martingale distortion transformation which was first introduced in the seminal paper \cite{zariphopoulou2001solution} and later transferred to a non-Markovian setting in \cite{Teh04,fouque2018aoptimal}.
Here we use the short notation $J_t^{\alpha}$ for $J_t(X_t^{\alpha},V_t)$.

\smallskip
\noindent The main properties of $\Gamma_t$ are summarized in Proposition \ref{prop:ExpoPower_riccati_1} below where \(\tilde{g}_t(s)\) denotes the conditional $\tilde{\mathbb{P}}$-expected ajusted variance process. 
\vspace{-.2cm} 
\begin{Proposition}
	\label{prop:ExpoPower_riccati_1}
	Assume that there exists  a solution $\psi \in C([0,T],\R^d)$ to the inhomogeneous Riccati-Volterra equation \eqref{eq:RiccatiPowerpsi1}-\eqref{eq:RiccatiPowerpsi2} below:
	\vspace{-.3cm} 
	\begin{align}
		\psi^i(t)&= \int_0^t K_i(t-s)\big(\frac{\gamma \theta_i^2}{2\delta(1-\gamma)}  + F_i(T-s,\psi(s))\big) ds,  \label{eq:RiccatiPowerpsi1} \\
		F_i(s,\psi) &=  \frac{\gamma}{1-\gamma} \rho \theta_i \nu_i \varsigma^i(s) \psi^i + (D^\top \psi)_i + \frac {\nu_i^2} 2  (\varsigma^i(s)\psi^i)^2, \quad i=1,\ldots,d, \label{eq:RiccatiPowerpsi2}
	\end{align} 
	\vspace{-.4cm} 
	Then, 
	\vspace{-.3cm} 
	\begin{equation}\label{eq:GammaPower}
		\Gamma_t = \exp\Big( \gamma\int_t^T r(s) ds +  \sum_{i=1}^d\int_t^T  \big(\frac{\gamma \theta_i^2}{2(1-\gamma)}  + \delta F_i(s,\psi(T-s))\big) \tilde{g}^i_t(s) ds \Big).
	\end{equation}
	where $\tilde{g}=(\tilde{g}^1,\ldots,\tilde{g}^d)^\T$ is the $\R^d$-valued process \((\tilde{g}_t(s))_{t\leq s}\) denoting the adjusted conditional $\tilde{\mathbb{P}}$-expected variance.
	Moreover $\left(\Gamma, \Lambda\right)\in\mathbb{S}^{p}_{\F}([0,T], \R_+^*) \times L^2_{\F}([0,T], \R^d) $  for some sufficiently large $ p >  1 \wedge \frac{1-\gamma}{\gamma} $ and satisfy the following Riccati BSDE
	\vspace{-.3cm} 
	\begin{align}
		\frac{d\Gamma_t}{\Gamma_t} &= \Big(- \gamma r(t)  -\frac{\gamma}{2(1-\gamma)}\sum_{i=1}^d \Big( \theta_i^2  + \delta^2 \rho^2 \nu_i^2 (\varsigma^i(t)\psi^i(T-t))^2 \Big) V^i_t\Big)dt  + \;  \delta \sum_{i=1}^d \psi^i(T-t)\nu_i \varsigma^i(t) \sqrt{V^i_t}d\widetilde{W}^i_t\notag \\
		& =  \big(- \gamma r(t)  - \frac{\gamma}{2(1-\gamma)} \left| \lambda_t\right|^2 -\frac{\gamma}{2(1-\gamma)} \delta^2\left| \Sigma \Lambda_t \right|^2\big)dt +\delta \Lambda_t^\top d\widetilde{W}_t, \quad \Gamma_T=1. \label{eq:gamma_Powerheston}
	\end{align}
\end{Proposition}
\noindent Note that Equation~\eqref{eq:gamma_Powerheston} is known in the litterature as a Ricatti backward stochastic differential equation (see e.g.~\cite[Theorem 3.1]{AbiJaberMillerPham2021} or ~\cite[Theorem 3.1]{ChiuWong2014}, upon setting \(\Tilde{\Lambda}_t= \Gamma_t\Lambda_t\) ).

\smallskip
\noindent Existence and uniqueness of the solution to ~\eqref{eq:RiccatiPowerpsi1}-\eqref{eq:RiccatiPowerpsi2} are established in \cite[Theorem A.2.]{Gnabeyeu2026a} based on the results of \cite{EGnabeyeuR2025} (see also the Remark on Proposition~\ref{prop:ExpoPower_riccati_2}).

\noindent By considering $\Gamma_t$, we overcome the non-Markovian and non-semimartingale difficulty in the variance process (\ref{VolSqrt_}).

\noindent {\bf Remark:} One should note that the Ansatz or the constructed family of stochastic processes \((J^\alpha)_{\alpha}\) in~\eqref{eq:ansatz1} can be rewritten for every \(t\in [0,T]\) as
\vspace{-.3cm} 
\begin{equation}
	J^\alpha_t = \frac{x^\gamma_0}{\gamma} \exp\!\left(
	\int_0^t \gamma\Big(r(s) + \alpha_s^\top \lambda_s - \tfrac12 \left|\alpha_s\right|^2\Big)\,ds
	+
	\int_0^t \gamma\alpha_s^\top \, dB_{s} + Y_t
	\right), \quad x_0 \geq 0.
\end{equation}
where the pair \((Y, \Lambda)\) satisfies a rather simple backward stochastic differential equation (BSDE) under \(\tilde{\P}\) with the final condition \(Y_T = 0\), so that \(\Gamma_t := \exp\big(Y_t\big)\) satisfies the BSDE~\eqref{eq:gamma_Powerheston}. This approach is analogous to that in~\cite[Equation~13]{HuImkellerMueller2005}.
The main result we present in this case is as follows:

\vspace{-.2cm} 
\begin{Theorem}\label{Thm:ExpoPower_riccati_1}
	Let \( \psi \) be the unique, continuous non-continuable solution of the inhomogeneous Ricatti-Volterra equation  	~\eqref{eq:RiccatiPowerpsi1}-~\eqref{eq:RiccatiPowerpsi2}
	on the interval \( [0, T_{\text{max}}) \) so that Assumption~\ref{assm:gen} is in force. Then \( 	J^\alpha_t = \frac{(X_t^{\alpha})^{\gamma}}{\gamma}\Gamma_t \) satisfies the martingale optimality principle for \( t \in [0, T] \), \( T < T_{\text{max}} \), and the optimal portfolio strategy \( \alpha^* \) is given by
	\vspace{-.3cm} 
	\begin{align}
		\label{Eq:alpha_power*1}
		\c^*_t &= \frac{1}{1 - \gamma} \left( \lambda_t + \delta \Sigma \Lambda_t \right) = \frac{1}{1 - \gamma} \sqrt{\diag(V_t)} \left( \theta + \delta \Sigma \nu \varsigma(t) \psi(T - t) \right), \quad 0 \leq t \leq T \\
		&= \;  \Big(\frac{1}{1 - \gamma}  \big(\theta_i + \delta \rho\nu_i\varsigma^i(t) \psi^i(T-t) \big) \sqrt{V_t^i}   \Big)_{1 \leq i \leq d}, \quad 0 \leq t \leq T.\label{Eq:alpha_power*2}
	\end{align}
	\vspace{-.3cm} 
	Moreover, \(X^{\alpha^*}\) satisfies~\eqref{eq:estimateXPower}
	and $\alpha^*$ is admissible. 
\end{Theorem}
\vspace{-.4cm} 
\subsubsection{The general correlation case: A verification argument}\label{general}
\noindent The martingale distortion approach, used by \cite{HanWong2020b} in the one dimensional case
can only be applied to the multivariate setting if the correlation structure is highly degenerate, i.e. \( \rho_1 = \ldots = \rho_d.\) (see also \cite{fouque2018aoptimal, AichingerDesmettre2021}).
For the case where the correlation in ~\eqref{eq:correstructureheston} is given by an arbitrary vector $(\rho_1,\dots,\rho_d)\in[-1,1]^d$, the martingale distortion arguments from the previous section do not work anymore. However, we remark that when \( \rho_1 = \ldots = \rho_d=\rho\), if $\psi$ is the unique global solution of the Riccati--Volterra equation~\eqref{eq:RiccatiPowerpsi1}--~\eqref{eq:RiccatiPowerpsi2}, then  setting $\tilde{\psi} = \delta\,\psi$ and $\theta=\diag{(\theta_1,\dots,\theta_d)}$, we have using \(\delta=\frac{1-\gamma}{1-\gamma+\gamma\rho^2}\) 
\vspace{-.3cm} 
\begin{align*}
	\delta F(s,\psi) &=  \frac{\gamma}{1-\gamma} \nu \varsigma(s) \Sigma\theta (\delta \psi) + D^\top (\delta\psi) + \frac {1}{2\delta} \Big(\nu\varsigma(s)(\delta\psi)\Big)^2\\
	&= \frac{\gamma}{1-\gamma} \nu \varsigma(s) \Sigma\theta (\delta \psi) + D^\top (\delta\psi) + \frac {1}{2}\left[\Big(\nu\varsigma(s)(\delta\psi)\Big)^2 +  \frac{\gamma}{1-\gamma} \Big(\nu\varsigma(s)\Sigma(\delta\psi)\Big)^2 \right] = \tilde{F}(s,\tilde{\psi})
\end{align*}
where \(\tilde{F}(s,\tilde{\psi}):= \frac{\gamma}{1-\gamma} \nu \varsigma(s) \Sigma\theta \tilde{\psi} + D^\top \tilde{\psi} + \frac {1}{2}\left[(\nu\varsigma(s)\tilde{\psi})^2 +  \frac{\gamma}{1-\gamma} (\nu\varsigma(s)\Sigma\tilde{\psi})^2 \right]\) so that $\tilde{\psi}\in C([0,T],\R^d)$ is the unique global solution of the inhomogeneous Riccati--Volterra equation
\begin{align}
	\tilde{\psi}^i(t)&= \int_0^t K_i(t-s)\big(\frac{\gamma \theta_i^2}{2(1-\gamma)}  + \tilde{F}_i(T-s,\tilde{\psi}(s))\big) ds,  \label{eq:RiccatiPowerTilpsi1} \\
	\tilde{F}_i(s,\tilde\psi) &=  \frac{\gamma}{1-\gamma} \theta_i \rho_i \nu_i \varsigma^i(s) \tilde{\psi}^i + (D^\top \tilde{\psi})_i + \frac {\nu_i^2} 2 \left[  (\varsigma^i(s)\tilde{\psi}^i)^2+\frac{\gamma}{1-\gamma}\rho_i^2(\varsigma^i(s)\tilde{\psi}^i)^2\right], \quad i=1,\ldots,d, \label{eq:RiccatiPowerTilpsi2}
\end{align} 
Consequently, to avoid restrictions on the correlation structure linked to the martingale distortion transformation, we  use a verification arguments in the spirit of \cite{BauerleLi2013} to solve the optimization problem, thus extending the results obtained in the preveous section to the more general correlation structure. 
\begin{Proposition}
	\label{prop:ExpoPower_riccati_2}
	Assume that there exists  a solution $\psi \in C([0,T],\R^d)$ to the inhomogeneous Riccati-Volterra equation:
	\vspace{-.3cm} 
	\begin{align}
		\qquad\; \psi^i(t)&= \int_0^t K_i(t-s)\big(\frac{\gamma \theta_i^2}{2(1-\gamma)}  + F_i(T-s,\psi(s))\big) ds, 
		  \label{eq:RiccatiPowerTilpsi1} \\
		F_i(s,\psi) &=  \frac{\gamma}{1-\gamma} \theta_i \rho_i \nu_i \varsigma^i(s) \psi^i + (D^\top \psi)_i + \frac {\nu_i^2} 2 \left[  (\varsigma^i(s)\psi^i)^2+\frac{\gamma}{1-\gamma}\rho_i^2(\varsigma^i(s)\psi^i)^2\right], \; i=1,\ldots,d.
		\label{eq:RiccatiPowerTilpsi2}
	\end{align} 
	Let $\left(\Gamma, \Lambda\right)$ be defined as 
	\begin{equation}\label{eq:GammaPower2}
		\left\{
		\begin{array}{ccl}
			\Gamma_t &=& \exp\Big( \gamma\int_t^T r(s) ds +  \sum_{i=1}^d\int_t^T  \big(\frac{\gamma \theta_i^2}{2(1-\gamma)}  + F_i(s,\psi(T-s))\big) g^i_t(s) ds \Big), \\
			\Lambda_t^i &=&  \nu_i\varsigma^i(t) \psi^i(T-t) \sqrt{V^i_t}, \quad i=1,\ldots,d, \quad 0 \leq t \leq T, 
		\end{array}  
		\right.
	\end{equation}
	where $g$ $=$ $(g^1,\ldots,g^d)^\T$ is given by \eqref{eq:processg} i.e. the $\R^d$-valued process \((g_t(s))_{t\leq s}\)  is defined in~\eqref{eq:processg}.  Then, for some $ p >  1 \wedge \frac{1-\gamma}{\gamma} $, $\left(\Gamma,\Lambda\right)$  is a  $\mathbb{S}^{p}_{\F}([0,T], \R_+^*) \times L^2_{\F}([0,T], \R^d)$-valued solution to the Riccati BSDE~\eqref{eq:gamma_PowerTilheston}
	\begin{equation}\label{eq:gamma_PowerTilheston}
		\left\{
		\begin{array}{ccl}
			\frac{d\Gamma_t}{\Gamma_t} &=&  \big(- \gamma r(t)  - \frac{\gamma}{2(1-\gamma)} \left| \lambda_t+ \Sigma \Lambda_t \right|^2\big)dt + \Lambda_t^\top dW_t, \\
			\Gamma_T &=&  1. 
		\end{array}  
		\right.
	\end{equation}
\end{Proposition}
\noindent {\bf Proof: } The proof that $\left(\Gamma, \Lambda\right)$ satisfy~\eqref{eq:gamma_PowerTilheston} is a straightforward adaptation of the arguments from the proof of Proposition~\ref{prop:ExpoPower_riccati_1}.
{It remains to show that $\Gamma_t >0$ for all $ t \in [0, T]$, $\p$-$\as$ and $\left(\Gamma, \Lambda\right) \in \mathbb{S}^{p}_{\F}([0,T],\R) \times L^2_{\F}([0,T], \R^d)$} for some sufficiently large \(p>1\).
An application of It\^o's formula to the process \[\bar{M}_t := \Gamma_t \exp\Big(\int_0^t  \big(\gamma r(s) + \frac{\gamma}{2(1-\gamma)} \left| \lambda_s+ \Sigma \Lambda_s \right|^2\big) ds\Big), \; t \leq T,\]
combined with the dynamics ~\eqref{eq:gamma_PowerTilheston} shows that $d\bar{M}_t = \bar{M}_t  \Lambda_t^\top dW_t$, so that $\bar{M}$ is a local martingale of the form \(\bar{M}_t =\;\mathcal E\Big( \int_0^t \Lambda_s^\top dW_s\Big)=\; \mathcal E\Big( \int_0^t \sum_{i=1}^d \nu_i\varsigma^i(s)  \psi^i(T-s)\sqrt{V^i_s}dW^i_s \Big).\)
\noindent Since $\psi$ is continuous, it is bounded; likewise, $\varsigma$ is bounded. Therefore a straightforward application of Lemma~\ref{lm: extended_m_AJ_lemma} with $g_{2} = 0$ and $g_{1,i}(s) = \nu_i \varsigma^i(s)\psi^i (T-s) \in L^{\infty}([0,T],\R)$,  recall~\eqref{eq:moments V1}, yields that the stochastic exponential \(\bar{M}\) is a true $\P$- martingale. Now, as $\Gamma_T=1$, writing $\E[\bar{M}_T|\mathcal F_t]=\bar{M}_t$, we obtain
\begin{align}
	\Gamma_t &= \E \Big[ \exp\Big(\int_t^T  \big(\gamma r(s) + \frac{\gamma}{2(1-\gamma)} \left| \lambda_s+ \Sigma \Lambda_s \right|^2\big) ds\Big) \mid \mathcal F_t\Big], \quad t\leq T \label{eq:Power_ito_gamma_all}\\
	&\leq \exp\Big(\int_0^T \gamma r(s) ds \Big) M_t, \;\text{with}\; M_t:= \E \Big[ \exp\Big(\frac{\gamma}{2(1-\gamma)}\int_0^T \left| \lambda_s+ \Sigma \Lambda_s \right|^2ds\Big) \mid \mathcal F_t\Big], \quad t\leq T\nonumber
\end{align}
We then have in view of~\eqref{eq:Power_ito_gamma_all} that there exists some positive constant \(m>0\) such that $\Gamma_t \geq m > 0$ for every \(t\in [0,T]\).
Using the elementary inequality \((a+b)^p \le 2^{(p-1)^+}\big(a^p + b^p\big)\) for \(a,b>0\), one gets
\begin{equation*}
	|\lambda_s +  \Sigma\Lambda_s|^2 \leq 2(|\lambda_s|^2 +  |\Sigma\Lambda_s|^2) \leq 2(1 + |\Sigma|^2)(|\lambda_s|^2 +  |\Lambda_s|^2) \leq 2\beta(1 + |\Sigma|^2)(|\lambda_s|^2 +  |\Lambda_s|^2).
\end{equation*}
where \(\beta>1\) is choosen such that \(\frac{\beta\gamma}{(1-\gamma)}>1\). Consequently, by simple calculation:
\begin{align*}
	\E \Big[ \exp\Big(\frac{\gamma}{2(1-\gamma)}\int_0^T \left| \lambda_s+ \Sigma \Lambda_s \right|^2ds\Big) \Big] &\leq \E \Big[ \exp\Big(\frac{\beta\gamma}{(1-\gamma)}(1 + |\Sigma|^2)\int_0^T \big(|\lambda_s|^2 +  |\Lambda_s|^2\big)ds\Big) \Big]<\infty
\end{align*}
thanks to the Novikov-type condition~\eqref{eq:assumption_novikov} with constant \(a(\frac{\beta\gamma}{(1-\gamma)}) = \frac{\beta\gamma}{(1-\gamma)}(1 + |\Sigma|^2)\).  Therefore, $M_t$ is a martingale under $\P$.  
Let $ p > 1 \wedge \frac{1-\gamma}{\gamma} $.
By {virtue of Doob's maximal inequality}, in the second line,

\begin{align*}
	&\, \E \Big[ \sup_{ t \in [0, T]} \big| \Gamma_t \big|^p \Big] \leq e^{ p \gamma\int_0^Tr(s) ds} \E \Big[ \sup_{ t \in [0, T]} \big| M_t \big|^{p}\Big] \\
	& \hspace{2.5cm}\leq e^{ p \gamma\int_0^Tr(s) ds} \Big(\frac{p}{p-1}\Big)^{p} \E \Big[ \exp\Big(\frac{p\gamma}{2(1-\gamma)}\int_0^T \left| \lambda_s+ \Sigma \Lambda_s \right|^2ds\Big)\Big]\\
	& \hspace{2.5cm}\leq \Big(\frac{p}{p-1}\Big)^{p} e^{ p \gamma\int_0^Tr(s) ds} \E \left[ \exp\left({a(\frac{p\gamma}{(1-\gamma)}) \int_0^T \left( |\lambda_s|^2 + |\Lambda_s|^2\right)  ds}\right)    \right]< \infty
\end{align*}
where we used condition~\eqref{eq:assumption_novikov}
with \(a(\frac{p\gamma}{(1-\gamma)}) =\frac{p\gamma}{(1-\gamma)}(1 + |\Sigma|^2)  \).  Therefore, $\E\big[ \sup_{ t \in [0, T]} |\Gamma_t|^p \big] < \infty $ holds.
As for $\Lambda$, it is clear that it belongs to $L^2_{\F}([0,T], \R^d)$ since $\varsigma$ and $\psi$ are bounded, $\psi$ is continuous thus bounded and $\E \Big[\int_0^T  \sum_{i=1}^d V^i_s ds \Big] <  \infty$ by \eqref{eq:moments V1}. This complete the Proof\hfill $\Box$

\medskip
\noindent First of all, we point out that for each of the inhomogeneous Riccati-Volterra equations ~\eqref{eq:RiccatiPowerTilpsi1}-~\eqref{eq:RiccatiPowerTilpsi2} and ~\eqref{eq:RiccatiPowerpsi1}-~\eqref{eq:RiccatiPowerpsi2}, there exists a unique non-continuable continuous solution $(\psi, T_{\operatorname{max}})$ under Assumption~\ref{assump:kernelVolterra} by~\cite[Theorem A.2.]{Gnabeyeu2026a} as stated in the following Remark.

\medskip
\noindent{\bf Remark:} Fix \(T>0\)
and assume that $K$ satisfies Assumption~\ref{assump:kernelVolterra}. As the matrix $D$ in the drift of the volatility process is a diagonal matrix, i.e. $D=-\diag{(\lambda_1,\dots, \lambda_d)}$, for \(i=1,\ldots,d,\) $\psi^i$ satisfies the Volterra equation
\vspace{-.3cm} 
\begin{equation}\label{eq:comp_}
	\chi(t) = \int_0^t K_i(t-s)  \Big( \frac{\gamma \theta_i^2}{2(1-\gamma)} - \big(\lambda_i -\frac{\gamma}{1-\gamma}\theta_i \rho_i \nu_i \varsigma^i(T-s)\big) \chi(s)  + \frac {\nu_i^2} 2  \frac{1-\gamma+\gamma\rho_i^2}{1-\gamma}\varsigma^i(T-s)^2\chi(s) ^2\Big)ds, \; t\leq T.
\end{equation}
\cite[Theorem A.2.]{Gnabeyeu2026a} guarantees that there exists a unique non-continuable continuous solution \((\psi, T_{\max}^i)\) to Equation~\eqref{eq:comp_} with $\psi^i \in C([0,T_{max}^i),\R)$ in the sense that $\psi^i$ satisfies~\eqref{eq:comp_} on $[0,T_{max}^i)$ with $T_{max}^i \in (0,T]$ and $\sup_{t<T_{max}^i}|\psi^i( t)| = +\infty$, if $T_{max}^i<T$. 
Combining the component-wise solutions, we finally obtain the existence of a unique non-continuable solution $\psi \in C([0,T_{max}),\R^d)$ of the inhomogeneous Ricatti--Volterra Equation~\eqref{eq:RiccatiPowerTilpsi1}-~\eqref{eq:RiccatiPowerTilpsi2} (and in particular of ~\eqref{eq:RiccatiPowerpsi1}--~\eqref{eq:RiccatiPowerpsi2} ). 

\medskip
\noindent  Note that our candidate for the optimal portfolio strategy $\alpha^*$ is given, using~\eqref{Eq:alpha_power*1}--~\eqref{Eq:alpha_power*2}, by
\vspace{-.1cm} 
\begin{equation}\label{eq:optcandPower}
	\alpha_t^*=\frac{1}{1 - \gamma} \left( \lambda_t + \Sigma \Lambda_t \right) = \;  \Big(\frac{1}{1 - \gamma}  \big(\theta_i + \rho_i\nu_i\varsigma^i(t) \psi^i(T-t) \big) \sqrt{V_t^i}   \Big)_{1 \leq i \leq d}, \; 0 \leq t \leq T.
\end{equation}
Since owing to Lemma~\ref{lm: extended_m_AJ_lemma}, \(\mathcal E\big(\int_0^t\alpha_s^{*\T} dB_s\big)\) is a true exponential martingale, we redefine the set of admissible portfolio strategy in this \textit{non-degenerate setting} by
    \begin{equation}\label{eq:StrongAdmStratPower}
	\mathcal A  = \left\{ 
	\begin{aligned}
		& \alpha:=(\alpha_t)_{t \in [0,T]}, \; \F-\text{prog.~measurable such that~} \E\big[\exp\big(\int_0^T \left|\alpha_s\right|^2 ds\big)\big] < \infty\\
		& \mbox{ and Equation~\eqref{Eq:wealth} has a positive solution satisfying } ~\eqref{eq:estimateXPower} 
	\end{aligned}
	\right\}
\end{equation}
\noindent {\bf Remark:}
In particular, \(\alpha \in {L^{2}_{\F}([0,T], \R^d)}\) and recalling that \(\alpha_t := \sigma(V_t)^\top \pi_t = \sqrt{\diag(V_t)}^\top \pi_t\) where $(\pi_t)_{t\in[0,T]}$ is a deterministic $\mathbb{R}^d$-valued process representing a \textit{proportion} of wealth, we can assume that 
\((\pi_t)_{t \in [0,T]}\) is bounded by some constant \(\pi^{\max} \in \mathbb{R}^d\) so that
a sufficient condition for the exponential integrability in~\eqref{eq:StrongAdmStratPower} is that Equation~\eqref{eq:CondExpoRepreA} admits a $C([0,T],\mathbb{R}^d)$ solution for a sufficiently large constant $a\geq\|\pi^{\max}\|:=\|\pi\|_{\infty,T}$

\smallskip
\noindent We will directly show that the value attained by the portfolio strategy~\eqref{eq:optcandPower} is \(\frac{x_0^\gamma}{\gamma}\Gamma_0\), and that no other admissible portfolio strategy can yield a larger value, through the following verification result:
\begin{Theorem}\label{Thm:powerUtilityGeneral}
	Let \( \psi \) be the unique, continuous non-continuable solution of the inhomogeneous Ricatti-Volterra equation  	~\eqref{eq:RiccatiPowerTilpsi1}-~\eqref{eq:RiccatiPowerTilpsi2}
	on the interval \( [0, T_{\text{max}}) \) so that Assumption~\ref{assm:gen} is in force. 
	Then for $t\in [0,T]$, $T< T_{\operatorname{max}}$, an optimal investment strategy $(\alpha_t^*)_{t\in [0,T]}$ for the Merton portfolio problem~\eqref{obj_power} is given by 
	\vspace{-.3cm} 
	\begin{align}
		\label{Eq:alpha_Generalpower*1}
		\c^*_t &= \frac{1}{1 - \gamma} \left( \lambda_t +  \Sigma \Lambda_t \right) = \frac{1}{1 - \gamma} \sqrt{\diag(V_t)} \left( \theta + \Sigma \nu \varsigma(t) \psi(T - t) \right) , \quad 0 \leq t \leq T \\
		&= \;  \Big(\frac{1}{1 - \gamma}  \big(\theta_i +  \rho_i\nu_i\varsigma^i(t) \psi^i(T-t) \big) \sqrt{V_t^i}   \Big)_{1 \leq i \leq d}, \quad 0 \leq t \leq T.\label{Eq:alpha_Generalpower*2}
	\end{align}
	Moreover,
	\vspace{-.3cm} 
	\begin{equation}\label{eq:boundPower}
		\E\Big[ \sup_{ t \in [0, T]} |X^{\alpha^*}_t|^p \Big] < \infty,\quad \text{for some sufficiently large}\quad p>1.
	\end{equation}
	and $\alpha^*$ is admissible. The value function defined in~\eqref{obj_power} can be written as
	\vspace{-.3cm} 
	\begin{equation*}
		\sup_{\alpha \in \mathcal{A}}
		\mathbb{E}_{x_0,V_0}
		\left[
		\frac{\big(X_T^{\alpha^*}\big)^{\gamma}}{\gamma}
		\right]
		= \mathcal{V}(x_0,V_0)=\frac{x_0^\gamma}{\gamma}\exp\Big( \gamma\int_0^T r(s) ds +  \sum_{i=1}^d\int_0^T  \big(\frac{\gamma \theta_i^2}{2(1-\gamma)}  + F_i(s,\psi(T-s))\big) g^i_0(s) ds \Big).
	\end{equation*}
\end{Theorem}
\noindent For the sake of space limitation, but without loss of self-containedness, we present a sketch of the proof. 

\smallskip
\noindent \emph{Sketch of Proof:}
In order to prove that $\alpha^*$ is indeed the optimal portfolio strategy, we show that for
\vspace{-.2cm}  
\[
G(x_0,V_0):=\frac{x_0^\gamma}{\gamma}\Gamma_0=\frac{x_0^\gamma}{\gamma}\exp\Big( \gamma\int_0^T r(s) ds +  \sum_{i=1}^d\int_0^T  \big(\frac{\gamma \theta_i^2}{2(1-\gamma)}  + F_i(s,\psi(T-s))\big) g^i_0(s) ds \Big), 
\]
we have
\vspace{-.1cm} 
\begin{enumerate}
	\item $\E_{x_0,V_0}{\big[\frac{(X_T^{\alpha^*})^{\gamma}}{\gamma}\big]}=G(x_0,V_0)$ for $\alpha_t^*=\frac{\lambda_t + \Sigma \Lambda_t}{1 - \gamma} =  \big(\frac{1}{1 - \gamma}  \big(\theta_i + \rho_i\nu_i\varsigma^i(t) \psi^i(T-t) \big) \sqrt{V_t^i}   \big)_{1 \leq i \leq d}, \; t\in [0,T]$,\label{a1}
	\item $\E_{x_0,V_0}{\big[\frac{(X_T^{\alpha})^{\gamma}}{\gamma}\big]}\leq G(x_0,V_0)$ for every other admissible strategy \(\alpha\in \mathcal{A}\).\label{b1}
\end{enumerate}
In fact, inserting the candidate $\alpha^*$ for the
optimal portfolio strategy in the SDE of the wealth equation show by~$(1)$ that the upper bound \(G\) can be obtained. Consequently  conditions~$(1)$ and~$(2)$ ensure that \(G\) is the value function of the problem~\eqref{obj_power} and $\alpha^*$ is the optimal portfolio strategy.


\vspace{-.2cm} 
\subsection{Optimal strategy for the exponential utility maximization problem}\label{subsect:Expo}
\noindent In this subsection, we consider the exponential utility case.
With a slightly different formulation, let $\pi_t$ denote the vector of the amounts invested in the risky assets $S$ at time $t$ in a self--financing strategy. We assume that the the process \((\pi_t)_{t\geq0}\) are progressively measurable. Then, the dynamics of the wealth $X^{\pi}$ of the portfolio we seek to optimize is given by 
\vspace{-.1cm} 
\begin{align*}
	\mathrm d X^\pi_t
	&=  \bigl(  \pi_t^\top \big( r(t) {\bold{1}_d} + \sigma(V_t) \lambda_t  \big) + (X^\pi_t - \pi_t^\top {\bold{1}_d})r(t) \bigr)\,\mathrm dt + \pi_t^\top\sigma(V_t)  \,\mathrm d B_t \\
	&= X^\pi_t \bigl(r(t) + \pi_t^\top  \sigma(V_t) \lambda_t\bigr)\,\mathrm dt + \pi_t^\top   \sigma(V_t) \,\mathrm d B_t=X_t^{\pi}(r(t)+\pi_t^{\top}{\diag(V_t)}{\theta})dt+\pi_t^{\top}\sqrt{\diag(V_t)}dB_{t}.
\end{align*}
Letting $\alpha_t := \sigma^\top(V_t)\pi_t $ be the investment strategy, the wealth $X^{\alpha}$ reads:
\begin{align}
	\label{eq:wealth}
	dX^{\alpha}_t &= \big( r(t) X^{\alpha}_t  + \alpha_t^\T \lambda_t \big) dt + \alpha_t^\T dB_t, \quad t \geq 0, \quad X_0^\alpha = x_0 \in \R. 
\end{align} 
By a standard calculation, the wealth process is then given by
\vspace{-.2cm} 
\begin{equation}\label{eq:wealthProcess} 
	X_t = e^{ \int^t_0 r(s) ds} \Big( x_0 + \int_0^t e^{ -\int^s_0 r(u) du}  \Big(\alpha_s^\T dB_s + \alpha_s^\T \lambda_s \mathrm{d}s \Big) \Big). 
\end{equation}
Note that it is sufficient to assume that \(\int_0^t (\left|\lambda_s\right|^2 + \left|\alpha_s\right|^2 )\, \mathrm{d}s < +\infty \) almost surely for all \( t \ge 0\)
in order to construct the stochastic integrals in Equation~\eqref{eq:wealthProcess}. This boundedness condition holds owing to the inequality \( |z|^2 \le 2 e^{|z|}, \, \forall z \in \R \), together with the condition~\eqref{eq:assumption_novikov}, and the following admissibility assumption, which is consistent with \cite[Definition~1]{HuImkellerMueller2005}. 
\begin{Definition}\label{Def:adm2}
	Let $\gamma > 0$, an investment strategy $\alpha(\cdot)$ is said to be admissible if 
	\begin{enumerate}
		\item[$(a)$] $\alpha(\cdot)$ is $\F$-adapted  
		 and $ \E[\int^t_0 \left|\alpha_s\right|^2  ds] < \infty$, $\forall \; t \in [0, T]$;
		\item[$(b)$] The wealth process (\ref{eq:wealth}) has a unique solution in terms of $(S,V,B)$, with $\p$-$\as$ continuous paths;
		\item[$(c)$] $\left\{\exp\big[ - \gamma e^{ \int^T_\tau r(u) du} X_\tau \big]: \tau \text{ stopping time with values in } [0, T] \right\}$ is a uniformly integrable family.
	\end{enumerate}
	In particular,
	\vspace{-.2cm}
	\begin{equation}\label{eq:uniformInt}
		\sup_{\tau \in [0, T]}\E\Big[ \exp\left(- p\gamma e^{ \int^T_\tau r(u) du} X^{\alpha}_\tau\right) \Big] < \infty, \text{ for some } p > 1.
	\end{equation}
	The set of all admissible investment strategies is denoted as $\cA$ and is naturally defined by
	$$\mathcal A  = \{ \alpha \in {L^{2,loc}_{\F}([0,T], \R^d)} \mbox{ such that \eqref{eq:wealth} has a  solution satisfying condition } ~\eqref{eq:uniformInt} \}.$$
\end{Definition}
\noindent The investor now considers the Merton exponential utility optimization problem, i.e. its aim is to find the value function $\mathcal{V}(x_0,V_0)$ for the CARA utility function such that
\vspace{-.2cm}
\begin{equation}\label{Expobj}
	\mathcal{V}(x_0,V_0) = \sup_{\alpha(\cdot) \in \cA } \E_{x_0,V_0}\left[ - \frac{1}{\gamma} \exp\Big(- \gamma X_T\Big) \right], \quad \gamma > 0.
\end{equation}
The solution method is similar to the power utility case. With slightly abuse of notations, we still use $\psi(\cdot)$, $\Gamma$, etc. However, they are redefined and not mixed with counterparts in power utility case.

\vspace{-.2cm}
\subsubsection{The degenerate correlation case}\label{degenerate}
\noindent We assume that the correlation in ~\eqref{eq:correstructureheston} is of the form $(\rho,\dots, \rho)$ for $\rho\in [-1,1]$. To construct the family of processes $\{ J^\alpha_t \}_{ t \in [0, T]}$, $\alpha \in \cA$, satisfying conditions~$(1)$-~$(2)$-~$(3)$ in Definition~\ref{def:Martopt}, we introduce the new probability measure  $\tilde{\mathbb{P}}$ defined via the Radon-Nikodym density or derivative at \(\mathcal{F}_T\) from
\vspace{-.2cm}
\[
\frac{d\tilde{\mathbb{P}}}{d\mathbb{P}}|_{\mathcal{F}_t}= \mathcal E\Big( -\int_0^t \sum_{i=1}^d\theta_i \sqrt{V^i_s}dB^i_s \Big) =\operatorname{exp}\Big(-\int_0^t \lambda_s^\top dB_{s}-\frac{1}{2}\int_0^t \left| \lambda_s \right|^2 ds\Big)
\]
where the stochastic exponential is a true martingale by Lemma~\ref{lm: extended_m_AJ_lemma} together with the new standard brownian motion under $\tilde{\mathbb{P}}$ , \(\widetilde{B}_{t}=B_{t}+\int_0^t\lambda_s ds.\)
Define the new process \(\widetilde{W}\) by
\vspace{-.2cm}
\[\widetilde{W}_t =  \Sigma  \widetilde{B}_t + \sqrt{I-\Sigma^\T \Sigma} B^{\perp}_t = W_{t}+\int_0^t\Sigma\lambda_s ds,\]
Notice that, by the 
Girsanov theorem, \(\widetilde{B}\) and \(\widetilde{W}\) are standard Wiener processes under the measure $\tilde{\mathbb{P}}$.
As in the one dimensional case (see e.g. \cite{HanWong2020b}), the Ansatz
\vspace{-.2cm} 
\begin{equation}\label{eq:ansatz2}
J_t^{\alpha}=-\frac{1}{\gamma}\exp\big(-\gamma e^{\int_{t}^{T} r(s)ds}X^\alpha_t\big)\Big(\E^{\tilde{\mathbb{P}}}\Big[\operatorname{exp}\big(-\frac{1-\rho^2}{2}\int_t^T \left| \lambda_s \right|^2 ds\big)|\mathcal{F}_t\Big]\Big)^{\frac{1}{1-\rho^2}}=:-\frac{1}{\gamma}\exp\big(-\gamma e^{\int_{t}^{T} r(s)ds}X^\alpha_t\big)\Gamma_t.
\end{equation}
is inspired by the martingale distortion transformation in a non-Markovian setting in \cite{Teh04}  for the case of the
exponential utility function, where the distortion
power arises from simple H\"older-type inequalities.
Here we still use the short notation $J_t^{\alpha}$ for $J_t(X_t^{\alpha},V_t)$. 
\vspace{-.2cm}
\begin{Proposition}
	\label{prop:ExpoExp_riccati_1}Assume that there exists  a solution $\psi \in C([0,T],\R^d)$ to the inhomogeneous Riccati-Volterra equation \eqref{eq:RiccatiExpopsi1}-\eqref{eq:RiccatiExpopsi2} below:
	\vspace{-.3cm}
	\begin{align}
		\psi^i(t)&= \int_0^t K_i(t-s)\big(-\frac{ \theta_i^2}{2}  + F_i(T-s,\psi(s))\big) ds,  \label{eq:RiccatiExpopsi1} \\
		F_i(s,\psi) &= -\rho \theta_i \nu_i \varsigma^i(s) \psi^i + (D^\top \psi)_i + \frac {\nu_i^2} 2 (1-\rho^2) (\varsigma^i(s)\psi^i)^2, \quad i=1,\ldots,d, \label{eq:RiccatiExpopsi2}
	\end{align} 
	\vspace{-.2cm}
	Then, 
	\vspace{-.3cm}
	\begin{equation}\label{eq:GammaExpo}
		\Gamma_t = \exp\Big( \sum_{i=1}^d\int_t^T  \big(-\frac{ \theta_i^2}{2}  +  F_i(s,\psi(T-s))\big) \tilde{g}^i_t(s) ds \Big).
	\end{equation}
	where $\tilde{g}=(\tilde{g}^1,\ldots,\tilde{g}^d)^\T$ is the $\R^d$-valued process \((\tilde{g}_t(s))_{t\leq s}\) denoting the adjusted conditional $\tilde{\mathbb{P}}$-expected variance.
	Moreover,  $\Gamma_t$ is essentially bounded. Specifically, $0 < \Gamma_t \leq 1$, $\forall \; t \in [0, T]$, $\p$-$\as$.
	Let $\Lambda$ be defined as 
	\begin{equation}
		\Lambda_t^i =  \nu_i\varsigma^i(t) \psi^i(T-t) \sqrt{V^i_t}, \quad i=1,\ldots,d, \quad 0 \leq t \leq T, 
	\end{equation}
	Moreover, $\left(\Gamma, \Lambda\right)\in\mathbb{S}^{\infty}_{\F}([0,T], \R) \times L^2_{\F}([0,T], \R^d)$ and
	\vspace{-.3cm}
	\begin{align}
		d\Gamma_t &= \;  \Gamma_t \Big(\sum_{i=1}^d \Big(\frac{ \theta_i^2}{2}  + \frac {\nu_i^2} 2  \rho^2 (\varsigma^i(t)\psi^i(T-t))^2 \Big) V^i_t  + \sum_{i=1}^d \psi^i(T-t)\nu_i \varsigma^i(t) \sqrt{V^i_t}d\widetilde{W}^i_t \Big)\\
		&=  \;  \Gamma_t \Big[\frac{1}{2} \big( \left| \lambda_t\right|^2 + \left| \Sigma \Lambda_t \right|^2\big)dt + \Lambda_t^\top d\widetilde{W}_t \Big],\quad \Gamma_T = 1.\label{eq:gamma_Expoheston}
	\end{align}
\end{Proposition}
\noindent{\bf Remark} 
Assume that $K$ satisfies the Assumption~\ref{assump:kernelVolterra}. As $1-\rho^2 \geq 0$, then~\cite[Theorem A.2.]{Gnabeyeu2026a} provides the existence of a unique global continuous solution  on \([0,T]\) to ~\eqref{eq:RiccatiExpopsi1}--~\eqref{eq:RiccatiExpopsi2}.
The main result we provide for this case is the following:
\vspace{-.1cm}
\begin{Theorem}\label{Thm:ExpoExp_riccati_1}
	Let \( \psi \) be the unique, continuous solution of the inhomogeneous Ricatti-Volterra equation  	~\eqref{eq:RiccatiExpopsi1}-~\eqref{eq:RiccatiExpopsi2}
	on the interval \( [0, T] \), such that Assumption~\ref{assm:gen} is in force. Then \( J^\alpha_t \) in~\eqref{eq:ansatz2} satisfies the martingale optimality principle for \( t \in [0, T] \), and the optimal portfolio strategy \( \alpha^* \) is given by
	\vspace{-.2cm}
	\begin{align}
		\label{Eq:alpha_Expo*1}
		\c^*_t &= \frac1\gamma  e^{-\int_{t}^{T} r(s)ds} \left( \lambda_t +  \Sigma \Lambda_t \right) = \frac1\gamma  e^{-\int_{t}^{T} r(s)ds}\sqrt{\diag(V_t)} \left( \theta + \Sigma \nu \varsigma(t) \psi(T - t) \right), \quad 0 \leq t \leq T \\
		&= \;  \Big(\frac1\gamma  e^{-\int_{t}^{T} r(s)ds}  \big(\theta_i + \rho\nu_i\varsigma^i(t) \psi^i(T-t) \big) \sqrt{V_t^i}   \Big)_{1 \leq i \leq d}, \quad 0 \leq t \leq T.\label{Eq:alpha_Expo*2}
	\end{align}
	\vspace{-.2cm}
	Moreover, \(X^{\alpha^*}\) satisfies Equation~\eqref{eq:uniformInt}
	and $\alpha^*$ is admissible.
\end{Theorem}
\noindent  The proof is a straightforward adaption of the arguments from the proof of Theorem~\ref{Thm:ExpoUtilityGeneral} to the particular case of degenerate correlation.
\vspace{-.5cm}
\subsubsection{The general correlation case: A verification argument}\label{general2}
\vspace{-.2cm}
\noindent As developed in the previous section, when the correlation structure is highly degenerate, that is, when 
\( \rho_1 = \ldots = \rho_d \), the martingale distortion approach used by \cite{HanWong2020b} in the one-dimensional case extends naturally to the multivariate setting.
When the correlation structure in~\eqref{eq:correstructureheston} is given by an arbitrary vector 
\((\rho_1,\dots,\rho_d)\in[-1,1]^d\), we rewrite the Riccati--Volterra equations 
\eqref{eq:RiccatiExpopsi1}--\eqref{eq:RiccatiExpopsi2} as
\vspace{-.2cm}
\begin{align}
	\psi^i(t) &= \int_0^t K_i(t-s)\Big(-\frac{\theta_i^2}{2} + F_i(T-s,\psi(s))\Big)\, ds,  \label{eq:RiccatiExpoTilpsi1} \\
	F_i(s,\psi) &= -\theta_i \rho_i \nu_i \varsigma^i(s)\psi^i + (D^\top \psi)_i 
	+ \frac{\nu_i^2}{2}(1-\rho_i^2)\big(\varsigma^i(s)\psi^i\big)^2, 
	\quad i=1,\ldots,d. \label{eq:RiccatiExpoTilpsi2} 
\end{align}
First, note that if $K_i$ satisfies the Assumption~\ref{assump:kernelVolterra} for \(i=1,\ldots,d\). As $1-\rho_i^2 \geq 0$, then~\cite[Theorem A.2.]{Gnabeyeu2026a} provides the existence of a unique global continuous solution  on \([0,T]\) to ~\eqref{eq:RiccatiExpoTilpsi1}--~\eqref{eq:RiccatiExpoTilpsi2}. More details are given in the Remark on Proposition~\ref{prop:ExpoExp_riccati_2}.
\noindent In particular, when \( \rho_1 = \ldots = \rho_d = \rho \), the function \(\psi\) coincides with the unique global solution to the Riccati--Volterra equations \eqref{eq:RiccatiExpopsi1}--\eqref{eq:RiccatiExpopsi2}.
Consequently, to avoid restrictions on the correlation structure, we use a verification approach for the Martingale optimality principle to
solve the optimization problem, thus extending the results obtained in the preveous section to the more general correlation structure.  

\begin{Proposition}
	\label{prop:ExpoExp_riccati_2}
	Assume that there exists  a solution $\psi \in C([0,T],\R^d)$ to the inhomogeneous Riccati-Volterra equation:
	\vspace{-.5cm}
	\begin{align}
		\psi^i(t)&= \int_0^t K_i(t-s)\big(-\frac{ \theta_i^2}{2}  + F_i(T-s,\psi(s))\big) ds, 
		\label{eq:RiccatiExpTilpsi1} \\
		F_i(s,\psi) &=-\theta_i \rho_i \nu_i \varsigma^i(s) \psi^i + (D^\top \psi)_i + \frac {\nu_i^2} 2 (1-\rho^2_i) (\varsigma^i(s)\psi^i)^2, \quad i=1,\ldots,d,
		\label{eq:RiccatiExpTilpsi2}
	\end{align} 
	\vspace{-.2cm}
	Let $\left(\Gamma, \Lambda\right)$ be defined as 
	\vspace{-.2cm}
	\begin{equation}\label{eq:GammaExp}
		\left\{
		\begin{array}{ccl}
			\Gamma_t &=& \exp\Big( \sum_{i=1}^d\int_t^T  \big(-\frac{ \theta_i^2}{2}  +  F_i(s,\psi(T-s))\big) g^i_t(s) ds \Big), \\
			\Lambda_t^i &=&  \nu_i\varsigma^i(t) \psi^i(T-t) \sqrt{V^i_t}, \quad i=1,\ldots,d, \quad 0 \leq t \leq T, 
		\end{array}  
		\right.
	\end{equation}
	where $g$ $=$ $(g^1,\ldots,g^d)^\T$ is given by \eqref{eq:processg} i.e. the $\R^d$-valued process \((g_t(s))_{t\leq s}\)  is defined in~\eqref{eq:processg}.  
	
	\noindent Then, $\left(\Gamma,\Lambda\right)$  is a  $\mathbb{S}^{\infty}_{\F}([0,T], \R) \times L^2_{\F}([0,T], \R^d)$-valued solution to the Riccati BSDE~\eqref{eq:gamma_ExpTilheston} below.
	\vspace{-.1cm}
	\begin{equation}\label{eq:gamma_ExpTilheston}
		\left\{
		\begin{array}{ccl}
			\frac{d\Gamma_t}{\Gamma_t} &=&   \frac{1}{2} \left| \lambda_t + \Sigma \Lambda_t \right|^2\,dt + \Lambda_t^\top dW_t, \\
			\Gamma_T &=&  1. 
		\end{array}  
		\right.
	\end{equation}
	Moreover,  $\Gamma_t$ is essentially bounded. Specifically, $0 < \Gamma_t \leq 1$ for all $ t \in [0, T]$, $\p$-$\as$.
\end{Proposition}
\noindent {\bf Proof: } The proof that $\left(\Gamma, \Lambda\right)$ satisfy~\eqref{eq:gamma_ExpTilheston} is a straightforward adaptation of the arguments from the proof of Proposition~\ref{prop:ExpoExp_riccati_1}. {It remains to show that $0 < \Gamma_t \leq 1$ for all $ t \in [0, T]$, $\p$-$\as$ and $\left(\Gamma, \Lambda\right) \in \mathbb{S}^{\infty}_{\F}([0,T],\R) \times L^2_{\F}([0,T], \R^d)$}.  For this, define the process 
\vspace{-.3cm}
\begin{align*}
	M_t &= \Gamma_t \exp\Big(-\frac12\int_0^t \left| \lambda_s + \Sigma \Lambda_s \right|^2 ds\Big)=  \Gamma_t \exp\Big(-\frac12\int_0^t \sum_{i=1}^d V^i_s \big(\theta_i + \rho_i \nu_i \varsigma^i(s)\psi^i(T-s)\big)^2 ds\Big), \quad t \leq T.
\end{align*}
An application of It\^o's formula combined with the dynamics ~\eqref{eq:gamma_ExpTilheston} shows that $dM_t = M_t  \Lambda_t^\top dW_t$, so that $M$ is a local martingale of the form
\vspace{-.3cm}
\begin{align} 
	M_t &=\;\mathcal E\Big( \int_0^t \Lambda_s^\top dW_s\Big)=\; \mathcal E\Big( \int_0^t \sum_{i=1}^d \nu_i\varsigma^i(s)  \psi^i(T-s)\sqrt{V^i_s}dW^i_s \Big). 
\end{align}
\noindent Since $\psi$ is continuous, it is bounded; likewise, $\varsigma$ is bounded. Therefore a straightforward application of Lemma~\ref{lm: extended_m_AJ_lemma} with $g_{2} = 0$ and $g_{1,i}(s) = \nu_i \varsigma^i(s)\psi^i (T-s) \in L^{\infty}(\R^+,\R)$,  recall~\eqref{eq:moments V1}, yields that the stochastic exponential \(M\) is a true $\P$- martingale. Now, as $\Gamma_T=1$, writing $\E[M_T|\mathcal F_t]=M_t$, we obtain
\vspace{-.3cm}
\begin{equation}\label{eq:ito_gammaExp}
	\Gamma_t = \E \Big[ \exp\Big(-\frac12\int_t^T \sum_{i=1}^d V^i_s \big(\theta_i + \rho_i \nu_i \varsigma^i(s) \psi^i(T-s)\big)^2 ds\Big) \mid \mathcal F_t\Big], \quad t\leq T,
\end{equation}
which ensures that $0<\Gamma_t\leq1$, $\P-a.s.$, since $V_t$ is non-negative ($V\in \mathbb R^d_+$). As for $\Lambda$, it is clear that it belongs to $L^2_{\F}([0,T], \R^d)$ since $\varsigma$ and $\psi$ are bounded, $\psi$ is continuous thus bounded and $\E \Big[\int_0^T  \sum_{i=1}^d V^i_s ds \Big] <  \infty$ by \eqref{eq:moments V1}. This complete the Proof\hfill $\Box$

\medskip
\noindent{\bf Remark:} 
Assume that $K$ satisfies the Assumption~\ref{assump:kernelVolterra}.
Then \cite[Theorem A.2.]{Gnabeyeu2026a} provides the existence of a unique global continuous solution $\psi \in C([0,T],\R^d)$ to ~\eqref{eq:RiccatiExpTilpsi1}--~\eqref{eq:RiccatiExpTilpsi2} (and in particular to ~\eqref{eq:RiccatiExpopsi1}--~\eqref{eq:RiccatiExpopsi2} ) and \(\psi< 0\) for \(t>0\).
\noindent More precisely, as the matrix $D$ in the drift of the volatility process is a diagonal matrix, i.e. $D=-\diag{(\lambda_1,\dots, \lambda_d)}$, for \(i=1,\ldots,d,\) by \cite[Theorem A.2.~$(c)$]{Gnabeyeu2026a}, since $-\frac{\theta_i^2}{2}<0$, $\psi^i \in C([0,T],\R_-)$ is unique global solution to the following Volterra equation 
\vspace{-.2cm}
\begin{equation}\label{eq:comp}
	\chi(t) = \int_0^t K_i(t-s)  \Big( -\frac{\theta_i^2}{2} - \big(\lambda_i + \theta_i \rho_i \nu_i \varsigma^i(T-s)\big) \chi(s)  + \frac {\nu_i^2} 2  ( 1- \rho_i^2)\varsigma^i(T-s)^2\chi(s) ^2\Big)ds, \; t\leq T.
\end{equation}
Combining the component-wise solutions, we finally obtain the unique global solution $\psi$ of the inhomogeneous Ricatti--Volterra Equation~\eqref{eq:RiccatiExpTilpsi1}--~\eqref{eq:RiccatiExpTilpsi1} (and in particular to ~\eqref{eq:RiccatiExpopsi1}--~\eqref{eq:RiccatiExpopsi2} ). 

\smallskip
\noindent Moreover, it follows in this case that the condition \eqref{eq:condtheta} can be made more explicit by bounding $\psi$ with respect to the vector $\theta$.
Indeed setting for \(i=1,\ldots,d,\) \(\bar{\lambda}_i:=\inf_{t\in [0,T]} \big(\lambda_i+\nu_i\rho_i\theta_i\varsigma^i(t)\big)= \lambda_i+\nu_i\rho_i\theta_i \|\varsigma^i\|_\infty {\bold{1}_{\rho_i \leq 0}}, \) (owing to Example~\ref{Ex:FractionalKernel2} for the function \(\varsigma\)) and assuming that \(\bar{\lambda}_i \neq 0\), by \cite[Corollary A.3.]{Gnabeyeu2026a}, we have:
\vspace{-.2cm}
\begin{equation}
	\sup_{t \in [0,T]} |\psi^i(t)| \leq \frac{|\theta_i |^2}{2\bar{\lambda}_i}\int_0^T f_{\bar{\lambda}_i}(s)ds = \frac{|\theta_i |^2}{2\bar{\lambda}_i}(1- R_{\bar{\lambda}_i}(T)), \;i=1,\ldots,d.
\end{equation}
where $R_{\bar{\lambda}_i}$ is the \textit{ $\bar{\lambda}_i$-resolvent} associated to the real-valued kernel $K_i$ and $f_{\bar{\lambda}_i}$ its antiderivative.
Consequently, combining those component-wise estimates, we finally obtain that, a sufficient condition on $\theta$ to ensure \eqref{eq:condtheta} would be 
\begin{equation}
	\theta^2_i\left( 1+(\nu_i\|\varsigma^i\|_\infty\frac{\theta_i}{2\bar{\lambda}_i})^2\big(1- R_{\bar{\lambda}_i}(T)\big)^2 \right) \leq \frac{a}{a(p)}  \quad \text{for all}\quad i=1,\ldots,d.
\end{equation}

\medskip
\noindent
Now, observe that the proposed candidate for the optimal portfolio strategy $\alpha^*$ follows from
\eqref{Eq:alpha_Expo*1}--\eqref{Eq:alpha_Expo*2} and is given by
\vspace{-.3cm}
\begin{equation}\label{eq:optcandExpo}
	\alpha_t^*
	=\frac1\gamma  e^{-\int_{t}^{T} r(s)ds} \left( \lambda_t + \Sigma \Lambda_t \right)
	= \Bigg(\frac1\gamma  e^{-\int_{t}^{T} r(s)ds}
	\big(\theta_i + \rho_i \nu_i \varsigma^i(t)\psi^i(T-t)\big)
	\sqrt{V_t^i}\Bigg)_{1 \leq i \leq d},
	\quad 0 \leq t \leq T.
\end{equation}
We will show directly that this strategy attains the value
$-\frac{1}{\gamma}\exp\left(-\gamma e^{\int_{t}^{T} r(s)ds}x_0\right)\Gamma_0$, and that no other admissible portfolio
strategy can achieve a higher value.
The main verification result for this setting is stated as follows:
\begin{Theorem}\label{Thm:ExpoUtilityGeneral}
	Let \( \psi \) be the unique, continuous solution of the inhomogeneous Ricatti-Volterra equation  	~\eqref{eq:RiccatiExpTilpsi1}-~\eqref{eq:RiccatiExpTilpsi2}
	on the interval \( [0, T] \), such that Assumption~\ref{assm:gen} is in force.
	Then for $t\in [0,T]$, an optimal investment strategy $(\alpha_t^*)_{t\in [0,T]}$ for the Merton portfolio problem~\eqref{Expobj} is given by 
	\vspace{-.2cm}
	\begin{align}
		\label{Eq:alpha_GeneralExpo*}
		\c^*_t &= \frac1\gamma  e^{-\int_{t}^{T} r(s)ds} \left( \lambda_t +  \Sigma \Lambda_t \right) = \frac1\gamma  e^{-\int_{t}^{T} r(s)ds} \sqrt{\diag(V_t)} \left( \theta + \Sigma \nu \varsigma(t) \psi(T - t) \right), \quad 0 \leq t \leq T \\
		&= \;  \Big(\frac1\gamma  e^{-\int_{t}^{T} r(s)ds}  \big(\theta_i + \rho_i\nu_i\varsigma^i(t) \psi^i(T-t) \big) \sqrt{V_t^i}   \Big)_{1 \leq i \leq d}, \quad 0 \leq t \leq T.\label{Eq:alpha_GeneralExp*2}
	\end{align}
	Moreover,
	\vspace{-.2cm}
	\begin{equation}\label{eq:boundExpo}
		\sup_{\tau \in [0, T]}\E\Big[ \exp\left(- p\gamma e^{ \int^T_\tau r(u) du} X^{\alpha^*}_\tau\right) \Big] < \infty, \text{ for some } p > 1.
	\end{equation}
	and $\alpha^*$ is admissible. The value function defined in~\eqref{Expobj} can be written as
	\vspace{-.2cm}
	\begin{equation}\label{eq:Optvalue_Exp}
		\mathcal{V}(x_0,V_0)=-\frac{1}{\gamma}\exp\left(-\gamma e^{\int_{0}^{T} r(s)ds}x_0\right)\exp\Big( \sum_{i=1}^d\int_0^T  \big(-\frac{ \theta_i^2}{2}  +  F_i(s,\psi(T-s))\big) g^i_0(s) ds \Big).
	\end{equation}
\end{Theorem}
\noindent For the sake of brevity, yet without compromising self-containment, we present a sketch of the proof, deferring the detailed proof to Section~\ref{sect:proofMresult}.

\smallskip
\noindent \emph{Sketch of Proof:}
We adopt a different verification approach from that of Section~\ref{general}. In the spirit of~\cite{HuImkellerMueller2005}, we consider the family of stochastic processes $\{ J^\alpha_t \}_{ t \in [0, T]}$, $\alpha \in \cA$  defined for every \( t\in[0,T]\) by \begin{equation}\label{eq:ansatzExpo}J_t^\alpha:=
	-\frac{1}{\gamma}\exp(-\gamma x_0  e^{ \int^T_0 r(s) ds}) \exp\!\Big(-\gamma \int_0^t e^{ \int^T_s r(u) du}  \big(\alpha_s^\T dB_s + \alpha_s^\T \lambda_s \mathrm{d}s \big) \Big)  \Gamma_t = -\frac{ \Gamma_t}{\gamma}\exp(-\gamma  e^{\int_{t}^{T} r(s)ds}X^\alpha_t)
\end{equation}
where the pair \((\Gamma,\Lambda)\) satisfies the Riccati Backward SDE (BSDE)~\eqref{eq:GammaExp}--\eqref{eq:gamma_ExpTilheston} under \(\P\).
We verify that $\{ J^\alpha_t \}_{ t \in [0, T]}$, $\alpha \in \cA$ satisfies the Martingale optimality principle in the sense of Definition~\ref{def:Martopt}.

\subsection{Optimal strategy for the logarithmic utility maximization problem}\label{subsect:logarithmic}
\noindent To complete the spectrum of important utility functions,
in this section we shall consider logarithmic utility. Here, we assume the model~\eqref{eq:stocks}--~\eqref{eq:hestonS} for $S_t$, that the dynamics of the controlled wealth process is given by equations~\eqref{Eq:wealth}--\eqref{eq:wealthPowerSol} and the utility function is of the form \(U(x):=\log(x)\).

\noindent Let the set of all admissible investment strategies, still denoted as $\cA$ be given by
	\begin{equation}\label{eq:AdmStrat}
		\mathcal{A} = \left\{ \alpha=(\alpha_t)_{t \in [0,T]} \in {L^{2}_{\F}([0,T], \R^d)} 
		\right\}.
	\end{equation}
\noindent Note that, for any \(\alpha\in\mathcal A\), Equation \eqref{Eq:wealth} has a positive solution \((X_t^{\alpha})_{t\in[0,T]}\). We want to solve the Merton logarithmic utility optimization case, i.e. our aim is to find the value function defined in~\eqref{eq:value0} for the logarithmic utility function such that
\begin{equation}\label{obj_log}
	\begin{aligned}
		\hspace{1cm}	\mathcal{V}(x_0,v_\infty)&:= \operatorname*{\,sup}_{\alpha(\cdot) \in \mathcal A} \E\Big[\log(X_T^{\alpha}) \Big],\qquad X_0^{\alpha}:=x_0>0, \quad v_\infty:= \mathbb{E}[V_0]\\
		&= \log(x_0) + \sup_{\alpha(\cdot) \in \mathcal{A}} 
		\mathbb{E}\Bigl[\int_0^T \bigl(r(s) + \alpha_s^\top \lambda_s 
		- \tfrac{1}{2}\left|\alpha_s\right|^2\bigr)ds
		+ \int_0^T \alpha_s^\top\,dB_s\Bigr].
	\end{aligned}
\end{equation}
\begin{Theorem}\label{Thm:LogUtilityGeneral}
	Let \(T>0\) be fixed and that Assumption~\ref{assm:gen} is in force.
	Then for $t\in [0,T]$, an optimal investment strategy $(\alpha_t^*)_{t\in [0,T]}$ for the Merton portfolio problem~\eqref{obj_log} is given by 
	\begin{align}
		\label{Eq:alpha_GeneralLog*}
		\c^*_t &= \lambda_t = \;  \Big(\theta_i  \sqrt{V_t^i}   \Big)_{1 \leq i \leq d}, \quad 0 \leq t \leq T.
	\end{align}
	Moreover $\alpha^*$ is admissible and the value function or maximum expected utility defined in~\eqref{obj_log} satisfies
	\begin{equation}
		\mathcal{V}(x_0,v_\infty):=\log(x_0) +\int_{0}^{T} r(s)ds+ \frac{T}{2}\sum_{i=1}^d \theta_i^2\,v_\infty^i,\quad v_\infty^i:= \mathbb{E}[V_0^i] = \frac{1-a_i}{1-a_i\varphi_\infty^i}\frac{\mu_\infty^i}{\lambda_i}\text{~for~} i=1,\ldots,d.
	\end{equation}
\end{Theorem}
	\section{Numerical experiments:The fake stationary rough Heston volatility}\label{Sec:Num}
	\vspace{-.2cm} 
	\noindent In this section, we illustrate the results of Section~\ref{Sect:affine} by numerically computing the optimal portfolio strategy for a special case of two-dimensional fake stationary rough heston model model as described in section~\ref{Sec:SolMerton}. We consider a financial market consisting of one risk-free asset and \(d = 2\) risky assets, with an investment horizon of \(T = 1\) year. 	
	To model the roughness of the asset price dynamics, we employ an appropriate integration kernel. We choose a fractional kernel of Remark~\ref{rm:Kernels} and Example~\ref{Ex:FractionalKernel2} of the form:
	\[
	K(t) = \begin{pmatrix}
		\frac{t^{\alpha_1-1}}{\Gamma(\alpha_1)} & 0 \\
		0 & \frac{t^{\alpha_2-1}}{\Gamma(\alpha_2)}
	\end{pmatrix}, \quad 0.4 +\frac12=\alpha_1,\, 0.1 +\frac12=\alpha_2 \in \big( \frac{1}{2}, 1 \big).
	\]
	Here, \( \Gamma(\alpha) \) is the Gamma function, and the parameter \( \alpha \) controls the degree of roughness in the model. 
	
	\medskip
	\noindent 
	Note that, the model is sufficiently rich to capture several well-known stylized facts of financial markets:
	
	\begin{itemize}
		\item Each asset \(S^i\), \(i=1,2\) exhibits stochastic rough volatility driven by the process \(V^i\), with different Hurst indices \(\alpha_i\).
		\item Each stock \(S^i\) is correlated with its own volatility process through the parameter \(\rho_i\) to take into
		account the leverage effect.
	\end{itemize}
	\medskip
	\noindent
	We consider the setting where in Equation~\eqref{VolSqrt_}, the simplified specification
	\(\varphi(t) = I_{2\times2}, \; t \ge 0, \) holds almost surely, in which case the \(\R^d-\) valued mean-reverting function \(\mu\) is constant in time, that is, \(\mu(t) = \mu_0 \in \R^2, \; \forall\,t \ge 0,\) (see, e.g.,~\cite{EGnabeyeu2025}).
	
	\medskip
	\noindent We consider the following estimates for the model parameters \(V_0^i \sim \mathcal{N}(\frac{\mu_0^i}{\lambda_i},v_0^i)\) defined in~\eqref{eq:VolterraVarTime_1} ($\textit{($E_{\lambda_i, c_i}$)}$):
	\[
	c = \begin{pmatrix}
		0.01 \\
		0.03
	\end{pmatrix}, \quad
	\mu_0 = \begin{pmatrix}
		2.0 \\
		2.5
	\end{pmatrix},\;
	D = \begin{pmatrix}
		-0.2 & 0 \\
		0 & -0.6
	\end{pmatrix}, \quad
	\Sigma = \begin{pmatrix}
		-0.7 & 0 \\
		0 & -0.55
	\end{pmatrix}, \;
	\theta = \begin{pmatrix}
		0.1 \\
		0.1 
	\end{pmatrix}, \;
	\nu = \begin{pmatrix}
	0.4 \\
	0.2
	\end{pmatrix}.
	\]
\medskip
\noindent {\bf Remark:}
	In order to numerically implement the optimal strategy \eqref{Eq:alpha_Generalpower*2}--\eqref{Eq:alpha_GeneralExp*2}, one needs to simulate  the non-Markovian process $V$ in Equation~\eqref{VolSqrt_}--~\eqref{VolSqrt2} and to discretize the Riccati-Volterra equation  for $\psi$ in~\eqref{eq:RiccatiPowerTilpsi1}--~\eqref{eq:RiccatiPowerTilpsi2} and~\eqref{eq:RiccatiExpTilpsi1}--~\eqref{eq:RiccatiExpTilpsi2} respectively.\\
 	
\noindent To simulate the Volterra process~\eqref{VolSqrt_}--\eqref{VolSqrt2}, we first rewrite it using the equivalent Wiener--Hopf transform (see, e.g.,~\cite[Propositions~2.8 and~2.2]{EGnabeyeu2025,EGnabeyeuPR2025}). 
We then introduce the $f_\lambda$-integrated discrete time Euler-Maruyama scheme defined by the below equation on the time grid $t_k =t^n_k =\frac{kT}{n}, k=0, \dots, n$ and in the fractional kernel case, namely recursively for \(i=1,2\), $\;\overline V^{i,n}_{0}=V^{i,0}$ and for every $k=1,\ldots,n$, 
{\small  
\begin{equation*}\label{eq:EulerXdisc}
	\overline V^{i,n}_{t_{k}} 
	=  \frac{\mu_0^i}{\lambda_i} + \big(V^{i,0} -\frac{\mu_0^i}{\lambda_i}  \big)R_{\lambda_i}(t_k) +  \frac{\nu_i}{\lambda_i}\sum_{\ell=1}^{k} \varsigma^i(t_{\ell})\sqrt{ \overline{V}_{t_{\ell-1}}^{i,n}} \int_{t_{\ell-1}}^{t_{\ell}} f_{\alpha_i,\lambda_i}(t_{k} -s) \, dW_s = h^i(t_k) + \frac{\nu_i}{\lambda_i}\sum_{\ell=1}^{k} \varsigma^i(t_{\ell})\sqrt{ \overline{V}_{t_{\ell-1}}^{i,n}} I^{i,\ell}_k.
\end{equation*}
}
\noindent where the stochastic integrals $\left(I^{i,\ell}_k= \int_{t_{\ell-1}}^{t_{\ell}} f_{\alpha_i,\lambda_i}(t_{k} -s) \, dW_s \right)$ can be simulated on the discrete  grid \((t^n_k)_{0\leq k\leq n}\) by generating an independent sequence of gaussian vectors \( G^{n,\ell}, \ell=1 \cdots n\) using an extended and stable version of Cholesky decomposition of a well-defined covariance matrix \(C\).
The reader is referred to \cite[Appendix A]{EGnabeyeu2025} for further details about the simulation of the Gaussian stochastic integrals terms in the semi-integrated Euler scheme introduced in this context for Equation~\eqref{VolSqrt_}--~\eqref{VolSqrt2}.
Theoretical guarantees for the convergence of this numerical scheme, as well as the convergence rate are established in~\cite{GnabeyeuPages2026} for more general Kernels and path-dependent coefficients.

\medskip
\noindent 
To numerically solve the two-dimensional Riccati--Volterra system 
\eqref{eq:RiccatiPowerTilpsi1}--\eqref{eq:RiccatiPowerTilpsi2} and 
\eqref{eq:RiccatiExpTilpsi1}--\eqref{eq:RiccatiExpTilpsi2}, 
we employ, as in~\cite{Gnabeyeu2026a} (see also~\cite[Section~5.1]{el2019characteristic}), 
the generalized Adams--Bashforth--Moulton scheme, also known as the \textit{fractional Adams method}. 
This predictor--corrector method, introduced in~\cite{DiethelmFordFreed2002,DiethelmFordFreed2004} 
for fractional ordinary differential equations, comes with proven convergence guarantees 
established in~\cite{LiTao2009}.

\begin{figure}[H]
	\centering
	\includegraphics[width=0.93\linewidth]{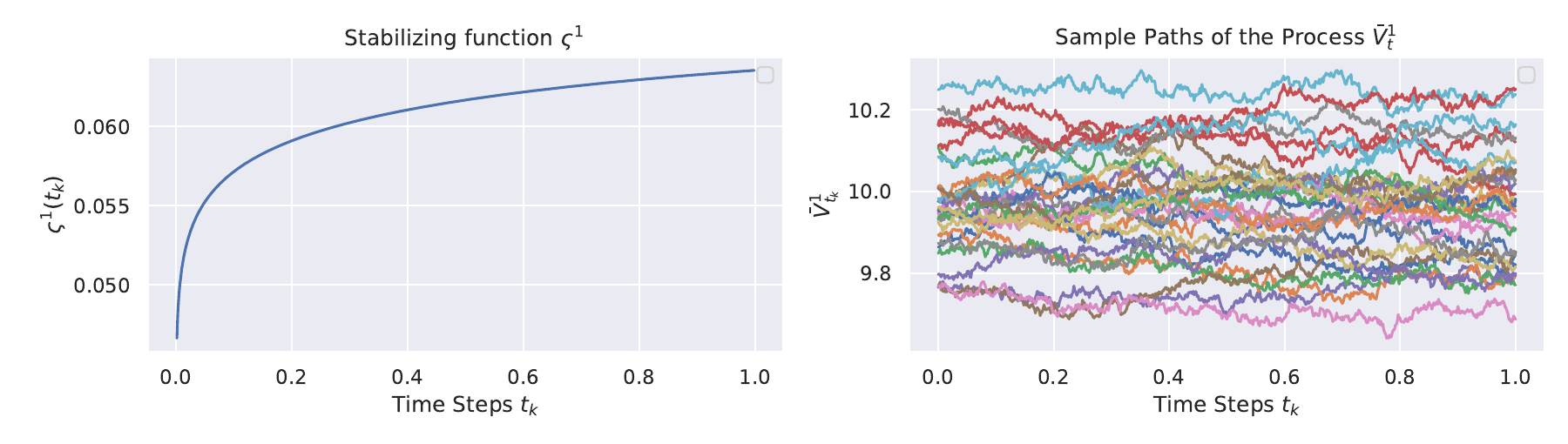
	}
	\caption{\textit{Graph of the stabilizer $ t \to \varsigma_{\alpha_1,\lambda_1,c_1}(t)$ (left) and \(30\) samples paths \( t_k \mapsto V^1_{t_k} \) (right) over the time interval \( [0, 1] \), for the Hurst esponent \( H = 0.4 \), \( c_1 = 0.01 \) and number of time steps \( n = 600 \).}}
\end{figure}\label{fig:stabil_Mean}
\noindent The drift and stabilizing functions $\varsigma$ introduced in Proposition~\ref{prop:timeDen_} and Example~\ref{Ex:FractionalKernel2} are designed to ensure that the volatility processes $(V^1, V^2)$ possess constant marginal means (Figure~\ref{fig:stabil_Mean}) and variances (Figure~\ref{fig:_variance}) over time. This guarantee the invariance of the first two moments under time shifts. 
\begin{figure}[H]
	\centering
	\includegraphics[width=0.93\linewidth]{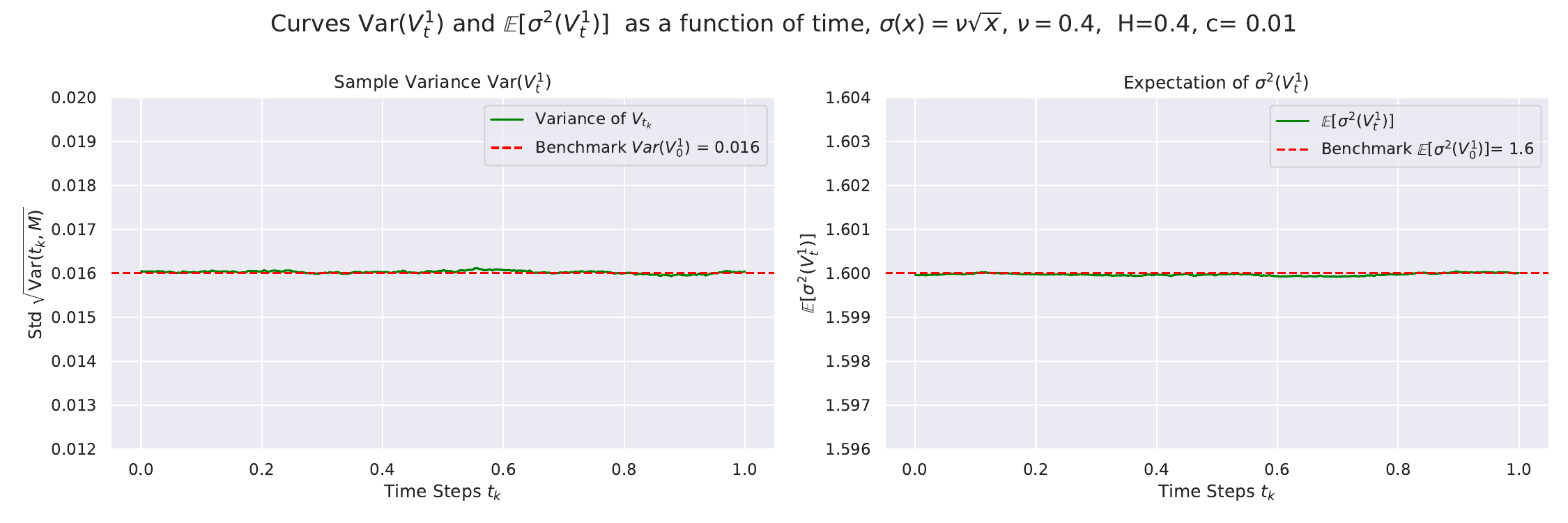}
	\caption{\textit{Graph of \( t_k \mapsto \text{Var}(V^1_{t_k}, M) \) and \( t_k \mapsto \mathbb{E}[\sigma^2(V^1_{t_k},M)] \) over \( [0, 1] \), \( c_1 = 0.01 \) and \( n = 600 \).}}
\end{figure}\label{fig:_variance}

\medskip
\noindent Figures~\ref{Fig:Riccati} below also confirm the Remark on Proposition~\ref{prop:ExpoPower_riccati_2}, that is the claim that $\psi \leq 0$ for every \(\rho_i \in (0,1)\) and every \(i\in\{1,2\}\) in the exponential utility case. The lower the Hurst coefficient $H$, the more negative is $\psi$. In the power utility case however, the left panel of the figure shows that, in accordance with~\cite{HanWong2020b}, $\psi(t) $
is positive for $t > 0 $ and its value becomes larger for a smaller hurst exponent $H$. 
\vspace{-.2cm}
\begin{figure}[H]
	\centering
	\includegraphics[width=0.87\linewidth]{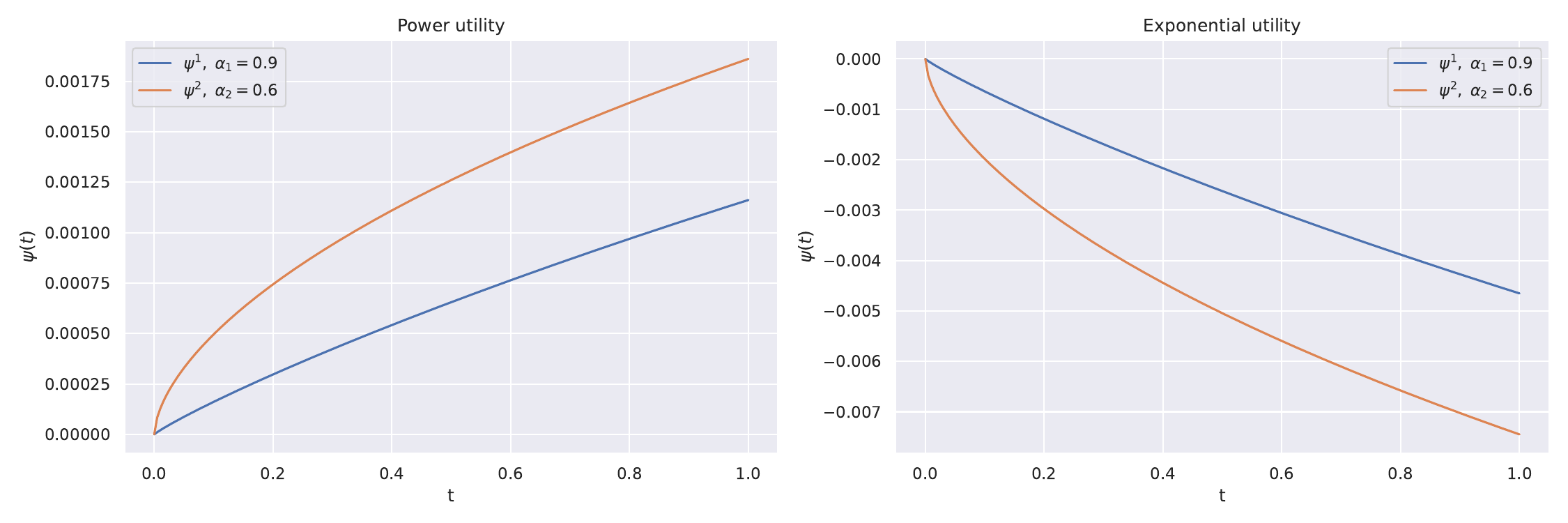}
	\vspace{-.2cm}
	\caption{\textit{Graph of \( t_k \mapsto \psi^1_{t_k} \) and \( t_k \mapsto \psi^2_{t_k} \) over \( [0, 1] \) with the fractional Adams algorithm, \(\gamma=0.2\) and the number of time steps \( n = 200 \) both for Power (left) and Exponential (right) utilities functions.}}\label{Fig:Riccati}
\end{figure}
\vspace{-.3cm}
\noindent Since the optimal strategies $(\alpha^*_t)_{t\in[0,T]}$ given by~\eqref{Eq:alpha_Generalpower*2}--\eqref{Eq:alpha_GeneralExp*2} are stochastic processes, we rather consider the evolution of the optimal vector of amounts invested in each stock, 
that is, the associated deterministic mapping \(t \mapsto \pi_t^*,\)
(recall that $\alpha_t^* = \sigma(V_t)^{\top}\pi_t^*$ with $\sigma(V_t) = \sqrt{\mathrm{diag}(V_t)}$, 
and $\alpha^*$ is given by~\eqref{Eq:alpha_Generalpower*2}--\eqref{Eq:alpha_GeneralExp*2}).

\bigskip
\noindent Plots of Figure~\ref{fig:strategy_gamma} below show that, in addition to the impact of the
roughness of asset volatility on the optimal allocation documented in~\cite{HanWong2020b,AichingerDesmettre2021}, the stabilizing function also affects the optimal portfolio allocation, and hence the hedging demand over time. The form of the control is consistent across different time scales.  The roughness dominates the very short end, while the  stabilizing function 
dictates the long-term behaviour. 

\medskip
\noindent While the rough volatility models capture high-frequency dynamics with remarkable accuracy, their lack of stationarity make them unsuitable for long-term investment analysis. This framework is then also essential for formulating a well-posed infinite-horizon investment problem, time-series analysis and long-term econometric estimation. 

\vspace{-.3cm}
\begin{figure}[H]
	\vspace{-.4cm}
	\centering
	\begin{subfigure}{0.88\linewidth}
		\centering
		\caption*{(a) $\gamma=0.2$ and $T=1$}
		\includegraphics[width=\linewidth]{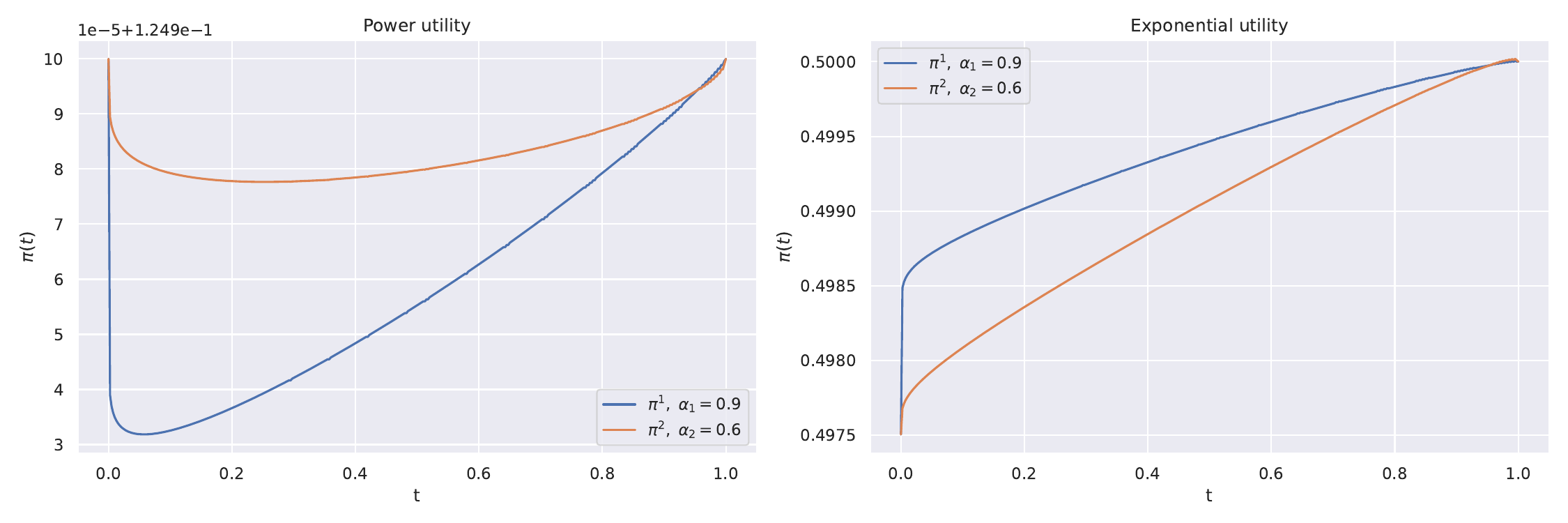}
	\end{subfigure}
\vspace{-.2cm}
	\begin{subfigure}{0.88\linewidth}
		\centering
		\caption*{(b) $\gamma=0.8$ and $T=5$}
		\includegraphics[width=\linewidth]{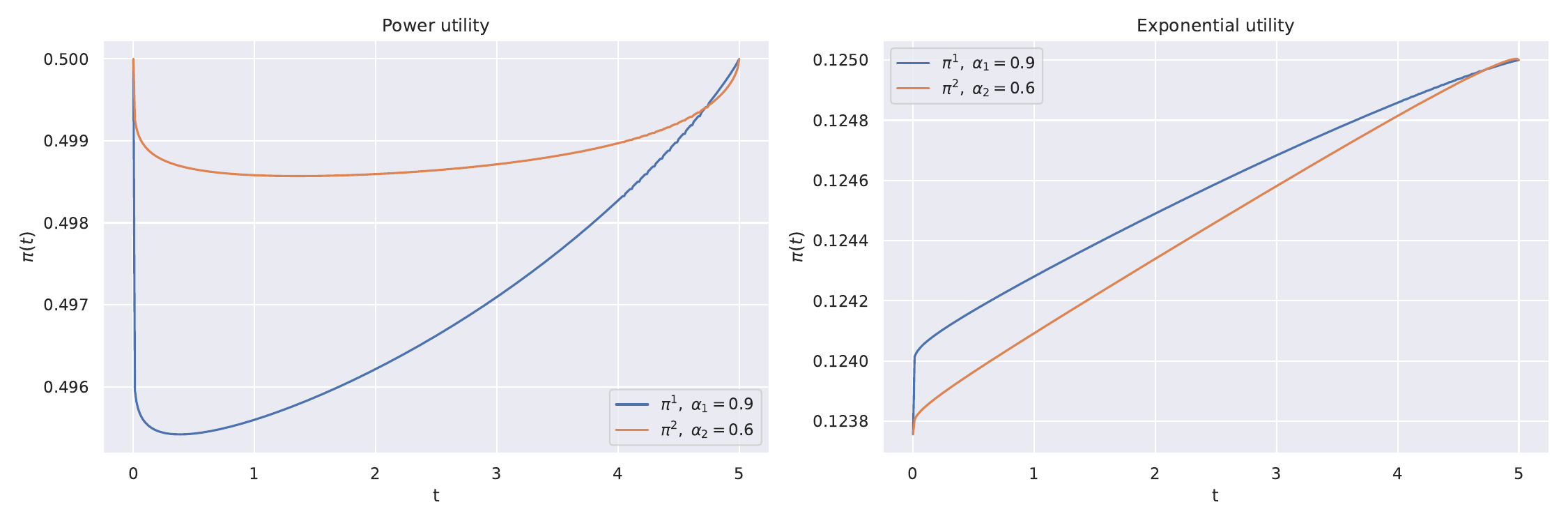}
	\end{subfigure}
	\vspace{-.2cm}
	\caption{\textit{Evolution of the optimal portfolio strategy for different levels of the risk aversion parameter $\gamma$ with \(T\in\{1,5\}\) for both the Power (left) and Exponential (right) utilities functions.}}
	\label{fig:strategy_gamma}
\end{figure}
\vspace{-.3cm}

\noindent 
Our illustrations further reveals that, the lower the investor's risk aversion, the smaller his optimal allocation in the power utility case. Additionally, the vector of amounts invested in each stock coincides at both the beginning and the end of the time horizon, as shown in the left panels of Figures~\ref{fig:strategy_gamma}
(Since $\varsigma^i(0)\psi^i(T)=\varsigma^i(T)\psi^i(0)=0$ for $i \in \{1,2\}$).

\noindent In contrast, the exponential utility suggests a completely different perspective: the optimal allocation becomes more moderated as the volatility paths become rougher and risk aversion increases.

 
 \noindent Economically, in line with~\cite{HanWong2020b}, these results indicate that distinct forms of risk aversion, captured by concave utility functions, display varying degrees of sensitivity in investment demand to the roughness of volatility.

\vspace{-.4cm}
\section{Proofs of the main results}\label{sect:proofMresult}
\vspace{-.3cm}
\noindent {$\rhd$ {\em Preliminaries}. As a first preliminary, we state the following Lemma which is an extension to the inhomogeneous setting~\eqref{VolSqrt_} of the result from \cite[Lemma~B.3]{AichingerDesmettre2021}, which is an enhancement of \cite[Appendix C]{AbiJaberMillerPham2021}. However, the proof does not rely on similar arguments, but rather on Novikov's type condition~\eqref{eq:assumption_novikov}. 
	\begin{Lemma}[Martingale property of stochastic exponentials]\label{lm: extended_m_AJ_lemma}
		Let $g_1$ and $g_2$ be two deterministic \(\R-\)valued bounded processes such that $g_1,g_2\in L^{\infty}([0,T],\R)$.
		Let $\left(\lambda, \Lambda\right) \in {L^{2}_{\F}([0,T], \R^d)\times L^{2}_{\F}([0,T], \R^d)}$ be two \(\R^d-\)valued stochastic processes, defined in~\eqref{eq:stocks} and~\eqref{eq:Lambda} respectively. 
		Let $B$, $W$ be two $d$-dimensional Wiener processes and \(\Sigma\) their corellation matrix,
		then under assumption~\ref{assm:gen}, for any \(\kappa\geq0\) the local martingale
	\begin{align*} Z_t := \mathcal E\Big(\int_0^tg_1(s) (\lambda_s + \kappa\Sigma\Lambda_s)^\T dB_s + g_2(s) \Lambda_s^\top dW_s \Big) 
	\end{align*} is a true martingale. 
	\end{Lemma}  
	\noindent \emph{Proof:} 
	Define \(M_t := \int_0^t g_1(s)(\lambda_s+\kappa\Sigma\Lambda_s)^\top dB_s
	+\int_0^t g_2(s)\Lambda_s^\top dW_s.\)
	Then $Z=\mathcal E(M)$ is a positive local martingale.
	We compute its quadratic variation. Using
	$d\langle B,W\rangle_t=\Sigma dt$, we obtain
	\begin{align*}
		\langle M\rangle_t
		&=
		\int_0^t g_1(s)^2|\lambda_s+\kappa\Sigma\Lambda_s|^2 ds
		+\int_0^t g_2(s)^2|\Lambda_s|^2 ds 
		+2\int_0^t g_1(s)g_2(s)(\lambda_s+\kappa\Sigma\Lambda_s)^\top \Sigma\Lambda_s ds\\
		&\leq \int_0^t 2 g_1(s)^2|\lambda_s+\kappa\Sigma\Lambda_s|^2 ds
		+\int_0^t g_2(s)^2\big(1+|\Sigma|^2\big)|\Lambda_s|^2 ds \leq  4 C \big(1+|\Sigma|^2\big)\int_0^t \left(|\lambda_s|^2+|\Lambda_s|^2\right)ds
	\end{align*}
where we use the elementary arithmetic mean-geometric mean (AM-GM) inequality $|ab|\leq (|a|^2+|b|^2)/2 $ and set \(C:=(1+\kappa^2)\cdot\max (1, \|g_1\|_{\infty,T}, \|g_2\|_{\infty,T})\geq 1\) owing to the boundedness of $g_1,g_2$.
	Hence, we have
		\begin{equation}
		\mathbb E\!\Big[\exp\!\Big(\tfrac12 \langle M\rangle_T\Big)\Big] \leq  \E \Big[ \exp\Big({a(2C) \int_0^T \Big( |\lambda_s|^2 + |\Lambda_s|^2\Big)  ds}\Big)    \Big]< \infty,
	\end{equation}
	thanks to condition~\eqref{eq:assumption_novikov} (since  Assumption~\ref{assm:gen} is in force) with constant \(a(2C) =2C \big(1+|\Sigma|^2\big)\). Consequently,
	Novikov's condition holds, so $\mathcal E(M)$ is a true martingale.\qed
	
\medskip
	\noindent 
	We recall here the definition of one dimensional \textit{functional resolvent of the first kind} of an integral kernel \(K\) in the
	terminology of \cite[Def. 5.5.1]{gripenberg1990} (which can also be found e.g. in \cite{abi2019affine}). Given $K \in {L}^1_{loc}(\R_+) $, a function $r$ belonging to ${L}^1_{loc}(\R_+)$ ( measure L on $\R_+$) is called \textit{functional resolvent of the first kind  of $K$} if
	\vspace{-.1cm}
	\begin{equation}\label{eq:Resolvent}
		(K \star r)(t) = (r \star K)(t) = 1,
	\end{equation}	
	for all $t \in \mathbb R_+$.
	\noindent Note that the same notion
	can be defined for higher dimensions and in matrix form as follows: Let $K \in L^1_{\mathrm{loc}}(\mathbb{R}_+,\mathbb{R}^{d\times d})$ and let $R$ be an $\mathbb{R}^{d\times d}$-valued measure on $\mathbb{R}_+$. 
	Then $R$ is called the \textit{resolvent of the first kind} of $K$ if
	\vspace{-.2cm}
	\begin{equation}\label{eq:resolvent_first_kind}
		K * R = R * K \equiv I,
	\end{equation}
	where $I$ denotes the $d$-dimensional identity matrix.
	\begin{Lemma}\label{lemma:Volt-H} Let $f, g: \R_+ \to \R$ be two  locally bounded Borel function, let $K \! \in L^1_{loc}(Leb_{\R_+})$ and let $r\!\in     L^1_{loc}(Leb_{\R_+})$ be its functional resolvent of the first kind. Then,
		\begin{enumerate}
			\item[$(a)$]
			The Volterra Equation  of the first kind
			$\forall\, t\ge 0, \quad f(t) = \int_0^t K(t-s) x(s) ds$
			(also reading $f= K*x$)  has   a solution given by:
			\vspace{-.4cm}
			\begin{equation}\label{eq:Resolventoperator-solu}
				\forall\, t\ge 0, \quad x(t) = \Big(\int_0^t f(t-s) r(ds)\Big)^\prime\; \quad \text{that is,} \quad x(t)=(f \star r)^\prime(t).
			\end{equation}	
			This solution is uniquely defined on $\R_+$ up to  $dt$-$a.e.$ equality.
			
			\item[$(b)$] The resolvent kernel \(r\) satisfies the following convolution identity:
			\vspace{-.3cm}
			\begin{equation}\label{eq:Resolventoperator}
				\forall\, t\ge 0, \quad \int_0^t (f \star r)^\prime(t-s) (K \star g)(s) ds =\int_0^t f(t-s)g(s)ds= (f \star g)(t) 
			\end{equation}
		\end{enumerate}
	\end{Lemma}
	\vspace{-.2cm}
	\noindent We provide a proof of this classical result in Appendix~\ref{subsect:proofLemmaVolt-H} for the reader's convenience.
	
    \medskip
	\noindent {\bf Proof of Theorem~\ref{Thm:powerUtilityGeneral}:}
	
	\smallskip
	\noindent {\sc Step~1} (\textit{Proof of Admissibility:})
	We begin by proving the admissibility of the candidate in equation~\eqref{eq:optcandPower} for the optimal portfolio strategy $\alpha^*$, and, en route, 
	we deduce the bound in equation~\eqref{eq:boundPower}.
	
	\smallskip
	\noindent In order to show that $\alpha^*$ is admissible, we have to show that (a), (b), (c) of Definition~\ref{Def:adm} hold. Part (a) is true because~\eqref{Eq:wealth} has a unique solution~\eqref{eq:wealthPowerSol} in terms of $(S,V,B)$.   
	For part (b) it suffices to show that \(\mathbb{E}[\operatorname{sup}_{t\in[0,T]}|X_t^{\alpha^*}|^p]<\infty\) for every \(p>0\).
	Inserting the explicit solution of the control problem into~\eqref{Eq:wealth}
	and using  It\^o's lemma (or equivalently into~\eqref{eq:wealthPowerSol}), the unique continuous solution of the controlled wealth process $X^{\alpha^*}$ is given by :
	\vspace{-.3cm}
	\begin{equation}
		X_t^{\alpha^{*}}
		=
		x_0 \exp\!\Big(
		\int_0^t \Big(r(s) + \alpha_s^{*\top} \lambda_s - \tfrac12 |\alpha_s^{*}|^2\Big)\,ds
		+
		\int_0^t \alpha_s^{*\top} \, dB_{s}
		\Big).
		\vspace{-.1cm}
	\end{equation}
	Observe by virtue of~\eqref{eq:assumption_novikov}, that the Dol\'eans-Dade exponential $\mathcal E\left( \int_0^{\cdot}\alpha_s^{*\top} dB_s\right)$ satisfies Novikov's condition, and is therefore a true martingale (see also Lemma~\ref{lm: extended_m_AJ_lemma}). 
	By {virtue of the Cauchy-Schwarz inequality}, and then Doob's maximal inequality in the second line, for every \(p>1\) 
	\vspace{-.2cm}
	\begin{align*}
		&\, \E \Big[ \sup_{ t \in [0, T]} | X^{\alpha^*}_t |^p \Big] \leq x_0^p\E \Big[ \sup_{t \in [0,T]}  \big|  e^{\int_0^t {\left(r(s) + \lambda_s^\T \alpha_s^{*} \right)}ds} \big|^{2p}  \Big]^\frac12 \E \Big[ \sup_{ t \in [0, T]} \Big|  \exp \Big(- \tfrac12 \int^t_0 \left|\alpha_s^{*}\right|^2 ds + \int^t_0  \alpha_s^{*\top} \, dB_{s} \Big) \Big|^{2p}\Big]^\frac12 \\
		& \hspace{2.3cm}\leq x_0^p e^{ p \int_0^Tr(s) ds} \Big(\frac{2p}{2p-1}\Big)^{p} \E \Big[   e^{ 2p \int_0^T | \lambda_s^\T \alpha_s^{*} | ds}  \Big]^{\frac12} \E \Big[ \exp \Big(- p\int^T_0 \left|\alpha_s^{*}\right|^2 ds + 2p \int^T_0 \alpha_s^{*\top} \, dB_{s} \Big) \Big]^{\frac12}.
	\end{align*}
	The first term is finite on the first hand by condition~\eqref{eq:assumption_novikov} with constant \(a(\frac{p}{1-\gamma}) = \frac{p}{1-\gamma}\left(2 + |\Sigma| \right)\)  and owing to the elementary arithmetic mean-geometric mean (AM-GM) inequality $|ab|\leq (|a|^2+|b|^2)/2 $
	\vspace{-.2cm}
	\begin{equation}
		\E \Big[   e^{ 2p  \int_0^T | \lambda_s^\T \alpha_s^{*} | ds}  \Big] \leq  \E \left[ \exp\left({a(\frac{p}{1-\gamma}) \int_0^T \left( |\lambda_s|^2 + |\Lambda_s|^2\right)  ds}\right)    \right]< \infty,
	\end{equation}
	and, on the other hand, the second term is also finite since {by virtue of  Hölder's inequality},
	\vspace{-.2cm}
	\begin{align*}
		\E \Big[ e^{- \int^T_0 p \left|\alpha_s^{*}\right|^2 ds + \int^T_0 2p \alpha_s^{*\top} \, dB_{s} } \Big] 
		& \leq   \left(\E \Big[ e^{(8p^2 - 2p) \int^T_0 \left|\alpha_s^{*}\right|^2 ds} \Big] \right)^{1/2}\left(\E\Big[ e^{- 8 p^2 \int^T_0 \left|\alpha_s^{*}\right|^2 ds + 4p \int^T_0 \alpha_s^{*\top} \, dB_{s} } \Big] \right)^{1/2}\\
		&\leq  \left(   \E \left[ e^{a(\frac{p}{1-\gamma}) \int_0^T \left(|\lambda_s|^2 + |\Lambda_s|^2 \right)  ds} \right]    \right)^{1/2} \times 1 < \infty.
	\end{align*}
	with constant \(a(\frac{p}{1-\gamma}) =2 (8\big(\frac{1}{1-\gamma}\big)^2 {- 2\frac{p}{1-\gamma}}) \left( 1  + |{\Sigma}|^2  \right)\) and where we used Jensen's inequality 
	to bound 
	\vspace{-.2cm}
	\begin{equation}
		\left|\alpha_s^{*}\right|^2 = \big(\frac{1}{1-\gamma}\big)^2|\lambda_s +  \Sigma\Lambda_s|^2 \leq 2\big(\frac{1}{1-\gamma}\big)^2(|\lambda_s|^2 +  |\Sigma\Lambda_s|^2) \leq 2\big(\frac{1}{1-\gamma}\big)^2 (1 + |\Sigma|^2)(|\lambda_s|^2 +  |\Lambda_s|^2).
	\end{equation}
	then noticing that \( \big(\frac{1}{1-\gamma}\big)^2(8p^2 {- 2p})\leq  8\big(\frac{p}{1-\gamma}\big)^2 {- 2\frac{p}{1-\gamma}}\) for all \(p>1, \, \gamma\in (0,1)\), together with condition~\eqref{eq:assumption_novikov} and Novikov's condition to the Dol\'eans-Dade exponential $\mathcal{E}(4p \int_0^{{\cdot}} \alpha_s^{*\T }dB_s)$. 
	This leads to Part (c).
	$\E \Big[ \sup_{ t \in [0, T]} | X^{\alpha^*}_t |^p \Big] < \infty$ is proved. It becomes straightforward to verify $\alpha^*$ is admissible. We are left to prove that $\alpha^* \in L^2_{\F}([0,T], \R^d)$. We have 
	\vspace{-.3cm}
	\begin{align*}
		\E \left[ \int_0^{T} |\c_s^*|^2 ds \right] & { =} \E \left[ \int_0^{T} \big(\frac{1}{1-\gamma}\big)^2|\lambda_s + \Sigma\Lambda_s|^2 ds \right] \leq   \E\left[ 2 \big(\frac{1}{1-\gamma}\big)^2 (1 + |\Sigma|^2) \int_0^T \left(|\lambda_s|^2 +  |\Lambda_s|^2\right) ds \right] < \infty,
	\end{align*}
	where the last term is finite due to condition~\eqref{eq:assumption_novikov} and the inequality $|z|^q\leq c_{q}e^{|z|},\; \forall \, q\geq1$.
	
	\smallskip
	\noindent {\sc Step~2} (\textit{Proof of the equality (1):})
	Let's us consider the Problem~\eqref{obj_power} with an arbitrary strategy $\alpha \in \mathcal{A}$ from Definition~\ref{Def:adm}.
	To ease notations, we set $h_t = \lambda_t + \Sigma \Lambda_t $. For any admissible strategy $\alpha \in \mathcal A$, using the Riccati BSDE~\eqref{eq:gamma_PowerTilheston} in Proposition~\ref{prop:ExpoPower_riccati_2} together with It\^o's lemma yield: 
	\vspace{-.2cm}
	\begin{align*}
		d \big(\frac{(X^\alpha_t)^\gamma}{\gamma}\Gamma_t \big) \;&=\frac{(X^\alpha_t)^\gamma}{\gamma}\Gamma_t \Big( D_t(\alpha_t) dt + \gamma \c_t^\T dB_t +  \Lambda_t^\top dW_t\Big).
	\end{align*}
	where the drift factor takes the form  \(D_t(\alpha_t) =  \frac{\gamma(\gamma-1)}{2}  \alpha_t^\T \alpha_t +\gamma  \alpha_t^\T h_t - \frac{\gamma}{2(1-\gamma)} h_t^\T h_t.\)  
	As a consequence, {using $\Gamma_T=1$}, we get
	\vspace{-.2cm} 
	\begin{equation}\label{eq:PowerUlti}
		\frac{(X^\alpha_T)^\gamma}{\gamma} =\frac{\Gamma_0 x_0^{\gamma}}{\gamma}e^{\int^T_0 D_s(\alpha_s) ds} F_T^\alpha.
	\end{equation}
	where \(F_t^\alpha\) is an $(\mathbb{F},\mathbb{P})$-local martingale as  $(\alpha,\Lambda)$ are in $L^{2,{loc}}_{\mathbb F}([0,T])^2$ (Dool\'{e}ans-Dade exponential) and write
	\vspace{-.2cm}
	\begin{equation*}
		F_t^\alpha= \mathcal E\Big(\int_0^t \gamma \c_s^\T dB_s + \Lambda_s^\top dW_s \Big) =\operatorname{exp}\Big(- \frac{1}{2} \int^t_0 \Big(\gamma^2 \left|\alpha_s\right|^2 + \left|\Lambda_s\right|^2+ 2\gamma \c_s^\T(\Sigma\Lambda_s)\Big)ds +  \gamma \int_0^t \c_s^\T dB_s + \int_0^t\Lambda_s^\top dW_s\Big).
	\end{equation*}
	At this stage, for $\alpha_t = \alpha^*_t$, $F_t^{\alpha^*}$ is a \(\P\)-martingale with expectation \(1 \) by Lemma~\ref{lm: extended_m_AJ_lemma} with $\kappa=1$, 
	 $g_{1}(t)= \frac{\gamma}{1-\gamma} $ and $g_{2}(t)=1$ for every \(t \in [0,T]\). Consequently, inserting the candidate~\eqref{eq:optcandPower} for the optimal strategy \(\alpha^*\) in the above Equation~\eqref{eq:PowerUlti}, and observing that \( D_s(\alpha_s^*) = 0 \; \forall\, s\in [0,T]\), then taking the expectation, it is straightforward that:
	 \vspace{-.2cm}
	\[\E_{x_0,V_0}{\big[\frac{(X^{\alpha^*}_T)^\gamma}{\gamma}  \big]}=\frac{\Gamma_0 x_0^{\gamma}}{\gamma} \E_{x_0,V_0}{\big[F_T^{\alpha^*} \big]}=\frac{ x_0^{\gamma}}{\gamma} \Gamma_0.\]
	where the last inequality comes from the martingality of \(F_T^{\alpha^*}\) and the desired result about the first part of the proof (\textit{equality (1)}) is completed. 
	It remains to show the \textit{inequality (2)} for arbitrary admissible portfolio strategies \(\alpha\in\mathcal{A}\) defined in~\eqref{eq:StrongAdmStratPower}.
	
	\smallskip
	\noindent {\sc Step~3} (\textit{Proof of the inequality (2):})
	Recall that the SDE for the wealth process~\eqref{Eq:wealth} can be solved explicitly so as to admit the representation
	\vspace{-.2cm}
	\begin{equation}
		X_T^\alpha
		=
		X_0^\alpha \exp\!\left(
		\int_0^T \Big(r(s) + \alpha_s^\top \lambda_s - \tfrac12 \left|\alpha_s\right|^2\Big)\,ds
		+
		\int_0^T \alpha_s^\top \, dB_{s}
		\right)\quad \text{with}\quad X_0^\alpha =x_0.
	\end{equation}
	Since by assumption, $\alpha \in \mathcal{A}$ redefined in~\eqref{eq:StrongAdmStratPower}, we can introduce a new probability measure $\Q$ with Radon-Nikodym density  or derivative at \(\mathcal{F}_T\) 
	\vspace{-.2cm}
	\begin{equation}\label{eq:density_Q1}
		{Z}_T:=\frac{d{\Q}}{d\mathbb{P}}|_{\mathcal{F}_T}=\exp\Big(\gamma\int_0^T \alpha_s^{\top} \, dB_{s}-\frac{\gamma^2}{2}\int_0^T \left|\alpha_s\right|^2 ds\Big),
	\end{equation}
	thanks to Novikov's condition. (In fact, since \(\gamma\in (0,1)\), the function \(x \mapsto x^{\gamma^2}\) is concave. Jensen's inequality together with ~\eqref{eq:StrongAdmStratPower} implies that \(\E\big[\exp\big(\frac{\gamma^2}{2}\int_0^T \left|\alpha_s\right|^2 ds\big)\big] < \infty\) so that $(Z_t)_{t\in[0,T]}$ is a martingale.)	 
	By Girsanov's theorem, we may write:
	\vspace{-.2cm}
	\begin{equation*}
		\begin{aligned}
			&\, (x_0)^{-\gamma}\E_{ x_0,V_0}\Big[(X_T^{\alpha})^{\gamma}\Big]
			=\E_{x_0,V_0}^{\Q}\Big[\exp\Big(\int_0^T \gamma r(s) ds+\gamma\int_0^T \big(\alpha_s^{\top} \lambda_s + \frac{\gamma-1}2 \left|\alpha_s\right|^2\big)ds \Big)\Big],\\
			&\hspace{3.6cm}=\exp\Big(\gamma\int_0^T r(s) ds \Big)\E_{x_0,V_0}^{\Q}\Big[\exp\Big(\int_0^T F_s ds \Big)\Big]
		\end{aligned}
	\end{equation*}
	with the real-valued deterministic process \(F_s\)  given by
	\vspace{-.2cm}
	\begin{align*}
		\forall\, s\in[0,T],\quad	F_s^\alpha=\gamma\alpha_s^{\top} \lambda_s + \frac{\gamma(\gamma-1)}2 \left|\alpha_s\right|^2
		=F_s^{\alpha^*}+\gamma \hat{\alpha}_s^{\top} \lambda_s +\gamma(\gamma-1) \alpha_s^{*\top} \hat{\alpha}_s +\frac{\gamma(\gamma-1)}2\left|\hat{\alpha}_s\right|^2
	\end{align*}
	where we write the arbitrary admissible strategy $(\alpha_t)_{t\in[0,T]}$ in terms of the optimal strategy \(\alpha^*_t=\frac{1}{1-\gamma} \left( \lambda_t +  \Sigma \Lambda_t \right)\) for every \(t\in[0,T]\) and some remainder $(\hat{\alpha}_t)_{t\in[0,T]}$ i.e. $\alpha_t=\alpha_t^* +\hat{\alpha}_t$ for every \(t\in[0,T]\).
	Note that, on the first hand, substituting the optimal strategy \(\alpha^*\) in~\eqref{eq:optcandPower} yields \(\forall\, s\in[0,T]\)
	\vspace{-.1cm}
	\begin{equation*}F_s^{\alpha^*} =\gamma\big(\alpha_s^{*\top} \lambda_s + \frac{\gamma-1}2 \left|\alpha_s^*\right|^2\big) =  \frac{\gamma}{2(1-\gamma)} \big(\left| \lambda_s\right|^2 -\left| \Sigma \Lambda_s \right|^2\big) =\frac{\gamma}{2(1-\gamma)}\sum_{i=1}^d \Big( \theta_i^2  - \rho_i^2 \nu_i^2 (\varsigma^i(s)\psi^i(T-s))^2 \Big) V^i_s
	\vspace{-.2cm} 
	\end{equation*}
	and on the other hand, recalling that $\alpha_s := \sqrt{\diag(V_s)}^\top \pi_s$ and defining $\hat{\pi}:=\pi-\pi^*$ so that $\pi_s=\pi_s^* +\hat{\pi}_s$, with $\pi_s^*=\frac{1}{1-\gamma} \left( \theta +  \Sigma \nu\varsigma(s)\psi(T-s) \right)$ so that \(\pi_{s,i}^{*}=\frac{1}{1-\gamma}(\theta_i+\rho_i\nu_i\varsigma^i(s)\psi^i(T-s))\) for \(i=1,\ldots,d\) and consequently, \(\forall\, s\in[0,T],\)
	\vspace{-.2cm}
	\begin{align*}
		F_s^\alpha&=F_s^{\alpha^*}+ \sum_{i=1}^d \Big(\gamma \theta_i \hat{\pi}_{s,i} +\gamma(\gamma-1)  \pi_{s,i}^{*} \hat{\pi}_{s,i} +\frac{\gamma(\gamma-1) }2 \hat{\pi}_{s,i}^2\Big) V^i_s\\
		&= \sum_{i=1}^d \Big(\frac{\gamma}{2(1-\gamma)}\Big( \theta_i^2  - \rho_i^2 \nu_i^2 (\varsigma^i(s)\psi^i(T-s))^2 \Big) - \gamma \rho_i\nu_i\varsigma^i(s)\psi^i(T-s) \hat{\pi}_{s,i} +\frac{\gamma(\gamma-1) }2 \hat{\pi}_{s,i}^2 \Big) V^i_s
	\end{align*}
	Using Equation~\eqref{eq:RiccatiPowerTilpsi1} reading \( \frac{\gamma}{2(1-\gamma)}\theta_i^2 + F_i(s,\psi(T-s)) = (\psi^i \star r^i)^\prime(T-s) \) owing to Lemma~\ref{lemma:Volt-H}~$(a)$ for \(i=1,\ldots,d\), where \(r^i\) the resolvent of the first kind of the integral kernel $K_i$, one gets:
	\vspace{-.2cm}
	\begin{align}
		\forall\, s\in[0,T],\;	F_s&=\sum_{i=1}^d  (\psi^i \star r^i)^\prime(T-s)V^i_s- \Big(F_i(s,\psi(T-s)) + \frac{\gamma\rho_i^2 \nu_i^2}{2(1-\gamma)} (\varsigma^i(s)\psi^i(T-s))^2\Big)V^i_s \label{eq:Fs1}\\
		&\hspace{1.5cm}+ \sum_{i=1}^d \Big( -\gamma \rho_i\nu_i\varsigma^i(s)\psi^i(T-s) \hat{\pi}_{s,i} +\frac{\gamma(\gamma-1) }2 \hat{\pi}_{s,i}^2 \Big) V^i_s\nonumber
	\end{align}
	Under the probability measure $\Q$ defined in \eqref{eq:density_Q1}, the processes \(	\widehat{B}_{t}:=B_{t}-\gamma\int_0^t\alpha_s ds\) and
	\vspace{-.2cm}
	\[ \widehat{W}_t :=  \Sigma  \widehat{B}_t + \sqrt{I-\Sigma^\T \Sigma} B^{\perp}_t = W_{t}-\gamma\int_0^t\Sigma\alpha_s ds = W_{t}-\gamma\int_0^t\Sigma\alpha^*_s ds-\gamma\int_0^t\Sigma\hat{\alpha}_s ds ,
	\]
	are standard Wiener processes thanks to the Cameron--Martin--Girsanov and Girsanov theorem respectively theorem.
	The dynamics of the variance process $V$ under $\Q$ can thus be written as
	\vspace{-.3cm}
	\begin{align*}
		&\, V_t^i = g_0^i(t) + \int_0^t K_i(t-s) H_s^i ds  + \int_0^t K_i(t-s) \nu_i \varsigma^i(s)\sqrt{V_s^i}d\widehat{W}^i_s,\quad i=1,\ldots,d; \quad\text{with}\\
		& H_s^i := (D V_s)_i +\gamma \rho_i \nu_i \varsigma^i(s) (\pi_{s,i}^{*}+ \hat{\pi}_{s,i})V_s^i, \quad i=1,\ldots,d. \\
		&\hspace{.6cm}=  \frac{\gamma}{1-\gamma} \theta_i \rho_i \nu_i \varsigma^i(s)V_s^i + (D V_s)_i + \frac{\gamma}{1-\gamma}(\rho_i\nu_i\varsigma^i(s))^2\psi^i(T-s)V_s^i +\gamma \rho_i \nu_i \varsigma^i(s) \hat{\pi}_{s,i}V_s^i
	\end{align*}
	Now, we insert the dynamics of $V^i$ into the expression
	$\int_t^T(\psi^i \star r^i)^\prime(T-s)V^i_sds$ and simplify using Lemma~\ref{lemma:Volt-H}~$(b)$, where we reads \(H^i\) and the stochastic term \textit{pathwise} (i.e.\ $\omega$ by $\omega$) as for example $g(s)=H^i_s(\omega)$. Its boils down that:
	\vspace{-.3cm}
	\begin{align*}
		&\,\int_t^T\sum_{i=1}^d (\psi^i \star r^i)^\prime(T-s)V^i_s ds = \sum_{i=1}^d \int_t^T(\psi^i \star r^i)^\prime(T-s)g_0^i(s) ds +  \int_t^T \psi^i (T-s)\big(H_s^i ds + \nu_i \varsigma^i(s)\sqrt{V_s^i}d\widehat{W}^i_s\big)\\
		&\hspace{1.5cm}=\sum_{i=1}^d \int_t^T\big(\frac{\gamma}{2(1-\gamma)}\theta_i^2 + F_i(s,\psi(T-s))\big)g_0^i(s) ds + \sum_{i=1}^d \int_t^T \psi^i (T-s)\big(H_s^i ds + \nu_i \varsigma^i(s)\sqrt{V_s^i}d\widehat{W}^i_s\big).
	\end{align*}
	Noting by a change of variables that
	\vspace{-.2cm}
	\begin{equation}\label{eq:chtVar} \sum_{j=1}^d \psi^j(T-s) (D V_s)_j = \sum_{j=1}^d \psi^j(T-s) \sum_{i=1}^d D_{ji} V_s^i = \sum_{i=1}^d V_s^i \sum_{j=1}^d D_{ji} \psi^j(T-s) = \sum_{i=1}^d (D^\top \psi)_i (T-s)V_s^i; 
	\end{equation}
	one gets using Equation~\eqref{eq:RiccatiPowerTilpsi2} 
	\vspace{-.2cm}
	\begin{align*}
		\sum_{i=1}^d \int_0^T \psi^i (T-s)H_s^i ds &= \sum_{i=1}^d \int_0^T \Big(F_i(s,\psi^i (T-s))+\big(\frac{\gamma}{2(1-\gamma)}\rho_i^2- \frac {1} 2\big) \nu_i^2 (\varsigma^i(s)\psi^i (T-s))^2 \Big)V_s^i ds \\
		&+\gamma\sum_{i=1}^d \int_0^T \rho_i \nu_i \varsigma^i(s) \psi^i (T-s) \hat{\pi}_{s,i} V_s^i ds. 
	\end{align*}
	Replacing back all in~\eqref{eq:Fs1}, some terms cancel out and we end up with
	\vspace{-.2cm}
	\begin{align*}
		&\int_0^T F_s \, ds 
		= \sum_{i=1}^d \int_0^T 
		\Big( 
		\frac{\gamma}{2(1-\gamma)} \theta_i^2 
		+ F_i(s,\psi(T-s))
		\Big) g_0^i(s) \, ds + \sum_{i=1}^d \int_0^T \frac{\gamma(\gamma-1)}2 \hat{\pi}_{s,i}^2  V^i_s\, ds   \\
		&\hspace{.02cm}\qquad + \sum_{i=1}^d \int_0^T 
		\Big(
		- \frac{\nu_i^2}{2} 
		\big(\varsigma^i(s)\psi^i(T-s)\big)^2 V_s^i\, ds 
		+ \nu_i \varsigma^i(s)\psi^i(T-s)
		\sqrt{V_s^i}\, d\widehat{W}^i_s
		\Big)= \sum_{i=1}^d \int_0^T \frac{\gamma(\gamma-1)}2 \hat{\pi}_{s,i}^2  V^i_s\, ds\\
		&+ \sum_{i=1}^d \int_0^T 
		\Big( 
		\frac{\gamma}{2(1-\gamma)} \theta_i^2 
		+ F_i(s,\psi(T-s))
		\Big) g_0^i(s) \, ds -\frac12\int_0^T |\Lambda_s|^2\,ds + \int_0^T \Lambda_s^\T d\widehat{W}^i_s.
	\end{align*}
	Hence we obtain
	\vspace{-.3cm}
	\begin{align}
		&\,	x_0^{-\gamma}\E_{x_0,V_0}[{(X_T^{\alpha})^{\gamma}}]=\exp\Big(\gamma\int_0^T r(s) ds \Big)\E_{x_0,V_0}^{\Q}\Big[\exp\Big(\int_0^T F_s^\alpha ds \Big)\Big], \nonumber\\
		&\hspace{3.5cm}= \exp\Big( \gamma\int_0^T r(s) ds +  \sum_{i=1}^d\int_0^T  \big(\frac{\gamma \theta_i^2}{2(1-\gamma)}  + F_i(s,\psi(T-s))\big) g^i_0(s) ds \Big) \nonumber\\
		&\hspace{2cm} \times\E_{x_0,V_0}^{\Q}\Big[\exp\Big(
		-\frac12\int_0^T |\Lambda_s|^2\,ds + \int_0^T \Lambda_s^\T d\widehat{W}^i_s + \sum_{i=1}^d \int_0^T \frac{\gamma(\gamma-1)}2 \hat{\pi}_{s,i}^2  V^i_s\, ds\Big)\Big].\label{eq:expAllExpo1}
	\end{align}
	\smallskip
	\noindent {\bf Remark:} (\textit{Alternative Proof of the equality (1):})	
	At this stage, note that if the admissible strategy \(\alpha\) is optimal, i.e. \(\alpha=\alpha^*\) so that \(\hat{\pi}\equiv0\) and \(\Q\equiv\Q^*\) in ~\eqref{eq:density_Q1}, then Equation~\eqref{eq:expAllExpo1} becomes:
	\vspace{-.2cm}
	\begin{equation*}
		\begin{aligned}
			{}
			& x_0^{-\gamma}\E_{x_0,V_0}[{(X_T^{\alpha^*})^{\gamma}}]=\exp\Big( \gamma\int_0^T r(s) ds +  \sum_{i=1}^d\int_0^T  \big(\frac{\gamma \theta_i^2}{2(1-\gamma)}  + F_i(s,\psi(T-s))\big) g^i_0(s) ds \Big)\\
			&\hspace{3.6cm} \times\E_{x_0,V_0}^{\Q^*}\Big[\exp\Big(-\frac12\int_0^T |\Lambda_s|^2\,ds + \int_0^T \Lambda_s^\T d\widehat{W}^{*i}_s\Big)\Big].
		\end{aligned}
	\end{equation*}
	Note that, the stochastic exponential is a true $\Q^*$- martingale with expectation $1$ by Lemma~\ref{lm: extended_m_AJ_lemma} with $g_{1} = 0$ and $g_{2}(s) = 1$ for every \(s \in [0,T]\). Consequently, we get
	\[
	\E_{x_0,V_0}\Big[{\frac{1}{\gamma}(X_T^{\alpha^*})^{\gamma}}\Big]=\frac{x_0^\gamma}{\gamma}\exp\Big( \gamma\int_0^T r(s) ds +  \sum_{i=1}^d\int_0^T  \big(\frac{\gamma \theta_i^2}{2(1-\gamma)} + F_i(s,\psi(T-s))\big) g^i_0(s) ds \Big).
	\]
	This again completes the first part of the proof (\textit{equality (1)}).
	
	\medskip
	\noindent Coming back to the proof of the \textit{inequality (2)} for arbitrary admissible portfolio strategies \(\alpha\in\mathcal{A}\), from Equation~\eqref{eq:expAllExpo1}, since \(V_s\) is positive
	definite and \(\gamma\in(0,1)\), the term \(\sum_{i=1}^d \int_0^T \frac{\gamma(\gamma-1)}2 \hat{\pi}_{s,i}^2  V^i_s\, ds=\sum_{i=1}^d \int_0^T \frac{\gamma(\gamma-1)}2 \hat{\alpha}_{s,i}^2\, ds\) has to be less than or equal to \(0\) so that
	\vspace{-.2cm}
	\begin{equation*}
		\begin{aligned}
			{}
			& x_0^{-\gamma}\E_{x_0,V_0}[{(X_T^{\alpha})^{\gamma}}]\leq\exp\Big( \gamma\int_0^T r(s) ds +  \sum_{i=1}^d\int_0^T  \big(\frac{\gamma \theta_i^2}{2(1-\gamma)}  + F_i(s,\psi(T-s))\big) g^i_0(s) ds \Big)\\
			&\hspace{2.2cm} \times\E_{x_0,V_0}^{\Q}\Big[\exp\Big(-\frac12\int_0^T |\Lambda_s|^2\,ds + \int_0^T \Lambda_s^\T d\widehat{W}^{i}_s\Big)\Big].
		\end{aligned}
	\end{equation*}
	Since the stochastic exponential
	is a \(\Q\)-martingale with expectation \(1\) still by Lemma~\ref{lm: extended_m_AJ_lemma}, we finally obtain	that  for every \(\alpha\in\mathcal{A}\)
	\vspace{-.2cm}
	\[
	\E_{x_0,V_0}\Big[{\frac{1}{\gamma}(X_T^{\alpha})^{\gamma}}\Big]\leq \frac{x_0^\gamma}{\gamma}\exp\Big( \gamma\int_0^T r(s) ds +  \sum_{i=1}^d\int_0^T  \big(\frac{\gamma \theta_i^2}{2(1-\gamma)} + F_i(s,\psi(T-s))\big) g^i_0(s) ds \Big).
	\]
	which completes the proof.  \hfill $\Box$

	\medskip
	\noindent{\bf Proof of Theorem~\ref{Thm:ExpoUtilityGeneral}}
	\noindent {\sc Step~1} (\textit{Proof of the Martingale optimality principle for \(J^\alpha\).})
	We show that $J_t^{\alpha}$ in~\eqref{eq:ansatzExpo} fulfills the martingale optimality principle~\ref{def:Martopt}. 
	For the first condition, note that $\Gamma_T=1$ and hence $J_T^{\alpha}=-\frac{1}{\gamma}\exp\left(-\gamma  X^\alpha_T\right) = U(X^\alpha_T)$. Since $\Gamma_0$ is a constant independent of $ \alpha \in \cA$, $J^\alpha_0 = -\frac{1}{\gamma}\exp\left(-\gamma  e^{\int_{0}^{T} r(s)ds}x_0\right) \Gamma_0$ is a constant independent of $ \alpha \in \cA$ and thus the second condition is also satisfied.
	In order to show that the third condition is also fulfilled, we apply Itô's formula on $J_t^{\alpha}$ defined in~\eqref{eq:ansatzExpo}. 
	Let's us consider the Problem~\eqref{Expobj} with an arbitrary strategy $\alpha \in \mathcal{A}$.
	To ease notations, we set \(U_t :=-\gamma  e^{\int_{t}^{T} r(s)ds}X^\alpha_t+ \log(\Gamma_t)\) and $h_t := \lambda_t + \Sigma \Lambda_t $. Note that
	 \[J^\alpha_t := -\frac{1}{\gamma}\exp(U_t)\quad\text{and}\quad dJ_{t}^{\c} = J_{t}^{\c} \Big( d U_t + \frac 1 2 d\langle U \rangle_t \Big)\]
	 For any admissible strategy $\alpha \in \mathcal A$, we write by It\^o's lemma:
	\begin{align*}
		dU_t &= \Big(\gamma r(s) e^{\int_{t}^{T} r(s)ds}X^\alpha_t - \gamma e^{\int_{t}^{T} r(s)ds}\big( r(t) X^{\alpha}_t  + \alpha_t^\T \lambda_t \big) \Big) dt - \gamma e^{\int_{t}^{T} r(s)ds} \alpha_t^\T dB_t  + d\log(\Gamma_t) \\
		&=  \big(  - \gamma e^{\int_{t}^{T} r(s)ds} \alpha_t^\T \lambda_t + f(Y_t,\Lambda_t, U_t)-\frac{1}{2} \left| \Lambda_t \right|^2 \big) dt  + \Lambda_t^\top dW_t
		 - \gamma e^{\int_{t}^{T} r(s)ds} \alpha_t^\T dB_t 
	\end{align*}
	where \(f\) is the driver or drift term in Equation~\eqref{eq:gamma_ExpTilheston}. Consequently, by It\^o's product rule, combined with the property of \(\Gamma\) in Equation~\eqref{eq:gamma_ExpTilheston} of Proposition~\ref{prop:ExpoExp_riccati_2}, one may write
	\begin{align*}
		dJ_t^{\c} &= \; J_{t}^{\c} \big( - \gamma e^{\int_{t}^{T} r(s)ds} \alpha_t^\T \lambda_t + f(Y_t,\Lambda_t, U_t)-\frac{1}{2} \left| \Lambda_t \right|^2 + \frac{\gamma^2}{2} e^{\int_{t}^{T} 2r(s)ds} \c^\T_t \c_t -\gamma e^{\int_{t}^{T} r(s)ds} \c^\T_t \big( \Sigma \Lambda_t\big) \\
		&+ \frac{1}{2} \left| \Lambda_t \right|^2 \big) dt + J_{t}^{\c} \big(- \gamma e^{\int_{t}^{T} r(s)ds} \c_t^\T dB_t +   \Lambda_t^\top dW_t \big)= J_{t}^{\c} \Big( D_t(\alpha_t) dt - \gamma  e^{\int_{t}^{T} r(s)ds} \c_t^\T dB_t +  \Lambda_t^\top dW_t \Big).
	\end{align*}
	where the drift factor takes the form: 
	\begin{align*}
		D_t(\alpha) = & \frac{\gamma^2}{2} e^{\int_{t}^{T} 2r(s)ds}  \c^\T \c -\gamma  e^{\int_{t}^{T} r(s)ds}  \c^\T h_t  + \frac{1}{2} h_t^\T h_t.
	\end{align*}
	Differentiating \(D_t(\alpha) \) with respect to \(\alpha\) and checking the second order condition, the maximizer \(\alpha^*_t:=\frac1\gamma  e^{-\int_{t}^{T} r(s)ds} h_t =\frac1\gamma  e^{-\int_{t}^{T} r(s)ds} \left( \lambda_t +   \Sigma \Lambda_t \right)\) for every \(t\in[0,T]\) that is the strategy given by Equation~\eqref{Eq:alpha_GeneralExpo*}.
	Evaluating the drift factor \(D_t \) at \(\alpha^*_t\) show that \(D_t(\alpha^*_t)\) vanishes to \(0\). 
	Note $D_t(\alpha)$ is a quadratic function on $\alpha$ and $\gamma - 1 < 0$. As $D_t(\alpha^*_t) =0$, then $D_t(\alpha_t) \geq 0$ for any admissible strategy \(\alpha\).
	Moreover, solving the stochastic differential equation for $J_t^{\alpha}$ yields
	\begin{align}\label{eq:JExpoU}
	&\hspace{2cm}	J_t^{\alpha}= -\frac{\Gamma_0}{\gamma}\exp\left(-\gamma  e^{\int_{0}^{T} r(s)ds}x_0\right)e^{\int^t_0 D_s(\alpha_s) ds} F_t^\alpha\\
	&\text{where}\hspace{4cm} F_t^\alpha= \mathcal E\Big(\int_0^t -\gamma  e^{\int_{s}^{T} r(u)du} \c_s^\T dB_s +  \Lambda_s^\top dW_s  \Big)\nonumber 
	\end{align}
	Following our assumptions on the admissible strategies in Definition~\ref{Def:adm2} and Proposition~\ref{prop:ExpoExp_riccati_2}, we have that
	$\left(\alpha, \Lambda\right) \in {L^{2,loc}_{\F}([0,T], \R^d)^2}$ and thus, the stochastic exponential $F_t^\alpha$ is a $(\mathbb{F},\mathbb{P})$-local martingale (which follows from the basic properties of the Dool\'{e}ans-Dade exponential). Therefore, there exists a sequence of stopping times $\{\tau_n\}_{n\geq1}$ satisfying $\lim_{n \rightarrow \infty} \tau_n = T$, $\p$-$\as$, such that that $F^\alpha_{t \wedge \tau_n}$ is a positive martingale for every $n$.
	
	\noindent Furthermore, $- \frac{\Gamma_0}{\gamma} \exp\big(- \gamma e^{ \int^T_0 r(u)du} x_0 \big) e^{\int^t_0 D_s(\alpha_s) ds}$ is non-increasing. Therefore, $J^\alpha_{t\wedge \tau_n}$ is a supermartingale. Then for $s \leq t$, $\E[ J^\alpha_{t\wedge \tau_n} | \cF_s] \leq J^\alpha_{s\wedge \tau_n}$. It implies that for any set $ A \in \cF_s$, 
	\begin{equation*}
		\E[ J^\alpha_{t\wedge \tau_n} \id_A] \leq \E[J^\alpha_{s\wedge \tau_n}\id_A], \quad s \leq t.
	\end{equation*}
	for every $n$.  Since $\left(\exp\big(- \gamma e^{ \int^T_{t\wedge \tau_n} r(u) du} X_{t \wedge \tau_n} \big)\right)_{n\geq0}$ is uniformly integrable (Definition~\ref{Def:adm2}~$(c)$) and $\Gamma$ is bounded (see the last claim of Proposition~\ref{prop:ExpoExp_riccati_2}), $\left(J^\alpha_{t\wedge \tau_n} \right)_n$ and $\left(J^\alpha_{s\wedge \tau_n} \right)_n$ are uniformly integrable. Let $n \rightarrow \infty$, then $\E[ J^\alpha_t \id_A] \leq \E[J^\alpha_s \id_A]$. Consequently, we deduce that $J^\alpha$ is a supermartingale  for every arbitrary admissible strategy $\alpha$.
	It remains to show that $J_t^{\alpha^*}$ is a true martingale for the optimal strategy $\alpha^*$ in which case $e^{\int_0^t D_s(\alpha_s^*)ds}=1$ and therefore the above equation~\eqref{eq:JExpoU}, when  inserting the candidate~\eqref{Eq:alpha_GeneralExpo*} for the optimal strategy \(\alpha^*\) reads $J_t^{\alpha^*}=-\frac{\exp(Y_0)}{\gamma}\exp\left(-\gamma  e^{\int_{0}^{T} r(s)ds}x_0\right) F_t^{\alpha^*}$. Now, for $\alpha_t = \alpha^*_t$, $F_t^{\alpha^*}$ is a \(\P\)-martingale with expectation \(1 \) by Lemma~\ref{lm: extended_m_AJ_lemma} with $g_{1}(t)=-1$ and $g_{2}(t)=1$ for every \(t \in [0,T]\).
	Subsequently, $J^{\alpha^*}_t$ is a true martingale and taking the expectation, its boils down that:
	\begin{align*}
		\E_{x_0,V_0}{\big[-\frac{1}{\gamma}\exp\big(-\gamma X^{\alpha^*}_T\big) \big]}&:=\E_{x_0,V_0}{\big[J_T^{\alpha^*}\big]}=-\frac{\Gamma_0}{\gamma}\exp\Big(-\gamma  e^{\int_{0}^{T} r(s)ds}x_0\Big) \E_{x_0,V_0}{\big[F_T^{\alpha^*} \big]}\\
		&= -\frac{1}{\gamma}\exp\Big(-\gamma  e^{\int_{0}^{T} r(s)ds}x_0\Big)\Gamma_0.
	\end{align*}
	where the last inequality comes from the fact that \(F_T^{\alpha^*}\) is a true martingale. This completes the  proof of the equality~\eqref{eq:Optvalue_Exp}.
	We have verified all conditions required by martingale optimality principle in Definition~\ref{def:Martopt}, except for the admissibility of $\alpha^*$, which follows from Step~2 below.
	
	
	\smallskip
	\noindent {\sc Step~2} (\textit{Proof of Admissibility:})
	We now prove the admissibility of the candidate in equation~\eqref{eq:optcandExpo} for the optimal portfolio strategy $\alpha^*$, and, on the way,
	we deduce the bound in equation~\eqref{eq:boundExpo}.
	
	\smallskip
	\noindent In order to show that $\alpha^*$ is admissible, we have to show that (a), (b), (c) of Definition~\ref{Def:adm2} hold. Part (b) is true because~\eqref{eq:wealth} has a unique solution~\eqref{eq:wealthProcess} in terms of $(S,V,B)$.   
	To prove (c), that is $\left\{\exp\big[ - \gamma e^{ \int^T_\tau r_v dv} X^{\alpha^*}_\tau \big]: \tau \text{ stopping time with values in } [0, T] \right\}$ is a uniformly integrable family, we only need to show
	\begin{equation}
		\sup_{\tau}\E\Big[ e^{- p\gamma e^{ \int^T_\tau r(u) du} X^{\alpha^*}_\tau} \Big] < \infty, \text{ for some } p > 1.
	\end{equation}
	Note that from~\eqref{eq:wealthProcess}, \(X^{\alpha^*}_\tau = e^{\int^\tau_0 r(s)ds}x_0 + \int^\tau_0 e^{ \int^\tau_s r(u) du} \alpha_s^{*\T} \lambda_s \mathrm{d}s  + \int^\tau_0 e^{\int^\tau_s r(u) du} \alpha_s^{*\T} dB_s,\)
	then: 
	\begin{align*}
		\E\Big[ e^{- p\gamma e^{ \int^T_\tau r(u) du} X^{\alpha^*}_\tau} \Big] &\leq e^{-p\gamma e^{ \int^T_0 r(u) du}x_0} \E \Big[ \exp \Big(- p \gamma \int^\tau_0 e^{ \int^T_s r(u) du} \alpha_s^{*\T} \lambda_s ds - p \gamma \int^\tau_0 e^{\int^T_s r(u) du} \alpha_s^{*\T} dB_s \Big) \Big] \\
		& \leq e^{-p\gamma e^{ \int^T_0 r(u) du}x_0} \E \Big[ \exp \Big(p \gamma \int^T_0 e^{ \int^T_s r(u) du} \big|\alpha_s^{*\T} \lambda_s\big| ds + p^2\gamma^2 \int^T_0 e^{\int^T_s 2r(u) du} \left|\alpha_s^*\right|^2 ds\Big) \\
		&\hspace{1.5cm} \times \exp \Big(- p^2\gamma^2 \int^\tau_0 e^{\int^T_s 2r(u) du} \left|\alpha_s^*\right|^2 ds - p \gamma \int^\tau_0 e^{\int^T_s r(u) du} \alpha_s^{*\T} dB_s \Big) \Big] 
	\end{align*}
	so that by  H\"older's inequality, followed by the elementary inequality $ab\leq (a^2+b^2)/2 $ in the third line, we may write
	\vspace{-.2cm}
	\begin{align*}
		&\,\sup_{\tau}\E\Big[ e^{- p\gamma e^{ \int^T_\tau r(u) du} X^{\alpha^*}_\tau} \Big]\leq  e^{-p\gamma e^{ \int^T_0 r(u) du}x_0} \E \Big[ \exp\Big(\int^T_0 (2 p^2 \gamma^2 e^{ \int^T_s 2 r(u) du} \left|\alpha_s^*\right|^2 + 2 p \gamma e^{ \int^T_sr(u) du} \big|\alpha_s^{*\T} \lambda_s\big|) ds \Big) \Big]^{\frac12} \\
		& \hspace{4.5cm} \times \sup_{\tau} \E\Big[ \exp \Big(- 2 p^2 \gamma^2 \int^\tau_0 e^{\int^T_s 2 r(u) du} \left|\alpha_s^*\right|^2 ds  - 2 p \gamma \int^\tau_0 e^{\int^T_s r(u) du} \alpha_s^{*\T} dB_s \Big) \Big]^{\frac12}\\
			& \; \leq  \frac{e^{-p\gamma e^{ \int^T_0 r(u) du}x_0}}{\sqrt{2}} \Big( \E \big[ \exp\big(4 p^2 \gamma^2 \int^T_0 e^{ \int^T_s 2 r(u) du} \left|\alpha_s^*\right|^2 ds \big) \big] +  \E \big[ \exp\big(4 p \gamma\int^T_0 e^{ \int^T_sr(u) du} \big|\alpha_s^{*\T} \lambda_s\big| ds \big) \big]   \Big)^{\frac12} \times 1 
	\end{align*}
	which is finite since on the first hand, the Dol\'eans-Dade exponential $\mathcal E\big( -\int_0^{\cdot} 2 p \gamma e^{\int^T_s r(u) du} \alpha_s^{*\T} dB_s \big)$ satisfies the Novikov's condition by virtue of  ~\eqref{eq:assumption_novikov}, and is therefore a true martingale (see also Lemma~\ref{lm: extended_m_AJ_lemma}) with Expectation \(1\). Then by optional sampling theorem with the fact that $\tau \leq T$, it follows that the supremum is bounded by \(1\).
	 
	\noindent On the others hand, we used Jensen's inequality to obtain the bound 
	\begin{equation*}
		e^{2 \int^T_s r(u) du}\left|\alpha_s^{*}\right|^2 = \frac1{\gamma^2} |\lambda_s + \Sigma\Lambda_s|^2 \leq \frac2{\gamma^2}(|\lambda_s|^2 + |\Sigma\Lambda_s|^2) \leq \frac2{\gamma^2} (1 + |\Sigma|^2)(|\lambda_s|^2 +  |\Lambda_s|^2)
	\end{equation*}
	so that	the first term in the bracket is finite thanks to condition~\eqref{eq:assumption_novikov} with constant \(a(p) =2(8p^2 {- 2p}) \left( 1  + |{\Sigma}|^2  \right)\) after noticing that \(8p^2\leq 2(8p^2 {- 2p})\) for all \(p>1\).
	\vspace{-.2cm}
	\begin{equation}
		\E \Big[ \exp\big(4 p^2 \gamma^2 \int^T_0 e^{ \int^T_s 2 r(u) du} \left|\alpha_s^*\right|^2 ds \big) \Big] \leq  \E \left[ \exp\left({a(p) \int_0^T \left( |\lambda_s|^2 + |\Lambda_s|^2\right)  ds}\right)    \right]< \infty,
	\end{equation}
	Still  owing to the elementary inequality $|ab|\leq (|a|^2+|b|^2)/2 $, we have
	\vspace{-.2cm}
	\begin{equation*}
		e^{\int^T_s r(u) du}|\alpha_s^{*\T} \lambda_s| = \frac1{\gamma} \big(|\lambda_s|^2 + |(\Sigma\Lambda_s)^\T\lambda_s|\big) \leq \frac1{\gamma}(|\lambda_s|^2 + |\Sigma| \frac{|\lambda_s|^2 + |\Lambda_s|^2}{2}) \leq \frac1{2\gamma} (2 + |\Sigma|)(|\lambda_s|^2 +  |\Lambda_s|^2)
	\end{equation*}
	which ensures that the second term in the bracket is finite i.e.
	\vspace{-.2cm}
	\begin{equation}
		\E \Big[ \exp\big(4 p \gamma\int^T_0 e^{ \int^T_sr(u) du} \big|\alpha_s^{*\T} \lambda_s\big| ds \big) \Big] \leq  \E \left[ \exp\left({a(2p) \int_0^T \left( |\lambda_s|^2 + |\Lambda_s|^2\right)  ds}\right)    \right]< \infty,
	\end{equation}
	 thanks to condition~\eqref{eq:assumption_novikov} with constant \(a(2p) =2p \left( 2  + |{\Sigma}|  \right)\).
	We are left to prove that $\alpha^* \in L^2_{\F}([0,T], \R^d)$. We have: 
	\vspace{-.2cm}
	\begin{align*}
		\E \left[ \int_0^{T} |\c_s^*|^2 ds \right] & { =} \E \left[ \int_0^{T} \frac1{\gamma^2}  e^{-2\int_{s}^{T} r(s)ds}|\lambda_s +  \Sigma\Lambda_s|^2 ds \right] \leq   \E\left[ \frac2{\gamma^2} (1 + |\Sigma|^2) \int_0^T \left(|\lambda_s|^2 +  |\Lambda_s|^2\right) ds \right] < \infty,
	\end{align*}
	where the last term is finite due to condition  \eqref{eq:assumption_novikov} and the inequality $|z|^q\leq c_{q}e^{|z|},\; \forall \, q\geq1$.
	It becomes straightforward that $\alpha^*$ in equation~\eqref{eq:optcandExpo} is admissible. This complete the Proof \hfill$\Box$
	 
	\medskip
	\noindent {\bf Proof of Theorem~\ref{Thm:LogUtilityGeneral}:}
	From the dynamic of the controlled wealth process given in Equation~\eqref{eq:wealthPowerSol} and using the utility function \(U(x):=\log(x)\), we have
	\vspace{-.2cm}
	\begin{equation}\label{eq:LogUtil}
		U(X_T^\alpha):=\log(X_T^\alpha)
		=
		\log(x_0) +
		\int_0^T \big(r(s) + \alpha_s^\top \lambda_s - \tfrac12 \left|\alpha_s\right|^2\big)\,ds
		+
		\int_0^T \alpha_s^\top \, dB_{s}.
	\end{equation}
	We consider the process \(J_t^\alpha:=
	\log(x_0) +
	\int_0^t \big(r(s) + \alpha_s^\top \lambda_s - \tfrac12 \left|\alpha_s\right|^2\big)\,ds
	+
	\int_0^t \alpha_s^\top \, dB_{s} + \Gamma_t\) where the pair \((\Gamma,\Lambda)\) satisfies a rather simple backward stochastic differential equation (BSDE) under \(\P\) with a driver \(f: [0,T] \times \R \times \R^d \to \R\) of the form:
	\vspace{-.2cm}
	\begin{equation}\label{eq:GammaDeflog}
		\left\{
		\begin{array}{ccl}
			d\Gamma_t &=&  \; -f(t,\Gamma_t ,\Lambda_t)dt + \Lambda_t^\top dW_t, \\
			\Gamma_T &=&  0. 
		\end{array}  
		\right.
	\end{equation}
	In these terms we are bound to choose a function \(f\) for which \(J_t^\alpha\) is a
	supermartingale for all \(\alpha\in \mathcal A\) and there exists a \(\alpha^*\in \mathcal A\) such that \(J_t^{\alpha^*}\) is a martingale.
Note that $J_t^{\alpha}$ satisfies conditions~$1$ and~$2$ of Definition~\ref{def:Martopt} since $\Gamma_T=0$ and $\Gamma_0$ is a constant independent of $ \alpha \in \cA$. Applying Itô's formula on $J_t^{\alpha}$ yields
	\[dJ_t^{\alpha}=\big(r(t) + \alpha_t^\top \lambda_t - \tfrac12 \left|\alpha_t\right|^2-f(t,\Gamma_t ,\Lambda_t)\big)\,dt + \alpha_t^\top \, dB_{t} + \Lambda_t^\top dW_t\].
	In order for the third condition to be fulfilled by $J_t^{\alpha}$, we may have that 
	\[r(t) + \alpha_t^\top \lambda_t - \tfrac12 \left|\alpha_t\right|^2\leq f(t,\Gamma_t ,\Lambda_t),\quad \text{for all}\; t\in [0,T]\;\text{and}\;\alpha\in \mathcal{A}.\] 
	The pointwise maximization yields \(\c^*_t:=\lambda_t\) for every \(t\in [0,T]\).
	We define the driver of the BSDE \eqref{eq:GammaDeflog} by
	 $f(t) := r(t) + \frac{1}{2}|\lambda_t|^2$ for $t \in [0,T]$, which is $\mathcal{F}_t$-measurable.
	Classical results on BSDEs (see for example \cite{Zhang2017}) yield the existence and 
	uniqueness of a solution $ (Y, \Lambda) \in \mathbb{S}^{2}_{\mathbb{F}}([0,T],\mathbb{R}) 
	\times L^2_{\mathbb{F}}([0,T], \mathbb{R}^d)$
	to \eqref{eq:GammaDeflog} so that \(dM_t:= \lambda_t^\top \, dB_{t} + \Lambda_t^\top dW_t\) is a true martingale by Lemma~\ref{lm: extended_m_AJ_lemma} ( since Assumption~\ref{assm:gen} is in force). The admissibility of $\alpha^*$ is straightforward thanks to the same assumptions (see condition~\eqref{eq:assumption_novikov}). The corresponding value function is determined simply by the initial value of the supermartingale \(J_t^\alpha\) at \(\c^*\).
	Subtituting \(\c^*\) into~\eqref{eq:LogUtil} and taking the expectation yields:
	\vspace{-.3cm}
	\begin{align*}
		\sup_{\alpha(\cdot) \in \mathcal A}\E\Big[\log(X_T^{\alpha})\Big]
		=
		\log(x_0) +
		\E\Big[\int_0^T \big(r(s) + \tfrac12 \left|\lambda_s\right|^2\big)\,ds\Big]= \log(x_0) +
		\int_0^T r(s) \,ds + \tfrac12 \int_0^T\sum_{i=1}^{d} \theta_i^2 \E\Big[V^i_s\Big]\,ds
	\end{align*}
	This completes the proof, calling upon Proposition~\ref{prop:timeDen_}, and we are done. \hfill \( \Box \)
	
	\bigskip
	\noindent {\bf Acknowledgement:}  I thank Gilles Pag\`es, Mathieu Rosenbaum and Dro Sigui for insightful discussions. 
	\vspace{-.9cm}
		
	\bibliographystyle{plainnat}
	\bibliography{Bibliography}
	
	\appendix
	\section{Supplementary materials and Proofs.}
	\subsection{Proofs of Proposition~\ref{prop:ExpoPower_riccati_1} 
		and Theorem~\ref{Thm:ExpoPower_riccati_1}}\label{subsect:proofMresultPower2}
	\noindent {\bf Proof of Proposition~\ref{prop:ExpoPower_riccati_1} :} We write for \( \; 0\leq t\leq T\):
	{\small
		\begin{equation}\label{eq:Power_ito_gamma}
			\Gamma_t^{\frac1\delta } =\E^{\tilde{\mathbb{P}}} \Big[ \exp\Big(\int_t^T\frac{\gamma}{\delta} \big(r(s) + \frac{\left| \lambda_s \right|^2}{2(1-\gamma)} \big) ds\Big) \mid \mathcal F_t\Big]= \E^{\tilde{\mathbb{P}}} \Big[ \exp\Big(\int_t^T\frac{\gamma}{\delta} \big(r(s) + \frac{1}{2(1-\gamma)}\sum_{i=1}^d \theta_i^2 V^i_s \big) ds\Big) \mid \mathcal F_t\Big]
		\end{equation}
	}
	\noindent which ensures that $\Gamma_t>0$ $\P-a.s.$, since $V_t$ is non-negative ($V\in \mathbb R^d_+$), $r(t) > 0$ is deterministic, and $ 1-\gamma \leq \delta \leq 1$. We then have in view of~\eqref{eq:Power_ito_gamma} that there exists some positive constant \(m>0\) such that $\Gamma_t \geq m > 0$ for every \(t\in [0,T]\). An application of the exponential-affine transform formula~\cite[Theorem A.4.]{Gnabeyeu2026a} with \(\mathcal{M} \ni m(\dd s) := \frac{\gamma}{2\delta(1-\gamma)}\theta\odot\theta\,{\rm Leb}_d(\dd s) \) (where $\odot$ denote the Hadamard (pointwise or component-wise) product) yields: 
	\begin{equation*}
		\E^{\tilde{\mathbb{P}}} \Big[ \exp\Big(\int_t^T \frac{\gamma}{2\delta(1-\gamma)}\sum_{i=1}^d \theta_i^2 V^i_s ds\Big) \mid \mathcal F_t\Big] = \exp\Big( \sum_{i=1}^d\int_t^T  \big(\frac{\gamma \theta_i^2}{2\delta(1-\gamma)}  + F_i(s,\psi(T-s))\big) \tilde{g}^i_t(s) ds \Big)
	\end{equation*}
	where $\tilde{g}=(\tilde{g}^1,\ldots,\tilde{g}^d)^\T$, given as in~\eqref{eq:Condprocessg_} denotes the adjusted conditional $\tilde{\mathbb{P}}$-expected variance and $\psi\in C([0,T],(\R^d)^*)$ solves the inhomogeneous Ricatti-Volterra equation  	~\eqref{eq:RiccatiPowerpsi1}-~\eqref{eq:RiccatiPowerpsi2}.
	Consequently, ~\eqref{eq:Power_ito_gamma} becomes
	{\small
		\begin{equation*}
			\E^{\tilde{\mathbb{P}}} \Big[ \exp\big(\int_t^T\frac{\gamma}{\delta} \big(r(s) + \frac{\left| \lambda_s \right|^2}{2(1-\gamma)} \big) ds\big) \mid \mathcal F_t\Big]=  \exp\big(\frac{\gamma}{\delta}\int_t^Tr(s) ds +  \sum_{i=1}^d\int_t^T  \big(\frac{\gamma \theta_i^2}{2\delta(1-\gamma)}  + F_i(s,\psi(T-s))\big) \tilde{g}^i_t(s) ds\big) 
		\end{equation*}
	}
	\noindent This yields ~\eqref{eq:GammaPower}. Now, we set $G_t =  \gamma\int_t^T r(s) ds +  \sum_{i=1}^d\int_t^T  \big(\frac{\gamma \theta_i^2}{2(1-\gamma)}  + \delta F_i(s,\psi(T-s))\big) \tilde{g}^i_t(s) ds, \quad t \leq T.$
	Then, $\Gamma = \exp(G)$ and \(	d\Gamma_t = \Gamma_t \Big( d G_t + \frac 1 2 d\langle G \rangle_t \Big).\)
	The dynamics of \( G \) can readily be obtained by recalling \( \tilde{g}_t(s) \) from ~\eqref{eq:processg} and by observing that for fixed \( s \), the dynamics of \( t \to \tilde{g}_t(s) \) are given by \(	d\tilde{g}_t(s) = K(s-t) \, d\tilde{Z}_t \quad t \leq s.\)
	Since \( \tilde{g}_t(t) = V_t \), it follows by stochastic Fubini's theorem, see \citet[Theorem 2.2]{Veraar2012}, that the dynamics of $G$ reads as  
	\begin{align*}
		dG_t &= \;  \Big(-\gamma r(t)  -   \sum_{i=1}^d \big(\frac{\gamma \theta_i^2}{2(1-\gamma)}  + \delta F_i(s,\psi(T-s))\big) V^i_t \Big)dt \\
		& \quad \quad + \delta\sum_{i=1}^d \int_t^T  {K_i}(s-t) \big(\frac{\gamma \theta_i^2}{2\delta(1-\gamma)}  + F_i(s,\psi(T-s))\big) ds d\tilde{Z}^i_t\\
		& = \; \Big(-\gamma r(t)  -   \sum_{i=1}^d \Big(\frac{\gamma \theta_i^2}{2(1-\gamma)}  + \delta \frac {\nu_i^2} 2  (\varsigma^i(t)\psi^i(T-t))^2 \Big) V^i_t \Big)dt + \;  \delta \sum_{i=1}^d \psi^i(T-t)\nu_i \varsigma^i(t) \sqrt{V^i_t}d\widetilde{W}^i_t,
	\end{align*}
	where we changed variables and  used the inhomogeneous Riccati--Volterra equation \eqref{eq:RiccatiPowerpsi2} for $\psi$ for the last equality. This yields that 
	the dynamics of $\Gamma$ is  given by 
	\begin{align*}
		d\Gamma_t &= \;  \Gamma_t \Big(- \Big(\gamma r(t)  +\sum_{i=1}^d \frac{\gamma \theta_i^2}{2(1-\gamma)} V^i_t\Big) -\frac{\gamma}{1-\gamma} \delta^2 \rho^2 \sum_{i=1}^d\frac {\nu_i^2} 2  (\varsigma^i(t)\psi^i(T-t))^2  V^i_t\Big)dt \\ 
		& \quad \quad  + \;  \delta \Gamma_t \sum_{i=1}^d \psi^i(T-t)\nu_i \varsigma^i(t) \sqrt{V^i_t}d\widetilde{W}^i_t \\
		& =  \;  \Gamma_t \Big[ \big( - \gamma r(t)  - \frac{\gamma}{2(1-\gamma)} \left| \lambda_t\right|^2 -\frac{\gamma}{2(1-\gamma)} \delta^2\left| \Sigma \Lambda_t \right|^2\big)dt +\delta \Lambda_t^\top d\widetilde{W}_t \Big],  \label{eq:gamma_Powerheston}
	\end{align*}
	where we used for the last identity the fact that  
	\begin{equation*}
		\left| \Sigma \Lambda_t \right|^2  = \;  \sum_{i=1}^d \left( \rho_i \nu_i \varsigma^i(t) \psi^i(T-t )\right)^2 V^i_t=\rho^2\sum_{i=1}^d \left( \nu_i \varsigma^i(t) \psi^i(T-t )\right)^2 V^i_t, \quad \text{and} \quad \left| \lambda_t\right|^2  = \; \sum_{i=1}^d \theta_i^2 V^i_t
	\end{equation*}
	Arguing as in the proof of Proposition~\ref{prop:ExpoPower_riccati_2}, we show that \(	\E^{\tilde{\P}}\Big[ \sup_{t \in [0,T]} |\Gamma_t|^p \Big] < \infty\)
	for some \(p > 1\).
	\noindent As for $\Lambda$, it is clear that it belongs to $L^2_{\F}([0,T], \R^d)$ for $\varsigma$ and $\psi$ are bounded and $\E \Big[\int_0^T  \sum_{i=1}^d V^i_s ds \Big] <  \infty$ thanks to \eqref{eq:moments V1}. \hfill $\Box$
	
	\medskip
	\noindent {\bf Proof of Theorem~\ref{Thm:ExpoPower_riccati_1}:}
	We show that $J_t^{\alpha}$ fulfills the martingale optimality principle in Definition~\ref{def:Martopt}. 
	For the first condition, note that $\Gamma_T=1$ and hence $J_T^{\alpha}=\frac{1}{\gamma}(X_T^{\alpha})^{\gamma}$. Since $\Gamma_0$ is a constant independent of $ \alpha \in \cA$, $J^\alpha_0 = \frac{x^\gamma_0}{\gamma} \Gamma_0$ is a constant independent of $ \alpha \in \cA$ and consequently, the second condition is also satisfied.
	In order to show that the third condition is fulfilled, we apply Itô's formula on $J_t^{\alpha}$. We have: 	\begin{align*}
		& dJ_t^{\c} = \; \frac{(X_t^{\alpha})^{\gamma}}{\gamma} \Gamma_t \Big[\big(- \gamma r(t)  - \frac{\gamma}{2(1-\gamma)} (\left| \lambda_t\right|^2 + \delta^2\left| \Sigma \Lambda_t \right|^2)\big)dt +\delta \Lambda_t^\top d\widetilde{W}_t \Big] \\
		&\hspace{.5cm}+ \frac{(X_t^{\alpha})^{\gamma-1}}\gamma\Gamma_t \Big(\gamma X_t^{\c} \big(r(t) +  \alpha_t^\T \lambda_t\big) + \frac{\gamma(\gamma-1)}{2}X_t^{\c}\c_t^\T \c_t  \Big) dt + \Gamma_t (X_t^{\alpha})^{\gamma} \c_t^\T dB_t + \delta \c^\T_t \left( \Sigma \Lambda_t\right) (X_t^{\alpha})^{\gamma} \Gamma_t dt \\
		&\hspace{1cm}= \; J_t^{\c} \Big(\frac{\gamma(\gamma-1)}{2}\c^\T \c + \gamma \alpha_t^\T \lambda_t + \delta \c^\T_t \left( \Sigma \Lambda_t\right) - \frac{\gamma}{2(1-\gamma)} (\left| \lambda_t\right|^2 + \delta^2\left| \Sigma \Lambda_t \right|^2) -\delta \frac{\gamma}{(1-\gamma)} \Lambda_t^\top \Sigma\lambda_t\Big) dt \\
		&\hspace{1cm}+ \gamma J_t^{\c} \c_t^\T dB_t + J_t^{\c} \delta \Lambda_t^\top dW_t = J_t^{\c}  D_t(\alpha_t) dt + \gamma J_t^{\c} \c_t^\T dB_t + J_t^{\c} \delta \Lambda_t^\top dW_t.
	\end{align*}
	where the drift factor takes the form: 
	\begin{align*}
		D_t(\alpha) = & \frac{\gamma(\gamma-1)}{2}\c^\T \c + \gamma \c^\T \big(\lambda_t + \delta \Sigma \Lambda_t \big)  - \frac{\gamma}{2(1-\gamma)} \left|\lambda_t + \delta \Sigma \Lambda_t \right| ^2.
	\end{align*}
	Differentiating \(D_t(\alpha) \) with respect to \(\alpha\) and checking the second order condition, one obtains the maximizer \(\alpha^*_t=\frac{1}{1 - \gamma} \left( \lambda_t + \delta \Sigma \Lambda_t \right)\) for every \(t\in[0,T]\) that is the strategy given by Equation~\eqref{Eq:alpha_power*1}.
	
	\medskip
	\noindent 
	Evaluating the drift factor \(D_t \) at \(\alpha^*_t\) show that \(D_t(\alpha^*_t)\) vanishes to \(0\). 
	Note $D_t(\alpha)$ is a quadratic function on $\alpha$ and $\gamma - 1 < 0$. As $D_t(\alpha^*_t) =0$, then $D_t(\alpha_t) \leq 0$ for any admissible strategy \(\alpha\).
	Moreover,	solving the stochastic differential equation for $J_t^{\alpha}$ yields
	\begin{align}
		& \hspace{1cm} \forall\, t\in [0,T] \quad J_t^{\alpha} = \frac{\Gamma_0 x_0^{\gamma}}{\gamma} e^{\int_0^t D_s(\alpha_s) \, ds} F_t^\alpha \label{eq:sdeJ}\\
		& \quad \text{where}\quad  F_t^\alpha = \mathcal{E} \big( \int_0^t \gamma \, \c_s^\top \, dB_s + \delta \, \Lambda_s^\top \, dW_s \big).
	\end{align}
	Now, since $D_s(\alpha_s)\leq 0$, $e^{\int^t_0 D_s(\alpha_s) ds}$ is a non-increasing function. By our assumptions on the admissible strategies~\ref{Def:adm} and Proposition~\ref{prop:ExpoPower_riccati_1},
	$\left(\alpha, \Lambda\right) \in {L^{2,loc}_{\F}([0,T], \R^d)^2}$ and thus the stochastic exponential $F_t^\alpha$ is a local martingale (which follows from the basic properties of the Dool\'{e}ans-Dade exponential). Therefore, there exists a sequence of stopping times $\{\tau_n\}_{n\geq1}$ satisfying $\lim_{n \rightarrow \infty} \tau_n = T$, $\p$-$\as$, such that 
	\begin{equation*}
		\E[J^\alpha_{t\wedge \tau_n} | \cF_s] \leq J^\alpha_{s\wedge\tau_n}, \quad s \leq t \leq T,
	\end{equation*}
	for every $n$.  Moreover, since $J^\alpha_t$ is bounded from below by \(0\) ( $J_t^{\alpha}\geq 0$), applying Fatou's Lemma for $n\rightarrow \infty$, we deduce that $J^\alpha_t$ is a supermartingale for every arbitrary admissible strategy $\alpha$.
	It remains to show that $J_t^{\alpha^*}$ is a true martingale for the optimal strategy $\alpha^*$ in which case $e^{\int_0^t D_s(\alpha_s^*)ds}=1$ and hence $J_t^{\alpha^*}=\frac{\Gamma_0 x_0^{\gamma}}{\gamma} F_t^{\alpha^*}$.
	
	\medskip
	\noindent 
	For $\alpha_t = \alpha^*_t$, $F_t^{\alpha^*}$ is a martingale by Lemma~\ref{lm: extended_m_AJ_lemma} with 
	$g_{1}(t)= \frac{\gamma}{1-\gamma}$  , $g_{2}(t)=\delta$ for every \(t \in [0,T]\) and \(\kappa=\delta\), 
	Subsequently, $J^{\alpha^*}_t$ is a true martingale. We have verified all conditions required by martingale optimality principle, except for the admissibility of $\alpha^*$, which follows directly from Step~1 in the proof of Theorem~\ref{Thm:powerUtilityGeneral} in the more general setting, as the arguments are similar. The proof is complete. \hfill $\Box$
	\subsection{Proofs of Proposition~\ref{prop:ExpoExp_riccati_1} 
		and Theorem~\ref{Thm:ExpoExp_riccati_1} }\label{subsect:proofMresultExpo2}
	\noindent {\bf Proof of Proposition~\ref{prop:ExpoExp_riccati_1} :} Note that if $1 - \rho^2 = 0$, then $\Gamma_t \leq 1$, $\tilde{\P}$-$\as$. If rather $1 - \rho^2 > 0$, we write for every \( 0\leq t\leq T\):
	\begin{equation}\label{eq:Expo_ito_gamma}
		\Gamma_t^{1-\rho^2}=\E^{\tilde{\mathbb{P}}}\Big[\operatorname{exp}\Big(-\frac{1-\rho^2}{2}\int_t^T \left| \lambda_s \right|^2 ds\Big)|\mathcal{F}_t\Big] = \E^{\tilde{\mathbb{P}}} \Big[ \exp\Big(-\frac{1-\rho^2}{2}\int_t^T\sum_{i=1}^d \theta_i^2 V^i_s  ds \Big)\mid \mathcal F_t\Big],
	\end{equation}
	which ensures that $\Gamma_t\leq 1$, $\tilde{\P}-a.s.$, since $V\in \mathbb R^d_+$. An application of the exponential-affine transform formula in~\cite[Theorem A.4]{Gnabeyeu2026a} with \(\mathcal{M} \ni m(\dd s) := -\frac{1-\rho^2}{2}\theta\odot\theta\,{\rm Leb}_d(\dd s) \) (where $\odot$ denote the Hadamard (pointwise or component-wise) product) yields: 
	\begin{equation}
		\E^{\tilde{\mathbb{P}}} \Big[ \exp\Big(\int_t^T -\frac{1-\rho^2}{2}\sum_{i=1}^d \theta_i^2 V^i_s ds\Big) \mid \mathcal F_t\Big] = \exp\Big( \sum_{i=1}^d\int_t^T  \big(-\frac{1-\rho^2}{2} \theta_i^2  + \tilde{F}_i(s,\tilde{\psi}(T-s))\big) \tilde{g}^i_t(s) ds \Big)
	\end{equation}
	where $\tilde{g}=(\tilde{g}^1,\ldots,\tilde{g}^d)^\T$, given as in~\eqref{eq:Condprocessg_} denotes the adjusted conditional $\tilde{\mathbb{P}}$-expected variance and \[\tilde{F}_i(s,\tilde{\psi}) = -\rho\theta_i \nu_i \varsigma^i(s) \tilde{\psi}^i + (D^\top \tilde{\psi})_i + \frac {\nu_i^2} 2 (\varsigma^i(s)\tilde{\psi}^i)^2, \quad i=1,\ldots,d,\] and 
	$\tilde{\psi}\in C([0,T],(\R^d)^*)$ solves the inhomogeneous Ricatti-Volterra equation 
	\begin{align}
		\tilde{\psi}^i(t)&= \int_0^t K_i(t-s)\big(-\frac{ \theta_i^2(1-\rho^2)}{2}  + \tilde{F}_i(T-s,\tilde{\psi}(s))\big) ds, \quad i=1,\ldots,d.
	\end{align} 
	Setting \(\tilde{\psi}= (1-\rho^2)\psi \) implies that \(\tilde{F}_i(s,\tilde{\psi}) = (1-\rho^2) F_i(s,\psi) \quad i=1,\ldots,d\), where \(F_i\) is given in~\eqref{eq:RiccatiExpopsi2}.
	Therefore, it holds that for all \(t\in[0,T]\),
	\begin{equation*}
		\E^{\tilde{\mathbb{P}}} \Big[ \exp\Big(\int_t^T -\frac{1-\rho^2}{2}\sum_{i=1}^d \theta_i^2 V^i_s ds\Big) \mid \mathcal F_t\Big] = \exp\Big( (1-\rho^2) \sum_{i=1}^d\int_t^T  \big(-\frac{\theta_i^2}{2}   + F_i(s,\psi(T-s))\big) \tilde{g}^i_t(s) ds \Big)
	\end{equation*}
	Consequently, ~\eqref{eq:GammaExpo} holds and  $\psi\in L^2([0,T],(\R^d)^*)$ solves the inhomogeneous Ricatti-Volterra equation ~\eqref{eq:RiccatiExpopsi1}-~\eqref{eq:RiccatiExpopsi2}.
	Note that for every \(t\in[0,T]\), Equation~\eqref{eq:GammaExpo} implies straightforwardly that $\Gamma_t>0$, $\P-a.s.$. 
	Now, we set 
	$$G_t = \sum_{i=1}^d\int_t^T  \big(-\frac{\theta_i^2}{2}   + F_i(s,\psi(T-s))\big) \tilde{g}^i_t(s) ds, \quad t \leq T.$$
	Then, $\Gamma = \exp(G)$ and \(d\Gamma_t = \Gamma_t \Big( d G_t + \frac 1 2 d\langle G \rangle_t \Big).\)
	The dynamics of \( G \) can readily be obtained by recalling \( \tilde{g}_t(s) \) from ~\eqref{eq:processg} and by observing that for fixed \( s \), the dynamics of \( t \to \tilde{g}_t(s) \) are given by \(d\tilde{g}_t(s) = K(s-t) \, d\tilde{Z}_t \quad t \leq s.\)
	Since \( \tilde{g}_t(t) = V_t \), it follows by stochastic Fubini's theorem, see \citet[Theorem 2.2]{Veraar2012}, that the dynamics of $G$ reads as  
	\begin{align*}
		dG_t &= \;  \Big( \sum_{i=1}^d \big(\frac{ \theta_i^2}{2}  - F_i(s,\psi(T-s))\big) V^i_t \Big)dt + \sum_{i=1}^d \int_t^T  {K_i}(s-t) \big(-\frac{ \theta_i^2}{2}  + F_i(s,\psi(T-s))\big) ds d\tilde{Z}^i_t\\
		& = \; \Big( \sum_{i=1}^d \Big(\frac{ \theta_i^2}{2}  - \frac {\nu_i^2} 2 (1-\rho^2) (\varsigma^i(t)\psi^i(T-t))^2 \Big) V^i_t \Big)dt + \sum_{i=1}^d \psi^i(T-t)\nu_i \varsigma^i(t) \sqrt{V^i_t}d\widetilde{W}^i_t,
	\end{align*}
	where we changed variables and  used the Riccati--Volterra equation \eqref{eq:RiccatiExpopsi2} for $\psi$ for the last equality. This yields that 
	the dynamics of $\Gamma$ is  given by 
	\begin{align*}
		d\Gamma_t &= \;  \Gamma_t \Big(\sum_{i=1}^d \Big(\frac{ \theta_i^2}{2}  - \frac {\nu_i^2} 2  (1-\rho^2) (\varsigma^i(t)\psi^i(T-t))^2 \Big) V^i_t + \sum_{i=1}^d\frac {\nu_i^2} 2  (\varsigma^i(t)\psi^i(T-t))^2  V^i_t\Big)dt \\ 
		& \quad \quad  + \; \Gamma_t \sum_{i=1}^d \psi^i(T-t)\nu_i \varsigma^i(t) \sqrt{V^i_t}d\widetilde{W}^i_t =  \;  \Gamma_t \Big[ \big( \frac{1}{2} \left| \lambda_t\right|^2 +\frac{1}{2} \left| \Sigma \Lambda_t \right|^2\big)dt + \Lambda_t^\top d\widetilde{W}_t \Big],  
	\end{align*}
	where we used for the last identity the fact that 
		\begin{equation*}
				\left| \Sigma \Lambda_t \right|^2  = \; \rho^2\sum_{i=1}^d \left( \nu_i \varsigma^i(t) \psi^i(T-t )\right)^2 V^i_t, \quad \text{and} \quad \left| \lambda_t\right|^2  = \; \sum_{i=1}^d \theta_i^2 V^i_t
			\end{equation*}
	As  $0<\Gamma_t\leq1$, $\P-a.s.$, we have that \(\Gamma \in \mathbb{S}^{\infty}_{\F}([0,T], \R)\).  As for $\Lambda$, it is clear that it belongs to $L^2_{\F}([0,T], \R^d)$ since $\varsigma$ and $\psi$ are bounded, $\psi$ is continuous thus bounded and $\E \Big[\int_0^T  \sum_{i=1}^d V^i_s ds \Big] <  \infty$ by \eqref{eq:moments V1}. Consequently,
	$\left(\Gamma, \Lambda\right)\in\mathbb{S}^{\infty}_{\F}([0,T], \R) \times L^2_{\F}([0,T], \R^d)$. This complete the Proof \hfill$\Box$
	
	\medskip
	\noindent {\bf Proof of  Theorem~\ref{Thm:ExpoExp_riccati_1} :}
	We show that $J_t^{\alpha}$ fulfills the martingale optimality principle. 
	For the first condition, note that $\Gamma_T=1$ and hence $J_T^{\alpha}=-\frac{1}{\gamma}\exp\left(-\gamma  X^\alpha_T\right)$. Since $\Gamma_0$ is a constant independent of $ \alpha \in \cA$, $J^\alpha_0 = -\frac{1}{\gamma}\exp\left(-\gamma  e^{\int_{0}^{T} r(s)ds}x_0\right) \Gamma_0$ is a constant independent of $ \alpha \in \cA$ and thus the second condition is also satisfied.
	In order to show that the third condition is also fulfilled, we apply Itô's formula on $J_t^{\alpha}$. 
	Setting \(Y_t := -\frac{1}{\gamma}\exp\left(-\gamma  e^{\int_{t}^{T} r(s)ds}X^\alpha_t\right)\), we write by It\^o's lemma:
	\begin{align*}
		dY_t &= \left(\gamma r(s) e^{\int_{t}^{T} r(s)ds}X^\alpha_t - \gamma e^{\int_{t}^{T} r(s)ds}\big( r(t) X^{\alpha}_t  + \alpha_t^\T \lambda_t \big) + \frac{\gamma^2}{2} \alpha_t^\T \alpha_t e^{\int_{t}^{T} 2r(s)ds}\right) Y_t dt - \gamma e^{\int_{t}^{T} r(s)ds} Y_t \alpha_t^\T dB_t \\
		&=\left( - \gamma e^{\int_{t}^{T} r(s)ds} \alpha_t^\T \lambda_t + \frac{\gamma^2}{2} \alpha_t^\T \alpha_t e^{\int_{t}^{T} 2r(s)ds}\right) Y_t dt - \gamma e^{\int_{t}^{T} r(s)ds} Y_t \alpha_t^\T dB_t
	\end{align*}
	Consequently, by It\^o's product rule, one may write
	\begin{align*}
		dJ_t^{\c} =& \; Y_t \Gamma_t \Big[ \big( \frac{1}{2} \left| \lambda_t\right|^2 +\frac{1}{2} \left| \Sigma \Lambda_t \right|^2\big)dt + \Lambda_t^\top d\widetilde{W}_t \Big] + Y_t\Gamma_t \left( - \gamma e^{\int_{t}^{T} r(s)ds} \alpha_t^\T \lambda_t + \frac{\gamma^2}{2} \alpha_t^\T \alpha_t e^{\int_{t}^{T} 2r(s)ds}\right) dt \\
		&-  Y_t \Gamma_t \gamma e^{\int_{t}^{T} r(s)ds} \alpha_t^\T dB_t -\gamma e^{\int_{t}^{T} r(s)ds} \c^\T_t \left( \Sigma \Lambda_t\right) Y_t \Gamma_t dt \\
		&= \; J_t^{\c} \Big(\frac{\gamma^2}{2} e^{\int_{t}^{T} 2r(s)ds} \c^\T_t \c_t - \gamma e^{\int_{t}^{T} r(s)ds} \alpha_t^\T \lambda_t  -\gamma e^{\int_{t}^{T} r(s)ds} \c^\T_t \left( \Sigma \Lambda_t\right)  + \frac{1}{2} (\left| \lambda_t \right|^2 + \left| \Sigma \Lambda_t \right|^2) + \Lambda_t^\top \Sigma\lambda_t\Big) dt\\
		&- \gamma J_t^{\c} e^{\int_{t}^{T} r(s)ds} \c_t^\T dB_t + J_t^{\c}  \Lambda_t^\top dW_t = J_t^{\c}  D_t(\alpha_t) dt - \gamma J_t^{\c} e^{\int_{t}^{T} r(s)ds} \c_t^\T dB_t + J_t^{\c}  \Lambda_t^\top dW_t.
	\end{align*}
	where the drift factor takes the form: 
	\begin{align*}
		D_t(\alpha) = & \frac{\gamma^2}{2} e^{\int_{t}^{T} 2r(s)ds}  \c^\T \c -\gamma  e^{\int_{t}^{T} r(s)ds}  \c^\T \big(\lambda_t + \Sigma \Lambda_t \big)  + \frac{1}{2} \left|\lambda_t + \Sigma \Lambda_t \right| ^2.
	\end{align*}
	Differentiating \(D_t(\alpha) \) with respect to \(\alpha\) and checking the second order condition, the maximizer \(\alpha^*_t=\frac1\gamma  e^{-\int_{t}^{T} r(s)ds} \left( \lambda_t +  \Sigma \Lambda_t \right)\) for every \(t\in[0,T]\) that is the strategy given by Equation~\eqref{Eq:alpha_Expo*1}.
	Evaluating the drift factor \(D_t \) at \(\alpha^*_t\) show that \(D_t(\alpha^*_t)\) vanishes to \(0\). 
	Note $D_t(\alpha)$ is a quadratic function on $\alpha$ and $\gamma - 1 < 0$. As $D_t(\alpha^*_t) =0$, then $D_t(\alpha_t) \geq 0$ for any admissible strategy \(\alpha\).
	Moreover, solving the stochastic differential equation for $J_t^{\alpha}$ yields
	\begin{align*}
		&\hspace{3cm}\forall \, t\in[0,T],\quad J_t^{\alpha}= -\frac{\Gamma_0}{\gamma}\exp\left(-\gamma  e^{\int_{0}^{T} r(s)ds}x_0\right)e^{\int^t_0 D_s(\alpha_s) ds} F_t^\alpha\\
		&\text{where} \quad F_t^\alpha= \mathcal E\Big(\int_0^t -\gamma  e^{\int_{s}^{T} r(u)du} \c_s^\T dB_s + \Lambda_s^\top dW_s \Big). 
	\end{align*}
	Note $D_t(\alpha^*_t) = 0$, then $D_t(\alpha_t) \geq 0$.
	By our assumptions on the admissible strategies~\ref{Def:adm2} and Proposition~\ref{prop:ExpoExp_riccati_1},
	$\left(\alpha, \Lambda\right) \in {L^{2,loc}_{\F}([0,T], \R^d)^2}$ and thus, the stochastic exponential $F_t^\alpha$ is a local martingale (which follows from the basic properties of the Dool\'{e}ans-Dade exponential). Therefore, there exists a sequence of stopping times $\{\tau_n\}_{n\geq1}$ satisfying $\lim_{n \rightarrow \infty} \tau_n = T$, $\p$-$\as$, such that that $F^\alpha_{t \wedge \tau_n}$ is a positive martingale for every $n$.
	
	\noindent Furthermore, $- \frac{\Gamma_0}{\gamma} \exp\big[ - \gamma e^{ \int^T_0 r(u)du} x_0 \big] e^{\int^t_0 D_s(\alpha_s) ds}$ is non-increasing. Therefore, $J^\alpha_{t\wedge \tau_n}$ is a supermartingale. Then for $s \leq t$, $\E[ J^\alpha_{t\wedge \tau_n} | \cF_s] \leq J^\alpha_{s\wedge \tau_n}$. It implies that for any set $ A \in \cF_s$, 
	\begin{equation*}
		\E[ J^\alpha_{t\wedge \tau_n} \id_A] \leq \E[J^\alpha_{s\wedge \tau_n}\id_A], \quad s \leq t.
	\end{equation*}
	for every $n$.  Since $\left(\exp\big[ - \gamma e^{ \int^T_{t\wedge \tau_n} r(u) du} X_{t \wedge \tau_n} \big]\right)_n$ is uniformly integrable (Definition~\ref{Def:adm2}) and $\Gamma$ is bounded, $\left(J^\alpha_{t\wedge \tau_n} \right)_n$ and $\left(J^\alpha_{s\wedge \tau_n} \right)_n$ are uniformly integrable. Let $n \rightarrow \infty$, then $\E[ J^\alpha_t \id_A] \leq \E[J^\alpha_s \id_A]$. Then we deduce that $J^\alpha$ is a supermartingale  for every arbitrary admissible strategy $\alpha$.\\
	It remains to show that $J_t^{\alpha^*}$ is a true martingale for the optimal strategy $\alpha^*$ in which case $e^{\int_0^t D_s(\alpha_s^*)ds}=1$ and hence $J_t^{\alpha^*}=-\frac{\Gamma_0}{\gamma}\exp\left(-\gamma  e^{\int_{0}^{T} r(s)ds}x_0\right) F_t^{\alpha^*}$.
	
	\medskip
	\noindent 
	For $\alpha_t = \alpha^*_t$, $F_t^{\alpha^*}$ is a martingale by Lemma~\ref{lm: extended_m_AJ_lemma} with $g_{1}(t)=-1$ and $g_{2}(t)=1$ for every \(t \in [0,T]\).
	Subsequently, $J^{\alpha^*}_t$ is a true martingale. We have verified all conditions required by martingale optimality principle in Definition~\ref{def:Martopt}, except for the admissibility of $\alpha^*$, which follows directly from Step~1 in the proof of Theorem~\ref{Thm:ExpoUtilityGeneral} below in the more general setting, as the arguments are similar. The proof is complete. \hfill $\Box$
	
	\subsection{Proof of Lemma~\ref{lemma:Volt-H} }\label{subsect:proofLemmaVolt-H}
	\noindent The first claim in~\eqref{eq:Resolventoperator-solu} follows by the associativity of
	the convolution and applying the fundamental theorem of calculus. For the second claim, one could employ the calculus of convolutions and resolvents. However, we instead use the Laplace transform to deduce that:
	\[L_{\big((f \star r)^\prime \star (K \star g)\big)} (t)= t L_{f \star r}(t)L_{K \star g}(t)= t L_{f}(t)L_{r}(t)L_{K}(t)L_{g}(t) = L_f(t)L_g(t) = L_{f*g}(t).\]
	\noindent where the penultimalte equality come from applying laplace transform to equation ~\ref{eq:Resolvent}. We next conclude by the injectivity of Laplace transform.\hfill $\Box$
\end{document}